\documentclass[a4paper,10pt]{amsart}

\usepackage{a4}
\usepackage{graphics,graphicx}
\usepackage{amssymb,amsmath}
\usepackage[all]{xy}
\usepackage{stmaryrd}
\usepackage{longtable}
\usepackage{multirow}
\usepackage{hyperref}
\usepackage{mathrsfs}
\usepackage{color}
\usepackage{tikz}
\usetikzlibrary{calc}

\newtheorem{theorem}{Theorem}[section]
\newtheorem{lemma}[theorem]{Lemma}
\newtheorem{proposition}[theorem]{Proposition}
\newtheorem{corollary}[theorem]{Corollary}

\theoremstyle{definition}

\newtheorem{example}[theorem]{Example}
\newtheorem{remark}[theorem]{Remark}

\newtheorem{setting}[theorem]{Setting}
\newtheorem{algorithm}[theorem]{Algorithm}

\theoremstyle{remark}

\def\aa{\ensuremath{\mathfrak{a}}}
\def\bb{\ensuremath{\mathfrak{b}}}

\def\TT{\mathbb{T}}
\def\KK{\mathbb{K}}
\def\ZZ{\mathbb{Z}}
\def\PP{\mathbb{P}}
\def\QQ{\mathbb{Q}}

\def\<{\langle}
\def\>{\rangle}

\def\cox{\mathcal{R}}

\makeatletter
\newcommand{\thickhline}{%
    \noalign {\ifnum 0=`}\fi \hrule height 1pt
    \futurelet \reserved@a \@xhline
}
\makeatother

\renewcommand{\phi}{\varphi}

\def\quot{/\!\!/}

\def\rq#1{\widehat{#1}}
\def\t#1{\widetilde{#1}}
\def\b#1{\overline{#1}}
\def\bangle#1{\langle #1 \rangle}

\def\KK{{\mathbb K}}
\def\TT{{\mathbb T}}
\def\ZZ{{\mathbb Z}}

\def\QQ{{\mathbb Q}}
\def\PP{{\mathbb P}}
\def\XX{{\mathbb X}}
\def\FF{{\mathbb F}}

\def\Cox{\cox}

\def\id{{\rm id}}

\def\Mov{{\rm Mov}}

\def\Cl{\operatorname{Cl}}

\def\Pic{\operatorname{Pic}}

\def\Spec{{\rm Spec}}
\def\Proj{{\rm Proj}}

\def\lin{{\rm lin}}

\def\rank{\operatorname{rank}}

\def\Fa{\FF_a}

\def\ZZZ{\ZZ_{\geq 0}}

\def\KT#1{\KK[T_1,\ldots,T_{#1}]}
\def\laurant#1{\KK[T_1^{\pm 1},\ldots,T_{#1}^{\pm 1}]}

 \author[J.~Hausen, S.~Keicher and A.~Laface]{J\"urgen~Hausen, Simon~Keicher and Antonio~Laface}
 \address{Mathematisches Institut, Universit\"at T\"ubingen,
Auf der Morgenstelle 10, 72076 T\"ubingen, Germany}
\email{juergen.hausen@uni-tuebingen.de}

\address{Mathematisches Institut, Universit\"at T\"ubingen,
Auf der Morgenstelle 10, 72076 T\"ubingen, Germany}
\email{keicher@mail.mathematik.uni-tuebingen.de}

\address{Departamento de Matem\'atica, Universidad de Concepci\'on,
Casilla 160-C, Concepci\'on, Chile}
\email{alaface@udec.cl}

\title[Computing Cox rings]{Computing Cox rings}
\subjclass[2000]{
14L24, 14L30, 14C20, 14Q10, 14Q15, 13A30, 52B55\\
The second author was partially supported 
by the DFG Priority Program SPP 1489.
The third author was partially supported 
by Proyecto FONDECYT Regular N. 1110096.
}

\begin{document}

\begin{abstract}
We consider modifications, for example blow ups, 
of Mori dream spaces and provide algorithms 
for investigating the effect on the Cox ring,
for example verifying finite generation or 
computing an explicit presentation in terms of 
generators and relations.
As a first application, we compute the Cox rings 
of all Gorenstein log del Pezzo surfaces of Picard 
number one.
Moreover, we show computationally that all smooth 
rational surfaces of Picard number at most six 
are Mori dream surfaces and we provide 
explicit presentations of the Cox ring for those
not admitting a torus action.
Finally, we provide the Cox rings of projective 
spaces blown up at certain special point 
configurations.
\end{abstract}

\maketitle

\section{Introduction}

Generalizing the well known construction of the 
homogeneous coordinate ring of a toric variety~\cite{cox}, 
one associates to any normal complete variety $X$ 
defined over an algebraically closed field $\KK$ 
of characteristic zero and having a finitely generated 
divisor class group $\Cl(X)$ its {\em Cox ring\/}
$$ 
\mathcal{R}(X)
\ = \ 
\bigoplus_{\Cl(X)} \Gamma(X,\mathcal{O}_X(D)).
$$
A characteristic feature of the Cox ring is its 
divisibility theory: it allows unique factorization 
in the multiplicative monoid of homogeneous 
elements~\cite{Ar,BeHa0,Ha2}.
Projective varieties with finitely generated 
Cox ring are called {\em Mori dream spaces\/}~\cite{HuKe}.
Such a Mori dream space is determined by its 
Cox ring up to a finite choice of  
possible  Mori chambers, which in turn 
correspond to GIT quotients of the action 
of $H = \Spec\, \KK[\Cl(X)]$ 
on the total coordinate space $\Spec\, \mathcal{R}(X)$.
In the surface case, the Cox ring $\mathcal{R}(X)$ 
even completely encodes~$X$. 
Once the Cox ring of a variety $X$ is 
known in terms of generators and relations, 
this opens an approach to the explicit study of $X$.
For example, in~\cite{BreBrDe}, Manin's conjecture 
was proven for the $E_6$-singular surface using 
such a presentation for its Cox ring.
Explicit Cox ring computations are often based on 
a detailed knowledge of the geometry of the 
underlying variety;
the pioneer work in this direction concerns
(generalized) smooth del Pezzo surfaces and is 
due to Batyrev/Popov~\cite{BaPo} and 
Hassett/Tschinkel~\cite{HaTs}.

The aim of the present paper is to provide 
computational methods for Cox rings not 
depending on a detailed geometric understanding 
of the underlying varieties.
We consider modifications $X_2 \to X_1$ of projective 
varieties, where one of the associated Cox rings 
$R_1$ and $R_2$ is explicitly given in terms of generators 
and relations,
for example, $X_1$ might be a projective space and 
$X_2 \to X_1$ a sequence of blow ups.
Whereas $R_1$ can be directly determined from $R_2$, 
see Proposition~\ref{prop:coximage}, the problem of computing 
$R_2$ from $R_1$ is in general delicate; even finite generation 
can get lost.

Our approach uses the technique of toric ambient 
modifications developed in~\cite{BaHaKe,Ha2} 
and upgraded in Section~\ref{sec:ambmod} according 
to our computational needs.
The idea is to realize $X_2 \to X_1$ via a 
modification of toric varieties $Z_2 \to Z_1$, where 
$X_1 \subseteq Z_1$ is embedded in a compatible way,
which means in particular that $R_1$ is the 
factor ring of the Cox ring of $Z_1$ by the 
$\Cl(X_1)$-homogeneous ideal describing $X_1$.
Prospective homogeneous generators $f_j$ for 
$R_2$ correspond either to exceptional divisors 
or can be encoded as prime divisors on $X_1$.
The general basic Algorithms~\ref{algo:stretchcemds} 
and~\ref{algo:modifycemds} check via the algebraic 
criteria~\ref{thm:ambientblow} and~\ref{rem:K1primcrit} 
if a given guess of $f_j$ indeed generates the Cox ring
$R_2$ and, in the affirmative case, compute the defining 
ideal of relations for $R_2$. 
The main computational issues are a saturation
process to compute the $\Cl(X_2)$-homogeneous 
ideal $I_2$ of the proper transform $X_2$ in the 
Cox ring $\mathcal{R}(Z_2)$ and $\Cl(X_2)$-primality tests 
of the $f_j \in \mathcal{R}(Z_2)/I_2$ to verify that $\mathcal{R}(Z_2)/I_2$ 
is the Cox ring $R_2$ of $X_2$ according to~\ref{thm:ambientblow} 
and~\ref{rem:K1primcrit}.

As a first application, we consider in 
Section~\ref{section:picnr1} the
Gorenstein log del Pezzo surfaces~$X$
of Picard number one; see~\cite[Theorem~8.3]{AlNi}
for a classification in terms of the singularity 
type.
The toric ones correspond to the reflexive lattice 
triangles, see for example~\cite{Koe}, and their 
Cox rings are directly obtained by~\cite{cox}.
The Cox rings of the nontoric $X$ allowing 
still a $\KK^*$-action have been determined 
in~\cite{HaSu}.
In Theorem~\ref{thm:gorensteinlogpezzos},
we provide the Cox rings for the remaining cases;
that means for the $X$ admitting no nontrivial 
torus action.
The approach is via a presentation 
$\PP_2 \leftarrow \t{X} \to X$, where $\t{X}$ is smooth.
From~\cite{AGL,Der,HaTs} we infer enough information 
on the generators 
of the Cox ring of $\t{X}$ for computing the Cox ring 
of $X$ by means of our algorithms; 
we note that an explicit computation of the Cox ring
of $\t{X}$ is not needed (and in fact was not always 
feasible on our systems).

The ``lattice ideal method'' presented in
Section~\ref{section:latticeideal}
produces systematically generators for 
the Cox ring of the blow up $X_2$ of a given 
Mori dream space $X_1$ at a subvariety
$C \subseteq X_1$ contained in the smooth 
locus of $X_1$. 
The theoretical basis for this is 
Proposition~\ref{prop:reesalg} where we 
describe the Cox ring $R_2$ of $X_2$ 
as a saturated Rees algebra defined by 
the $\Cl(X_1)$-homogeneous ideal of the 
center $C \subseteq X_1$ in the Cox ring
$R_1$ of $X_1$. 
Building on this, Algorithm~\ref{algo:latticeideal} 
verifies a given set of prospective 
generators for $R_2$ and, in the affirmative
case computes the ideal of relations of $R_2$; 
a major computational advantage compared 
to the more generally applicable 
Algorithm~\ref{algo:modifycemds} is that the 
involved primality checks are now replaced with 
essentially less complex dimension computations.
The $\ZZ$-grading of $R_2$ given by the Rees 
algebra structure allows to produce systematically 
generators of $R_2$ by computing stepwise 
generators for the $\ZZ$-homogeneous components.
This is implemented in Algorithm~\ref{algo:latticeideal2},
which basically requires the Cox ring
$R_1$ of $X_1$ in terms of generators and relations
and $\Cl(X_2)$-homogeneous generators of the 
ideal of the center of $X_2 \to X_1$. 
It then terminates if and only if the blow up 
$X_2$ is a Mori dream space and in this case, 
it provides the Cox ring $R_2$ of $X_2$.
A sample computation is performed in 
Example~\ref{ex:wpp345}

An application of the lattice ideal method
is given in Section~\ref{sec:smoothrat},
where we investigate smooth rational surfaces $X$ 
of Picard number $\varrho(X) \le 6$.
Using our algorithms, we show in 
Theorem~\ref{thm:pic6} that they 
are all Mori dream surfaces and we provide
the Cox rings for those $X$ that do not admit a 
nontrivial torus action; for the toric $X$ 
one obtains the Cox ring directly 
by~\cite{cox} and for the $X$ with a 
$\KK^*$-action, the methods of~\cite{HaSu}
apply.
Certain blow ups of the projective plane 
have also been considered earlier: 
general point configurations
lead to the smooth del Pezzo surfaces
and almost general ones lead 
to so-called weak del Pezzo surfaces, 
see~\cite{BaPo,Der,Der2,HaTs,StiTeVe}. 
Wheras the remaining blow ups of the plane 
can be settled by our methods in a purely
computational way, the blow ups of Hirzebruch 
surfaces require besides the algorithmic also 
a theoretical treatment.

In Section~\ref{section:lineargen}, we 
consider blow ups of point configurations 
in the projective space. 
Algorithm~\ref{algo:lineargen} tests whether 
the Cox ring is generated by proper transforms 
of hyperplanes and, if so, computes the 
Cox ring.
In Example~\ref{ex:almostfanoplane} we treat
the blow up of the projective plane at a symmetric 
configuration of seven points.
Then we leave the surface case and study blow ups 
of the projective space $\PP_3$ at configurations 
of six distinct points.
Recall that in the case of general position, 
Castravet/Tevelev~\cite{CaTe} determined 
generators of the Cox ring and 
Sturmfels/Xu~\cite{StuXu} the relations,
see also~\cite{StVe}.
Moreover, in~\cite{StuXu} special 
configurations are considered and 
a certain subring of the Cox ring 
is described, compare~\cite[p.~456]{StuXu}.
We obtain in Theorem~\ref{thm:PP3} that 
the blow up of six points not contained 
in a hyperplane is always a Mori dream space 
and we list the Cox rings for the 
\textit{edge-special} configurations, 
i.e.~four points are general 
and at least one of the six lies in two 
hyperplanes spanned by the others.

All our algorithms are stated explicitly and 
will be made available within a software package.
In our computations, we made intensive use of the software 
systems Macaulay2~\cite{M2}, Magma~\cite{magma}, 
Maple~\cite{maple} and Singular~\cite{singular}.
We would like to thank the developers for providing
such helpful tools.
Moreover, we are grateful to Cinzia Casagrande 
for her comments and 
discussions about Section~\ref{section:lineargen}.
Finally, we want to express our sincere thanks to the
referees for many 
valuable hints and helpful suggestions for improving 
our manuscript.

\tableofcontents

\section{Toric ambient modifications}
\label{sec:ambmod}

As indicated before, the ground field
$\KK$ is algebraically closed and of 
characteristic zero throughout the article.
In this section, we provide the necessary 
background on Cox rings, Mori dream spaces 
and their modifications.
Let us begin with recalling notation and basics 
from~\cite{coxrings}.
To any normal complete variety~$X$ with 
finitely generated divisor class group $\Cl(X)$, 
one can associate a {\em Cox sheaf\/} and a 
{\em Cox ring\/}
$$ 
\mathcal{R}
\ := \ 
\bigoplus_{\Cl(X)} \mathcal{O}_X(D),
\qquad\qquad
\mathcal{R}(X)
\ := \ 
\bigoplus_{\Cl(X)} \Gamma(X,\mathcal{O}_X(D)).
$$
The Cox ring $\mathcal{R}(X)$ is {\em factorially 
$\Cl(X)$-graded\/} in the sense that it is integral 
and every nonzero homogeneous nonunit is 
a product of $\Cl(X)$-primes.
Here, a nonzero homogeneous nonunit 
$f \in \mathcal{R}(X)$ is called {\em $\Cl(X)$-prime\/}
if for any two homogeneous $g,h \in \mathcal{R}(X)$
we have that $f \mid gh$ implies $f \mid g$ or 
$f \mid h$.
If the divisor class group $\Cl(X)$ is torsion free, 
then the Cox ring $\mathcal{R}(X)$ is even a UFD in 
the usual sense.

If $\mathcal{R}$ is locally of finite type, 
e.g.~$X$ is $\QQ$-factorial or $X$ is a 
{\em Mori dream space\/}, 
i.e.~a normal projective variety with
$\mathcal{R}(X)$ finitely generated~\cite{HuKe}, 
then one has the relative spectrum 
$\rq{X} := \Spec_X \mathcal{R}$.
The {\em characteristic quasitorus\/}  
$H := \Spec \, \KK[\Cl(X)]$ 
acts on $\rq{X}$ and the canonical map 
$p \colon \rq{X} \to X$ is a good quotient 
for this action.
We call $p \colon \rq{X} \to X$ a 
{\em characteristic space\/} over $X$.
If $X$ is a Mori dream space,
then one has the {\em total coordinate space\/}
$\b{X} :=  \Spec \, \mathcal{R}(X)$ 
and a canonical $H$-equivariant open 
embedding $\rq{X} \subseteq \b{X}$.
Note that the characteristic 
space $p \colon \rq{X} \to X$ coincides with 
the universal torsor introduced by Colliot-Th{\'e}l{\`e}ne
and Sansuc~\cite{CT1,CT2} if and only if $X$ is locally 
factorial in the sense that for every closed 
point $x \in X$ the local ring~$\mathcal{O}_{X,x}$ 
is a UFD.

Given a Mori dream space $X$ and a system 
$\mathfrak{F} = (f_1,\ldots,f_r)$ of pairwise 
non-associated $\Cl(X)$-prime generators of 
the Cox ring $\mathcal{R}(X)$, we can construct
an embedding into a toric variety.
First, with $\b{Z} := \KK^r$, we have a
closed embedding
$$ 
\b{X} \ \to \ \b{Z},
\qquad\qquad
\b{x} \ \mapsto \ (f_1(\b{x}), \ldots, f_r(\b{x})).
$$
This embedding is $H_X$-equivariant, where the 
characteristic quasitorus $H_X$ acts diagonally 
on $\b{Z}$ via the weights 
$\deg(f_i) \in \Cl(X) = \XX(H_X)$.
For any ample class $w \in \Cl(X)$ on $X$,
we obtain a set of $w$-semistable points
$$
\rq{Z} 
\ := \ 
\{\b{z} \in \b{Z}; \; 
f(\b{z}) \ne 0 \text{ for some } f \in \KK[T_1,\ldots,T_r]_{nw}, \, n > 0\}.
$$
The intersection $\rq{X} := \b{X} \cap \b{Z}$ is 
the set of $w$-semistable points for the action of 
$H_X$ on $\b{X}$.
Altogether, this gives rise to a commutative diagram
$$ 
\xymatrix{
{\b{X}}
\ar@{}[r]|\subseteq
&
{\KK^r}
\\
{\rq{X}}
\ar@{}[r]|\subseteq
\ar[d]_{p}
\ar[u]
&
{\rq{Z}}
\ar[d]^{p}
\ar[u]
\\
X
\ar@{}[r]|\subseteq
&
Z
}
$$
where $p \colon \rq{Z} \to Z$ is the toric 
characteristic space~\cite{cox} and we have 
an induced embedding $X \subseteq Z$ of quotients.
Then $\Cl(X) = \Cl(Z)$ holds and $X$ inherits
many geometric properties from $Z$, 
see~\cite[Sec.~III.2.5]{coxrings} for details.
We call $X \subseteq Z$ in this situation a 
{\em compatibly embedded Mori dream space (CEMDS)}.

\begin{remark}
\label{rem:semistab}
Consider a CEMDS $X \subseteq Z$.
Let $Q \colon \ZZ^r \to K := \Cl(Z) = \Cl(X)$ denote
the degree map of the Cox ring 
$\mathcal{R}(Z) = \KK[T_1, \ldots, T_r]$
of the ambient projective toric variety $Z$,
sending the $i$-th canonical basis vector $e_i \in \ZZ^r$ 
to the degree of the $i$-th variable~$T_i$ and 
let $P \colon \ZZ^r \to \ZZ^n$ be the Gale dual
map, i.e.~$P$ is dual to the inclusion 
$\ker(Q) \subseteq \ZZ^r$.
If $w \in \Cl(Z)$ is an ample class of $Z$
and hence for $X$,
then the fans $\rq{\Sigma}$ of $\rq{Z}$ and 
$\Sigma$ of $Z$ are given by 
$$ 
\rq{\Sigma}
\ := \ 
\{\rq{\sigma} \preceq \QQ^{r}_{\ge 0}; \; 
w \in Q(\rq{\sigma}^\perp \cap \QQ^r_{\geq 0})\},
\qquad\qquad
\Sigma^{\max}
\ = \ 
\{P(\rq{\sigma}); \; \rq{\sigma} \in \rq{\Sigma}^{\max} \},
$$
where we write $\preceq$ for the face relation 
of cones and regard $Q$ and $P$ as maps of the 
corresponding rational vector spaces.
If $X \subseteq Z$ is a CEMDS, then the ample 
class $w \in \Cl(Z) = \Cl(X)$ is also an ample 
class for $X$. 
Note that a different choice of the ample class 
$w' \in \Cl(X)$ may lead to another CEMDS 
$X \subseteq Z'$ according to 
the fact that the Mori chamber decomposition of $Z$ 
refines the one of $X$.
\end{remark}

We now consider modifications, that means proper 
birational morphisms $\pi \colon X_2 \to X_1$ of 
normal projective varieties.
A first general statement describes the Cox ring 
of $X_{1}$ in terms of the Cox ring  of~$X_{2}$.
By a morphism of graded algebras $A = \oplus_M A_m$ 
and $B = \oplus_N B_n$ we mean an algebra homomorphism
$\psi \colon A \to B$ together with an accompanying 
homomorphism $\t{\psi} \colon M \to N$ of the grading 
groups such that  $\psi(A_m) \subseteq B_{\t{\psi}(m)}$ 
holds for all $m \in M$.

\begin{proposition}
\label{prop:coximage}
Let $\pi \colon X_2 \to X_1$ be a modification
of normal projective varieties and let 
$C \subseteq X_{1}$ be the center of $\pi$.
Set $K_{i} := \Cl(X_{i})$ and $R_{i} := \mathcal{R}(X_{i})$ 
and identify 
$U := X_2 \setminus \pi^{-1}(C)$ with 
$X_1 \setminus C$.
Then we have canonical surjective push forward maps
$$ 
\t{\pi}_* \colon K_2 \to K_1, \ [D] \mapsto [\pi_*D],
\qquad
\pi_* \colon R_2 \to R_1, 
\ 
(R_2)_{[D]} \ni f \mapsto f_{\vert U} \in (R_1)_{[\t{\pi}_*D]}.
$$
Now suppose that $\mathcal{R}(X_2)$ is finitely 
generated,
let $E_1, \ldots, E_l \subseteq X_2$ denote 
the exceptional prime divisors and 
$f_1, \ldots, f_l \in \mathcal{R}(X_2)$ 
the corresponding canonical sections.
Then we have a commutative diagram
$$ 
\xymatrix{
R_2
\ar[rr]^{\pi_{*}}
\ar@{->>}[dr]_{\lambda}
&&
R_1
\\
&
R_2 / \bangle{f_i-1; \; 1 \le i \le l}
\ar[ur]_{\psi}^{\cong}
&
}
$$
of morphisms of graded algebras, where 
$\lambda$ is the canonical projection 
with the projection 
$\t{\lambda} \colon K_2 \to K_2 / \bangle{\deg(f_1), \ldots, \deg(f_l)}$
as accompanying homomorphism
and the induced morphism $\psi$ is an isomorphism.
\end{proposition}

\begin{lemma}
\label{lem:K1grad}
Let $R$ be a $K_2$-graded domain, 
$f \in R_{w}$ with $w$ of infinite order 
in $K_{2}$ and consider the downgrading 
of $R$ given by $ K_{2} \to K_{1} := K_2/ \bangle{w}$.
Then $f-1$ is $K_1$-prime.
\end{lemma}

\begin{proof}
Let $(f-1)g=ab$, where $g,a,b \in R$ are
$K_1$-homogeneous elements. 
Since $w$ has infinite order, any $K_1$-homogeneous 
element $u \in R$ can be uniquely written as a sum 
$u=u_{0} + \ldots + u_{n}$, 
where each $u_i$ is $K_2$-homogeneous,
both $u_0$ and $u_n$ are non-zero
and $\deg_{K_2}(u_i)=\deg_{K_2}(u_0) + iw$ 
holds for each $i$ such that $u_i$ is non-zero.
According to this observation we write
\begin{eqnarray*}
 (f-1)\sum_{i=0}^ng_i
 & = & 
 \left(\sum_{i=0}^{n_1}a_i\right)
 \left(\sum_{i=0}^{n_2}b_i\right).
\end{eqnarray*}
We have $-g_0=a_0b_0$ since otherwise, 
by equating the $K_2$-homogeneous 
elements of the same degree we would 
obtain either $g_0=0$ or $a_0b_0=0$.
Similarly, we see $fg_n=a_{n_1}b_{n_2}$.
Thus
$$
fg_n  
\ = \ 
a_{n_1}b_{n_2},
\quad
fg_{n-1} -g_{n} 
\ = \ 
a_{n_1}b_{n_2-1}+a_{n_1-1}b_{n_2},
\quad
\ldots
\quad
-g_0  
\ = \ 
a_0b_0.
$$
By an induction argument,
eliminating the $g_{i}$ gives
\begin{eqnarray*} 
 (a_0f^{n_1}+a_1f^{n_1-1}+\dots+a_{n_1}) 
 (b_0f^{n_2}+b_1f^{n_2-1}+\dots+b_{n_2})
&  = &
 0.
\end{eqnarray*}
Since $R$ is integral, one of the
two factors must be zero, say the first one. 
Then $f-1$ divides
$$
a
\ =  \
a_0+\dots+a_{n_1}
\ = \
a_0(1-f^{n_1})+\dots+a_{n_1-1}(1-f).
$$
\end{proof}

\begin{proof}[Proof of Proposition~\ref{prop:coximage}]
Let $x_i \in X_i$ be smooth points with $\pi(x_2) = x_1$ 
such that $x_2$ is not contained in any of the exceptional
divisors. Consider the divisorial sheaf $\mathcal{S}_i$ 
on $X_i$ associated to the subgroup of divisors avoiding 
the point $x_i$, see~\cite[Constr.~4.2.3]{coxrings}.
The open subset $U = X_2 \setminus \pi^{-1}(C) \subseteq X_2$ 
is mapped by $\pi$ isomorphically onto $X_1 \setminus C$.
This leads to 
canonical morphisms of graded algebras
$$ 
\Gamma(X_2, \mathcal{S}_2) 
\ \to \ 
\Gamma(U_2, \mathcal{S}_2) 
\ \to \ 
\Gamma(X_1 \setminus C, \mathcal{S}_1)
\ \to \ 
\Gamma(X_1, \mathcal{S}_1),
$$
where the accompanying homomorphisms of the grading 
groups are the respective push forwards of Weil divisors;
here we use that $C$ is of codimension at least two in~$X_1$
and thus any section of $\mathcal{S}_1$ over $X_1 \setminus C$ 
extends uniquely to a section  of $\mathcal{S}_1$ over~$X_1$.
The homomorphisms are compatible with the relations
of the Cox sheaves $\mathcal{R}_i$, see
again~\cite[Constr.~4.2.3]{coxrings}, and thus induce 
canonical morphisms of graded rings
$$ 
\Gamma(X_2, \mathcal{R}_2) 
\ \to \ 
\Gamma(U_2, \mathcal{R}_2) 
\ \to \ 
\Gamma(X_1 \setminus C, \mathcal{R}_1)
\ \to \ 
\Gamma(X_1, \mathcal{R}_1).
$$
This establishes the surjection  
$\pi_{*} \colon R_2 \to R_1$ with the  
canonical push forward $\t{\pi}_{*}$ of divisor 
class groups as accompanying homomorphism.
Clearly, the canonical sections $f_i$ of the exceptional
divisors are sent to $1 \in R_1$. 

We show that the induced map $\psi$ is an isomorphism.
As we may proceed by induction on $l$, it suffices to 
treat the case $l=1$.
Lemma~\ref{lem:K1grad} tells us that $f_1-1$ is 
$K_1$-prime.
From~\cite[Prop.~3.2]{Ha2} we infer that $\bangle{f_1-1}$ 
is a radical ideal in $R_2$.
Since $\Spec(\psi)$ is a closed embedding of 
varieties of the same dimension and equivariant 
with respect to the action of the 
quasitorus $\Spec \, \KK[K_1]$,
the assertion follows.
\end{proof}

We are grateful to the referee suggesting to us 
the following example as a geometric illustration.

\begin{example}
\label{ex:referee}
Let $X_2$ be the blow-up of 
$\mathbb{P}^2$ at
the three toric fixed points and $p=[1,1,1]$
and let $\pi \colon X_2\to X_1$ be the contraction
of the exceptional curve $E$ over~$p$. The Cox ring 
$R_2$ of $X_2$ is the coordinate ring of the affine 
cone $\b{X}_2$ 
over $G(2,5)$, that means that $R_2$ is $\KK[T_1,\ldots,T_{10}]$ 
modulo the ideal  $I_2$ generated 
by the Pl\"ucker relations
$$
 T_7T_8 - T_6T_9 + T_5T_{10}, \quad
 T_4T_6 - T_3T_7 - T_1T_{10}, \quad
 T_4T_8 - T_3T_9 + T_2T_{10}, 
$$
$$
 T_4T_5 - T_2T_7 - T_1T_9, \qquad
 T_3T_5 - T_2T_6 - T_1T_8.
$$
The $\Cl(X_2)$-grading is the finest one  
leaving variables and relations homogeneous.
We assume $E$ to be $V(T_{10})$ in Cox coordinates. 
According to Proposition~\ref{prop:coximage},
we have an epimorphism 
$R_2 \to R_2/\bangle{T_{10}-1} \cong R_1$ 
onto the Cox ring $R_1$ of $X_1$.
This defines a closed embedding of
$\b{X}_1 = \mathbb{K}^6$ as $\b{X}_2 \cap V(T_{10}-1)$ 
into  $\b{X}_2 \setminus V(T_{10}) \subseteq \KK^{10}$;
this embedding is explicitly given by
$$
(z_1,\dots,z_6)
\ \mapsto \ 
(
 z_2z_3-z_1z_4,\,
 z_1z_6 - z_2z_5,\,
 z_1,\,
 z_2,\,
 z_3z_6 - z_4z_5,\,
 z_3,\,
 z_4,\,
 z_5,\,
 z_6, \,
1).
$$
Observe that $\overline X_2 \setminus V(T_{10})$
is the subset of the affine cone over $G(2,5)$
corresponding to a Schubert cell consisting 
of all lines of $\PP_4$ not meeting a 
certain plane in $\PP_4$.
\end{example}

As an immediate consequence of 
Proposition~\ref{prop:coximage}, we obtain that 
$X_1$ is a Mori dream space provided $X_2$ is one; 
recall that in~\cite{Ok} it is more generally 
proven that for any dominant morphism $X_1 \to X_2$ 
of $\QQ$-factorial projective varieties, $X_2$ 
is a Mori dream space if $X_1$ is.
The converse question is in general delicate.
For a classical counterexample, consider 
points $x_1, \ldots, x_9 \in \PP_2$ that lie 
on precisely one smooth cubic $\Gamma \subseteq \PP_2$ 
and admit a line $\ell \subseteq \PP_2$ such that 
$3 \ell \vert_{\Gamma} -p_1 -\ldots -p_9$ is not a torsion element
of $\Pic^0(\Gamma)$. Then the blow up $X_1$ of 
$\PP_2$ at $x_1, \ldots, x_8$ is a Mori 
dream surface and the blow up $X_2$ of $X_{1}$ at $x_9$ 
is not, see for example~\cite[Prop.~4.3.4.5]{coxrings}.

We now upgrade the technique of toric ambient 
modifications developed in~\cite{Ha2,BaHaKe}
according to our computational purposes. 
In the following setting, $\rq{X}_i \to X_i$ 
is not necessarily a characteristic space and 
$\b{X}_i$ not necessarily a total coordinate space.

\begin{setting}
\label{set:ambmod}
Let $\pi \colon Z_2 \to Z_1$ be a toric modification,
i.e.~$Z_1$, $Z_2$ are complete toric varieties and 
$\pi$ is a proper birational toric morphism.
Moreover, let $X_i \subseteq Z_i$ be closed 
subvarieties, both intersecting the big $n$-torus 
$\TT^n \subseteq Z_i$, such that $\pi(X_2) = X_1$ 
holds. Then we have a commutative diagram
$$
\xymatrix{
{\KK^{r_2}}
\ar@{}[r]|=
&
{\b{Z}_2}
\ar@{}[r]|\supseteq
&
{\b{X}_2}
&
{\b{X}_1}
\ar@{}[r]|\subseteq
&
{\b{Z}_1}
\ar@{}[r]|=
&
{\KK^{r_1}}
\\
&
{\rq{Z}_2}
\ar@{}[r]|\supseteq
\ar@{}[u]|{\rotatebox[origin=c]{90}{$\scriptstyle \subseteq$}}
\ar[d]_{p_2}
&
{\rq{X}_2}
\ar@{}[u]|{\rotatebox[origin=c]{90}{$\scriptstyle \subseteq$}}
\ar[d]_{p_2}
&
{\rq{X}_1}
\ar@{}[u]|{\rotatebox[origin=c]{90}{$\scriptstyle \subseteq$}}
\ar[d]^{p_1}
\ar@{}[r]|\subseteq
&
{\rq{Z}_1}
\ar@{}[u]|{\rotatebox[origin=c]{90}{$\scriptstyle \subseteq$}}
\ar[d]^{p_1}
&
\\
&
Z_2
\ar@{}[r]|\supseteq
\ar@/_2pc/[rrr]_{\pi}
&
X_2
\ar[r]
&
X_1
\ar@{}[r]|\subseteq
&
Z_1
&
\\
&&&&&
}
$$
where the downwards maps $p_i \colon \rq{Z}_i \to Z_i$ 
are toric characteristic spaces and
$\rq{X}_i \subseteq \rq{Z}_i$ 
are the closures of the inverse image 
$p_i^{-1}(X_i \cap \TT^n)$. 
Let $I_i \subseteq \KK[T_1,\ldots,T_{r_i}]$
be the vanishing ideal of the closure 
$\b{X}_i \subseteq \KK^{r_i}$
of $\rq{X}_i \subseteq \rq{Z}_i$
and set $R_i := \KK[T_1,\ldots,T_{r_i}] / I_i$.
Note that $R_i$ is graded by 
$K_i := \Cl(Z_i)$.
\end{setting}

\begin{theorem}
\label{thm:ambientblow}
Consider the Setting~\ref{set:ambmod}.
\begin{enumerate}
\item
If $X_1 \subseteq Z_1$ is a CEMDS, the ring $R_2$ 
is normal and 
$T_{1}, \ldots, T_{r_2}$ define pairwise 
non-associated $K_2$-primes in~$R_2$,
then $X_2 \subseteq Z_2$ is a CEMDS.
In particular, $K_2$ is the divisor class group 
of $X_2$ and $R_2$ is the Cox ring of~$X_2$. 
\item
If $X_2 \subseteq Z_2$ is a CEMDS,
then $X_1 \subseteq Z_1$ is a CEMDS.
In particular, $K_1$ is the divisor class group 
of $X_1$ and $R_1$ is the Cox ring of~$X_1$. 
\end{enumerate}
\end{theorem}

\begin{proof}
First consider the lattice homomorphisms $P_{i} \colon \ZZ^{r_{i}} \to \ZZ^{n}$
associated to the toric morphisms $p_{i} \colon \rq{Z}_{i} \to Z_{i}$.
Viewing the $P_{i}$ as matrices, we may assume that $P_{2} = [P_{1},B]$
with a matrix $B$ of size $n \times (r_{2}-r_{1})$. 
We have a commutative diagram of lattice homomorphisms and 
the corresponding diagram of homomorphisms of tori:
\begin{align}
\label{eq:ambientblowup}
 \xymatrix{
&
{\ZZ^{r_2}}
\ar[dl]_{\genfrac{}{}{0pt}{}{e_i \mapsto e_i}{e_j \mapsto m_je_j}}
\ar[dr]^{[E_{r_1},A]}
&
\\
{\ZZ^{r_2}}
\ar[d]_{P_2 = [P_1,B]}
&&
{\ZZ^{r_1}}
\ar[d]^{P_1}
\\
{\ZZ^{n}}
\ar[rr]_{E_n}
&&
{\ZZ^{n}}
}
\qquad\qquad\qquad
\xymatrix{
&
{\TT^{r_2}}
\ar[dl]_{\mu}
\ar[dr]^{\alpha}
&
\\
{\TT^{r_2}}
\ar[d]_{p_2}
&&
{\TT^{r_1}}
\ar[d]^{p_1}
\\
{\TT^{n}}
\ar[rr]_{\id}
&&
{\TT^{n}}
}
\end{align}
where in the left diagram, the $e_i$ are the first $r_1$, 
the $e_j$ the last $r_2-r_1$ canonical basis vectors 
of $\ZZ^{r_2}$, the $m_{j}$ are positive 
integers and $E_n, E_{r_1}$ denote the unit matrices
of size $n,r_1$ respectively and $A$ is an integral
$r_1 \times (r_2-r_1)$ matrix.

We prove~(i). 
We first show that $R_{2}$ is integral.
By construction, it suffices to show that 
$p_{2}^{-1}(X_{1} \cap \TT^{n})$ is irreducible.
By assumption, $\b{X}_{1} \cap \TT^{r_{1}}$ is 
irreducible.
Since $\alpha$ has connected kernel, also 
$\alpha^{-1}(\b{X}_{1} \cap \TT^{r_{1}})$
is irreducible. 
We conclude that 
$\b{X}_{2} \cap \TT^{r_{2}} = 
\mu(\alpha^{-1}(\b{X}_{1} \cap \TT^{r_{1}}))$
is irreducible.
Moreover, since $X_{2}$ is complete and the $K_{2}$-grading
of $R_{2}$ has a pointed weight cone, we obtain that 
$R_{2}$ has only constant units.
Thus, \cite[Thm.~3.2]{BaHaKe} yields that 
$R_{2}$ is factorially $K_{2}$-graded.
Since the $T_{i}$ are pairwise non-associated
$K_{2}$-primes and $R_2$ is normal, we conclude
that $R_2$ is the Cox ring of $X_2$ and 
$X_{2} \subseteq Z_{2}$ is a CEMDS.

We turn to~(ii). Observe that for every
$f \in I_{2}$, the Laurent polynomials 
$\mu^{*}(f)$ and $\alpha^{*}(f(t_{1},\ldots,t_{r_{1}},1,\ldots,1))$
differ by a monomial factor. We conclude
$$ 
\KK[T_{1}^{\pm 1}, \ldots, T_{r_{2}}^{\pm 1}] \cdot I_{2}
\ = \ 
\bangle{\alpha^*(f(t_{1},\ldots,t_{r_{1}},1,\ldots,1)); f \in I_{2}}
\ \subseteq \
\KK[T_{1}^{\pm 1}, \ldots, T_{r_{2}}^{\pm 1}].
$$
Now Proposition~\ref{prop:coximage} tells us that 
$R_{1}$ is the Cox ring of $X_{1}$.
Since $T_{1}, \ldots, T_{r_{1}}$ define pairwise 
different prime divisors in $X_{1}$, we conclude that 
$X_{1} \subseteq Z_{1}$ is a CEMDS.
\end{proof}

The verification of normality as well as the primality 
tests needed for Theorem~\ref{thm:ambientblow}~(ii) 
are computationally involved. The following observation 
considerably reduces the effort in many cases.

\begin{remark}
\label{rem:K1primcrit}
See~\cite[Prop.~3.3]{BaHaKe}.
Consider the Setting~\ref{set:ambmod}
and assume that the canonical map 
$K_2 \to K_1$ admits a section (e.g. $K_1$ is free).
\begin{enumerate}
\item 
If $R_1$ is normal and $T_{r_1+1},\ldots, T_{r_2}$ 
define primes in $R_2$ (e.g. they are $K_2$-prime 
and $K_2$ is free), then $R_2$ is normal.
\item
Let $T_1,\ldots, T_{r_1}$ define $K_1$-primes in $R_1$
and $T_{r_1+1},\ldots, T_{r_2}$ define $K_2$-primes in $R_2$.
If no $T_j$ with $j \ge r_1+1$ divides a $T_i$ with 
$i \le r_1$ in $R_2$, then also $T_1,\ldots, T_{r_1}$ define 
$K_2$-primes in $R_2$.
\end{enumerate}
\end{remark}

As a consequence of Theorem~\ref{thm:ambientblow}, 
we obtain that the modifications preserving finite 
generation are exactly those arising from toric 
modifications as discussed.
More precisely, let $Z_{2} \to Z_{1}$ be a toric 
modification mapping $X_{2} \subseteq Z_{2}$ onto 
$X_{1} \subseteq Z_{1}$. We call $Z_{2} \to Z_{1}$ 
a {\em good\/} toric ambient modification if it 
is as in Theorem~\ref{thm:ambientblow}~(i).

\begin{corollary}
Let $X_{2} \to X_{1}$ be a birational morphism of 
normal projective varieties such that 
the Cox ring $\mathcal{R}(X_{1})$ is finitely 
generated.
Then the following statements are equivalent.
\begin{enumerate}
\item
The Cox ring $\mathcal{R}(X_{2})$ is finitely generated. 
\item
The morphism $X_{2} \to X_{1}$ arises from 
a good toric ambient modification.
\end{enumerate}
\end{corollary}

\begin{proof}
The implication ``(ii)$\Rightarrow$(i)'' is 
Theorem~\ref{thm:ambientblow}.
For the reverse direction, set $K_i := \Cl(X_i)$
and $R_i := \mathcal{R}(X_i)$.
Let $f_1,\ldots,f_{r_2}$ be pairwise nonassociated 
$K_2$-prime generators for $R_2$.
According to Proposition~\ref{prop:coximage},
we may assume, after suitably numbering,  that 
$f_1,\ldots,f_{r_1}$ define generators of $R_1$,
where $r_1 \le r_2$. 
Now take an ample class $w_1 \in K_1$.
Then the pullback $w_2' \in K_2$ of $w_1$
under $X_2 \to X_1$ is semiample on $X_2$.
Choose $w_2 \in K_2$ such that $w_2$ is ample 
on $X_2$ and the toric ambient variety $Z_2$ 
of $X_2$ defined by $w_2$ has an ample cone 
containing $w_2'$ in its closure.
Then, with the sets of semistable points 
$\rq{Z}_2,  \rq{Z}_2' \subseteq \KK^{r_2}$ 
defined by $w_2$, $w_2'$ respectively
and $\rq{Z}_1 \subseteq \KK^{r_1}$ the one
defined by $w_1$. Then we obtain morphisms
$$ 
Z_2 
\ = \ 
{\rq{Z}_2 \quot H_2}
\ \to \ 
{\rq{Z}_2' \quot H_2}
\ \cong \ 
{\rq{Z}_1 \quot H_2}
\ = \ 
Z_1,
$$
where $H_i := \Spec \, \KK[K_i]$ denotes the 
characteristic quasitorus of $Z_i$; observe 
that $\rq{Z}_2' \to \rq{Z}_2' \quot H_2$ is 
in general not a toric characteristic space.
Thus, we arrive at Setting~\ref{set:ambmod} 
and $Z_2 \to Z_1$ is the desired good toric 
ambient modification inducing the morphism
$X_2 \to X_1$.
\end{proof}

For a flexible use of Theorem~\ref{thm:ambientblow} we will 
have to adjust given embeddings of a Mori dream 
space, e.g.~bring general 
points of a CEMDS into a more special position,
or remove linear relations from a redundant 
presentation of the Cox ring.
The formal framework is the following.

\begin{setting}
\label{set:stretchcompress}
Let $Z_1$ be a projective toric variety with 
toric total coordinate space 
$\b{Z}_1 = \KK^{r_1}$,
toric characteristic space 
$p_1 \colon \rq{Z}_1 \to Z_1$
and ample class $w \in K_1 := \Cl(Z_1)$. 
Consider $K_1$-homogeneous polynomials 
$h_1, \ldots, h_l \in \KK[T_1,\ldots,T_{r_1}]$ 
and, with $r_1' := r_1+l$, the 
(in general non-toric) embedding 
$$ 
\bar{\imath} \colon \KK^{r_1} \ \to \ \KK^{r_1'},
\qquad\qquad
(z_1, \ldots, z_{r_1}) 
\ \mapsto \ 
(z_1, \ldots, z_{r_1}, h_1(z), \ldots, h_l(z)).
$$
Note that $\KK[T_1,\ldots,T_{r_1'}]$ is graded by 
$K_1' := K_1$ via attaching to $T_1,\ldots, T_{r_1}$ 
their former $K_1$-degrees and to $T_{r_1+i}$ 
the degree of $h_i$.
The class $w \in K_1'$ defines a toric variety
$Z_1'$ with toric total coordinate space 
$\b{Z}_1' = \KK^{r_1'}$ toric characteristic space 
$p_1 \colon \rq{Z}_1' \to Z_1'$.
Any closed subvariety $X_1 \subseteq Z_1$ 
and its image $X_1' := \imath(X_1)$ lead to 
a commutative diagram
$$
\xymatrix{
{\KK^{r_1}}
\ar@{}[r]|=
&
{\b{Z}_1}
\ar@/^2pc/[rrrrr]^{\b{\imath}}
\ar@{}[r]|\supseteq
&
{\rq{Z}_1}
\ar@{}[r]|\supseteq
\ar[d]_{p_1}
&
{\rq{X}_1}
\ar[d]_{p_1}
&
{\rq{X}_1'}
\ar[d]^{p_1'}
\ar@{}[r]|\subseteq
&
{\rq{Z}_1'}
\ar@{}[r]|\subseteq
\ar[d]^{p_1'}
&
{\b{Z}'_1}
\ar@{}[r]|=
&
{\KK^{r_1'}}
\\
&
&
Z_1
\ar@{}[r]|\supseteq
\ar@/_2pc/[rrr]_{\imath}
&
X_1
\ar[r]
&
X_1'
\ar@{}[r]|\subseteq
&
Z_1'
&
&
\\
&&&&&&&
}
$$
where 
$\rq{X}_1 \subseteq \rq{Z}_1$ 
and 
$\rq{X}_1' \subseteq \rq{Z}_1'$ 
are the closures of the inverse image 
$p_1^{-1}(X_1 \cap \TT^n)$ 
and $(p_1')^{-1}(X_1' \cap \TT^{n'})$ . 
Denote by $I_1$ and $I_1'$ 
the respective vanishing ideals
of the closures
$\b{X}_1 \subseteq \KK^{r_1}$
of $\rq{X}_1 \subseteq \rq{Z}_1$
and
$\b{X}_1' \subseteq \KK^{r_1'}$
of $\rq{X}_1' \subseteq \rq{Z}_1'$.
Set $R_1 := \KK[T_1,\ldots,T_{r_1}] / I_1$
and  $R_1' := \KK[T_1,\ldots,T_{r_1'}] / I_1'$.
\end{setting}

\begin{proposition}
\label{prop:stretchcompress}
Consider the Setting~\ref{set:stretchcompress}.
\begin{enumerate}
\item
If $X_1 \subseteq Z_1$ is a CEMDS and 
$T_1,\ldots, T_{r_1}, h_{1}, \ldots, h_{l}$ 
define pairwise non-associated 
$K_1$-primes in~$R_1$ then $X_1' \subseteq Z_1'$ 
is a CEMDS.
\item
If $R_1'$ is normal
the localization $(R_1')_{T_{1} \cdots T_{r_1}}$ 
is factorially $K_1'$-graded and $T_1,\ldots,T_{r_1}$ 
define pairwise non-associated $K_1$-primes in $R_1$
such that $K_1$ is generated by any $r_1-1$ of their 
degrees, then $X_1 \subseteq Z_1$ is a CEMDS.
\item
If $X_1' \subseteq Z_1'$ is a CEMDS, 
then $X_1 \subseteq Z_1$ is a CEMDS.
\end{enumerate}
\end{proposition}

\begin{proof}
First, observe that the ideal $I_1'$ equals 
$I_1 + \bangle{T_{r_1+1}-h_1, \ldots, T_{r_1'}-h_l}$.
Consequently, we have a canonical graded isomorphism 
$R_1' \to R_1$ sending $T_{r_1+i}$ to $h_i$.
Assertion~(i) follows directly.

We prove~(ii).
Since $(R_1')_{T_{1} \cdots T_{r_1}}$ is factorially 
$K_1'$-graded,
we obtain that $(R_1)_{T_{1} \cdots T_{r_1}}$ is 
factorially $K_1$-graded.
Since $T_1,\ldots,T_{r_1}$ define $K_1$-primes
in $R_1$, we can apply~\cite[Thm.~1.2]{Be} to 
see that $R_1$ is factorially $K_1$-graded.
Since $T_1,\ldots,T_{r_1}$ are pairwise 
non-associated we conclude that 
$X_1 \subseteq Z_1$ is a CEMDS.

We turn to~(iii). According to~(ii), we only 
have to show that any $r_1-1$ of the degrees 
of $T_1,\ldots,T_{r_1}$ generate $K_1$.
For this, it suffices to show that each
$\deg(T_j)$ for $j = r_1+1, \ldots, r_1+l$ is 
a linear combination of any $r_1-1$ of the 
first $r_1$ degrees. 
Since $T_1,\ldots,T_{r_1}$ generate
$R_1$ and $T_j$ is not a mutiple of any 
$T_i$, we see that for any $i = 1, \ldots, r_1$,
there is a monomial in $h_j$ not depending on 
$T_i$. The assertion follows.
\end{proof}

\section{Basic algorithms}
\label{sec:algos}

Here we provide the general algorithmic framework.
In order to encode a compatibly embedded
Mori dream space $X_i \subseteq Z_i$ and 
its Cox ring $R_i$, we use the triple 
$(P_i,\Sigma_i,G_i)$, where $P_i$ and $\Sigma_i$ 
are as in Remark~\ref{rem:semistab} 
and $G_i = (g_1,\ldots,g_s)$ 
is a system of generators of the defining ideal 
$I_i$ of the Cox ring $R_i$.
We call such a triple $(P_i,\Sigma_i,G_i)$ as well a
CEMDS.

Given a CEMDS $(P_i,\Sigma_i,G_i)$, the total
coordinate space $\b{X}_i$ is the common zero 
set of the functions in $G_i$ and the degree 
map $Q_i \colon \ZZ^{r_i} \to K_i$  and $P_i$ 
are Gale dual to each other in the sense that 
$Q_i$ is surjective and $P_i$ is the dual of the 
inclusion $\ker(Q_i) \subseteq \ZZ^{r_i}$.
Moreover, $p_i \colon \rq{X}_i \to X_i$ is the 
restriction of the toric morphism defined by~$P_i$.
The following two algorithms implement
Proposition~\ref{prop:stretchcompress}.

\begin{algorithm}[StretchCEMDS]
\label{algo:stretchcemds}
{\em Input: } a CEMDS $(P_1,\Sigma_1,G_1)$, a list 
$(f_1,\ldots,f_l)$ of polynomials $f_i \in \KK[T_1,\ldots,T_{r_1}]$ 
defining pairwise non-associated $K_1$-primes in~$R_1$.
\begin{itemize}
\item
Compute the Gale dual $Q_1 \colon \ZZ^{r_1} \to K_1$ 
of $P_1$.
\item 
Let $Q_1' \colon \ZZ^{r_1+l} \to K_1$ be the 
extension of $Q_1$ by the degrees of $f_1,\ldots,f_l$.
\item
Compute the Gale dual $P_1' \colon \ZZ^{r_1+l} \to \ZZ^{n'}$ 
of $Q_1'$ and the fan $\Sigma'$ in $\ZZ^{n'}$ defined by $P_1'$
and the ample class $w \in K_1' = K_1$ of $Z_1$.
\item
Set $G'_1 := (g_1,\ldots,g_s,T_{r_1+1}-f_1,\ldots,T_{r_1+l}-f_l)$,
where $G_1 = (g_1,\ldots,g_s)$.
\end{itemize}
{\em Output: } the CEMDS $(P_1',\Sigma_1',G_1')$.
\end{algorithm}

The input of the second algorithm is more generally an 
{\em embedded space\/} $X_1 \subseteq Z_1$ that means 
just a closed normal subvariety intersecting the 
big torus.
In particular, we do not care for the moment if $R_1$ 
is the Cox ring of $X_1$. 
We encode $X_1 \subseteq Z_1$ as well by a triple 
$(P_1,\Sigma_1,G_1)$ and name it for short an ES. 
For notational reasons we write $(P_1',\Sigma_1',G_1')$ 
for the input.

\begin{algorithm}[CompressCEMDS]
\label{algo:compresscemds}
{\em Input: } an ES $(P_1',\Sigma_1',G_1')$
such that 
$R_1'$ is normal,
the localization 
$(R_1')_{T_{1} \cdots T_{r_1}}$ 
is factorially $K_1'$-graded and the last $l$ relations 
in $G_1'$ are fake, i.e.~of the form 
$f_i = T_i - h_i$ with $h_i$ 
not depending on $T_i$.
{\em Option: } \texttt{verify}.
\begin{itemize}
\item
Successively substitute $T_i = h_i$ in $G_1'$.
Set $G_1 := (f_1,\ldots, f_{r_1})$,
where $G_1' = (f_1,\ldots, f_{r_1'})$ and $r_1 := r_1' - l$.
\item
Set $K_1 := K_1'$ and let $Q_1 \colon \ZZ^{r_1} \to K_1$ 
be the map sending $e_i$ to $\deg(T_i)$ for 
$1 \le i \le r_1$.
\item
Compute a Gale dual $P_1 \colon \ZZ^{r_1} \to \ZZ^n$
of $Q_1$ and the fan $\Sigma_1$ in $\ZZ^n$ defined 
by $P_1$ and the ample class $w \in K_1 = K_1'$ of $Z_1'$.
\item 
If \texttt{verify} 
was asked then
\begin{itemize}
\item
check if any $r_1-1$ of the degrees of 
$T_1, \ldots, T_{r_1}$ generate $K_1$,
\item
check if $\dim(I_1) - \dim(I_1 + \bangle{T_i,T_j})\geq 2$ 
for all $i \ne j$,
\item
check if $T_1, \ldots, T_{r_1}$ 
define 
$K_1$-primes in $R_1$.
\end{itemize}
\end{itemize}
{\em Output: } the ES $(P_1,\Sigma_1,G_1)$. If 
$(P_1',\Sigma_1',G_1')$ is a CEMDS or all verifications
were positive, then $(P_1,\Sigma_1,G_1)$ is a CEMDS.
In particular, then $R_1$ is the Cox ring of 
the corresponding subvariety $X_1 \subseteq Z_1$.
\end{algorithm}

We turn to the algorithmic version of 
Theorem~\ref{thm:ambientblow}.
We will work with the {\em saturation\/} 
of an ideal $\aa \subseteq \KK[T_1,\ldots,T_r]$
with respect to an ideal $\bb \subseteq \KK[T_1,\ldots,T_r]$;
recall that this is the ideal 
$$
\aa : \bb^\infty 
\ := \ 
\{g \in \KK[T_1,\ldots,T_r]; \; 
g\,\bb^k \subseteq \aa\text{ for some } k \in \ZZ_{\ge 0}\}
\ \subseteq \
\KK[T_1,\ldots,T_r].
$$ 
In case of a principal ideal $\bb = \<f\>$,
we write $\aa:f^\infty$ instead of $\aa:\bb^\infty$.
We say that an ideal $\aa \subseteq \KK[T_1,\ldots,T_r]$ is
$f$-\textit{saturated} if $\aa =\aa : f^\infty$ holds.
We will only consider saturations with respect to 
$f = T_1 \cdots T_r \in \KK[T_1,\ldots,T_r]$;
we refer to~\cite[Chap.~12]{Stu} for the computational aspect.
Let us recall the basic properties, see also~\cite{Jen}.

\begin{lemma}
\label{lem:satprops}
Consider  
$\KK[T,U^{\pm 1}]$ with 
$T = (T_1,\ldots, T_{r_1})$
and
$U = (U_1,\ldots, U_{r_2-r_1})$.
For $f := U_1 \cdots U_{r_2-r_1} \in \KK[U]$,
one has mutually inverse bijections
 \begin{eqnarray*}
\left\{
\begin{array}{c}
\text{ideals in}\ 
\KK[T, U^{\pm 1}]
\end{array}
\right\}
&  \ \longleftrightarrow \ &
\left\{
\begin{array}{c}
\text{$f$-saturated ideals in}\ 
\KK[T, U]
\end{array}
\right\}
\\
\aa
& \mapsto &
\aa\, \cap\, \KK[T, U]
\\
\bangle{\bb}_{\KK[T, U^{\pm 1}]}
&  \mapsfrom &
\bb.
\end{eqnarray*}
Under these maps, the prime ideals of 
$\KK[T, U^{\pm 1}]$
correspond to the $f$-saturated prime 
ideals of $\KK[T, U]$.
\end{lemma}

For transferring polynomials from $\KK[T_1,\ldots,T_{r_1}]$ 
to $\KK[T_1,\ldots,T_{r_2}]$ and vice versa, we 
use the following operations, compare also~\cite{GiMa}.
Consider a homomorphism $\pi \colon \TT^r \to \TT^n$ 
of tori and its kernel $H \subseteq \TT^r$.
\begin{itemize}
\item
By a {\em $\sharp$-pull back\/} of 
$g \in \KK[S_1^{\pm 1},\ldots,S_n^{\pm 1}]$ 
we mean a polynomial $\pi^\sharp g \in \KK[T_1,\ldots,T_r]$
with coprime monomials such that 
$\pi^*g$ and $\pi^\sharp g$ are associated in 
$\KK[T_1^{\pm 1},\ldots,T_r^{\pm 1}]$.
\item
By a {\em $\sharp$-push forward\/} of an $H$-homogeneous 
$h \in\KK[T_1^{\pm 1},\ldots,T_r^{\pm 1}]$ 
we mean a polynomial $\pi_\sharp h \in \KK[S_1,\ldots,S_n]$
with coprime monomials such that 
$h$ and $\pi^*\pi_\sharp h$ are associated in 
$\KK[T_1^{\pm 1},\ldots,T_r^{\pm 1}]$.
\end{itemize}
Note that $\sharp$-pull backs and $\sharp$-push forwards 
always exist and are unique up to constants.
The $\sharp$-pull back $\pi^\sharp g$ of a Laurent polynomial
is its usual pull back $\pi^* g$ scaled with a suitable 
monomial.
To compute the $\sharp$-push forward, factorize the 
describing $m \times n$
matrix of $\pi$ as $P = W \cdot D \cdot V$, where $W,V$ are 
invertible and $D$ is in Smith normal form. 
Then push forward with respect to the homomorphisms 
corresponding to the factors.

\goodbreak

\begin{lemma}
\label{lem:pushpullprops}
Consider 
a monomial dominant morphism 
$\pi\colon \KK^{n_1}\times \TT^{n_2} \to \KK^m$.
Write $T = (T_1,\ldots,T_{n_1})$
and $U = (U_1,\ldots,U_{n_2})$.
\begin{enumerate}
\item
If $\aa \subseteq \KK[T,U^{\pm 1}]$ is a prime ideal,
then $\bangle{\pi_\sharp \aa} \subseteq \KK[S^{\pm 1}]$ is a prime ideal.
\item 
If $\bb \subseteq \KK[S^{\pm 1}]$ is a radical ideal,
then $\bangle{\pi^\sharp \bb} \subseteq \KK[T,U^{\pm 1}]$ 
is a radical ideal.
\end{enumerate}
\end{lemma}

\begin{proof}
The first statement follows from 
$\bangle{\pi_\sharp \aa} = (\pi^*)^{-1}(\aa)$.
To prove~(ii), let $f\in \sqrt{\bangle{\pi^\sharp\bb}}$.
Since $\sqrt{\bangle{\pi^\sharp \bb}} = I(\pi^{-1}(V(\bb)))$ 
is invariant under $H := \ker(\pi|_{\TT^{n_1+n_2}})$, 
we may assume that $f$ is $H$-homogeneous,
i.e.~$f(h \cdot z) = \chi(h)f(z)$ holds
with some character $\chi \in \XX(H)$.
Choose $\eta \in \XX(\TT^{n_1+n_2})$ with $\chi = \eta_{|H}$.
Then $\eta^{-1}f$ is $H$-invariant 
and thus belongs to $\pi^*(I(V(\bb)))$.
Hilbert's Nullstellensatz
and the assumption give $\pi^*(I(V(\bb))) = \pi^*(\bb)$.
We conclude $f \in \bangle{\pi^\sharp \bb}$.
\end{proof}

We are ready for the  algorithm
treating the contraction problem.
We enter a {\em weak CEMDS\/} 
$(P_2,\Sigma_2,G_2)$ in the sense that 
$G_2$ provides generators for the 
extension of $I_2$ to 
$\KK[T_1^{\pm 1},\ldots,T_{r_2}^{\pm 1}]$
and a toric contraction $Z_2 \to Z_1$,
encoded by $(P_1,\Sigma_1)$.

\begin{algorithm}
[ContractCEMDS]
\label{algo:contractcemds}
{\em Input: } a weak CEMDS $(P_2,\Sigma_2,G_2)$ and
a pair $(P_1,\Sigma_1)$, where $P_2 = [P_1,B]$ 
and $\Sigma_1$ is a coarsening of $\Sigma_2$ 
removing the rays through the columns of $B$.
\begin{itemize}
\item
Set $h_i := g_i(T_1, \ldots, T_{r_1},1, \ldots,1) \in \KK[T_1,\ldots,T_{r_1}]$,
where $G_2 = (g_1, \ldots, g_s)$.
\item
Compute a system of generators $G_1'$ for
$I_1' := \bangle{h_1,\ldots,h_s}:(T_1\cdots T_{r_1})^\infty$. 
\item
Set $(P_1',\Sigma_1',G_1') := (P_1,\Sigma_1,G_1')$
and reorder the variables such that the last $l$ 
relations of $G_1'$ are as in
Algorithm~\ref{algo:compresscemds}.
\item
Apply Algorithm~\ref{algo:compresscemds}
to $(P_1',\Sigma_1',G_1')$ and
write $(P_1,\Sigma_1,G_1)$ for the output.
\end{itemize}
{\em Output: } $(P_1,\Sigma_1,G_1)$. 
This is a CEMDS.
In particular, $R_1$ is the Cox ring of 
the image $X_1 \subseteq Z_1$ of $X_2 \subseteq Z_2$
under $Z_2 \to Z_1$.
\end{algorithm}

\begin{proof}
First we claim that in $\KK[T_1^{\pm 1},\ldots,T_{r_1}^{\pm 1}]$,
the ideal generated by $h_1,\ldots,h_s$ coincides 
with the ideal generated by 
$p_1^\sharp(p_2)_\sharp g_1, \ldots, p_1^\sharp(p_2)_\sharp g_r$.
To see this, consider $p_i \colon \TT^{r_i} \to \TT^n$
and let $S_1, \ldots, S_n$ be the variables on $\TT^n$.
Then the claim follows from
$(P_2)_{ij} = (P_1)_{ij}$ for $j \le r_1$ and 
$$ 
p_2^*(S_i) \ = \ T_1^{(P_2)_{i1}} \cdots T_{r_2}^{(P_2)_{ir_2}} ,
\qquad\qquad
p_1^*(S_i) \ = \ T_1^{(P_1)_{i1}} \cdots T_{r_1}^{(P_1)_{ir_1}}.
$$

As a consequence of the claim, we may apply 
Lemma~\ref{lem:pushpullprops} and obtain that 
$G_1'$ defines a radical ideal in 
$\KK[T_1^{\pm},\ldots, T_{r_1}^{\pm}]$.
Moreover, from Theorem~\ref{thm:ambientblow} we 
infer that $\rq{X}_1'$, defined as in 
Setting~\ref{set:ambmod}, is irreducible.
Since $G_1'$ has $\rq{X}_1' \cap \TT^{r_1}$ 
as its zero set, it defines a prime ideal 
in $\KK[T_1^{\pm},\ldots, T_{r_1}^{\pm}]$.
Lemma~\ref{lem:satprops} then shows that 
$I_1' \subseteq \KK[T_1,\ldots, T_{r_1}]$
is a prime ideal.
Using Theorem~\ref{thm:ambientblow} again, 
we see that  $(P_1',\Sigma_1',G_1')$ as defined 
in the third step of the algorithm is a CEMDS.
Thus, we may enter Algorithm~\ref{algo:compresscemds}
and end up with a CEMDS.
\end{proof}

We turn to the modification problem.
Given a Mori dream space $X_1$ with Cox ring $R_1$
and a modification $X_2 \to X_1$, we want to know 
if $X_2$ is a Mori dream space, and if so, we 
ask for the Cox ring $R_2$ of $X_2$.
Our algorithm verifies a guess of prospective
generators for $R_2$ and, if successful, computes the 
relations. In practice, the generators are added 
via Algorithm~\ref{algo:stretchcemds}.

\begin{algorithm}[ModifyCEMDS]
\label{algo:modifycemds}
{\em Input:} a weak CEMDS $(P_1,\Sigma_1,G_1)$, a pair 
$(P_2,\Sigma_2)$ with a matrix $P_2 = [P_1,B]$ 
and a fan $\Sigma_2$ having the columns of $P_2$ 
as its primitive generators and refining $\Sigma_1$.
{\em Options: } \texttt{verify}.
\begin{itemize}
\item
Compute $G_2' := (h_1,\ldots,h_s)$,
where $h_i = p_2^\sharp(p_1)_\sharp(g_i)$
and $G_1 = (g_1,\ldots,g_s)$.
\item 
Compute a system
of generators $G_2$ of
$I_2 := \bangle{h_1,\ldots,h_s}:(T_{r_1+1}\cdots T_{r_2})^\infty$. 
\item
If \texttt{verify} was asked then
\begin{itemize}
\item 
compute a Gale dual $Q_2 \colon \ZZ^{r_2} \to K_2$ 
of $P_2$,
\item
check if $\dim(I_2) - \dim(I_2 + \bangle{T_i,T_j})\geq 2$ 
for all $i\not= j$,
\item
check if $T_1,\ldots, T_{r_2}$ define $K_2$-primes
in $R_2$.
\item 
check if $R_2$ is normal, e.g.~using Remark~\ref{rem:K1primcrit}.
\end{itemize}
\end{itemize}
{\em Output: } $(P_2,\Sigma_2,G_2)$, if the 
\texttt{verify}-checks were all positive, 
this is a CEMDS.
In particular, then $R_2$ is the Cox ring of 
the strict transform $X_2 \subseteq Z_2$ of
$X_1 \subseteq Z_1$ with respect to $Z_2 \to Z_1$.
\end{algorithm}

\begin{proof}
We write shortly
$\KK[T,U^{\pm 1}]$ with the tuples
$T = (T_1,\ldots, T_{r_1})$
and
$U = (U_1,\ldots, U_{r_2-r_1})$
of variables.
Lemma~\ref{lem:pushpullprops} ensures that 
$G_2$ generates a radical ideal in 
$\KK[T, U^\pm]$.
To see that the zero set 
$V(G_2) \subseteq \KK^{r_1} \times \TT^{r_2 - r_1}$ 
is irreducible,
consider the situation of equation~\eqref{eq:ambientblowup} 
in the proof of Theorem~\ref{thm:ambientblow}.
There, in the right hand side diagram, we may lift
the homomorphisms of tori to
\begin{align*}
\xymatrix{
&
{\KK^{r_1}\times\TT^{r_2-r_1}}
\ar[dl]_{\mu}
\ar[dr]^{\alpha}
&
\\
{\KK^{r_1}\times\TT^{r_2-r_1}}
\ar[d]_{p_2}
&&
{\KK^{r_1}}
\ar[d]^{p_1}
\\
{\TT^{n}}
\ar[rr]_{\id}
&&
{\TT^{n}}
}
\end{align*}
Observe that we have an isomorphism
$\phi = \alpha \times \id$
given by
\[
\KK^{r_1}\times \TT^{r_2-r_1} \to \KK^{r_1}\times \TT^{r_2-r_1}
,\qquad
 (z,z')\ \mapsto\ 
 \left(
 z_1 (z')^{A_{1*}},
 \ldots,
 z_{r_1} (z')^{A_{r_1*}},
 z'
 \right)
 .
\]
Since $\b{X}_{1}$ is irreducible,
so is
$ \phi^{-1}(\b X_1 \times \TT^{r_2-r_1}) = \alpha^{-1}(\b{X}_{1})$.
Hence, the image 
$\mu(\alpha^{-1}(\b{X}_{1})) 
= 
\b X_2\cap \KK^{r_1}\times \TT^{r_2-r_1} =V(G_2)$ 
is irreducible as well.
Moreover, Lemma~\ref{lem:satprops} implies that 
$G_2$ generates a prime ideal in 
$\KK[T, U]$.
If the \texttt{verify}-checks were all positive,
then Theorem~\ref{thm:ambientblow} tells us that 
$(P_2,\Sigma_2,G_2)$ is a CEMDS.
\end{proof}

\begin{remark}
Compare Remark~\ref{rem:K1primcrit}.
If the canonical map $K_{2} \to K_{1}$ admits a section,
e.g.~if $K_{1}$ is free, then it suffices to check the 
variables $T_{r_{1}+1}, \ldots, T_{r_{2}}$ for 
$K_{2}$-primeness in $R_{2}$ in the verification step 
of Algorithm~\ref{algo:modifycemds}.
\end{remark}

\begin{remark}[Verify $K$-primality]
 Let $\KK[T_1,\ldots,T_{r}]$ be graded by a finitely generated
 abelian group $K$ and $I$ a $K$-homogeneous ideal.
 By~\cite[Prop. 3.2]{Ha2}, $T_k$ being $K$-prime
 in $\KT{r}/I$ is equivalent to the divisor of 
$T_k$ in $V(I)\subseteq \KK^r$ being $H:=\Spec\, \KK[K]$-prime. 
In computational terms, this means computing the prime
components of $I +\<T_k\>$ and testing
whether $H$ permutes them transitively.
See~\cite{absfact,Ber,GrePfi} for the
computational background.
There are also recent methods from
 numerical algebraic geometry~\cite{SoVeWa}.
\end{remark}

\section{Gorenstein log del Pezzo surfaces}
\label{section:picnr1}
As an application of Algorithm~\ref{algo:contractcemds},
we compute Cox rings of Gorenstein log 
terminal del Pezzo surfaces $X$ of Picard 
number one.
Here, ``del Pezzo'' means that $X$ is a normal 
projective surface with ample anticanonical 
divisor $-K_X$ and the conditions
``Gorenstein'' and ``log terminal'' 
together are equivalent 
to saying that $X$ has at most canonical 
singularities which in turn are 
precisely the rational double points 
(also called ADE or du Val singularities),
see for example~\cite[Thm.~4-6-7]{Mat}.

A classification of Gorenstein log terminal 
del Pezzo surfaces $X$ of Picard number one
according to the singularity type, i.e.~the 
configuration ${\rm S}(X)$ of singularities,
has been given in~\cite[Theorem~8.3]{AlNi}.
Besides $\PP_2$, there are four toric 
surfaces, namely the singularity types $A_1$, 
$A_1A_2$, $2A_1A_3$ and $3A_2$. Moreover, 
there are thirteen (deformation types of) 
$\KK^*$-surfaces; they represent 
the singularity types $A_4$, $D_5$, $E_6$, 
$A_12A_3$, $3A_1D_4$, $A_1D_6$, $A_2A_5$, $E_7$, 
$A_1E_7$, $A_2E_6$, $E_8$, $2D_4$ and
their Cox rings have been determined 
in~\cite[Theorem~5.6]{HaSu}.

We now compute the Cox rings of the remaining 
ones using Algorithm~\ref{algo:stretchcemds}
and the knowledge of generators of their 
resolutions~\cite{Der,AGL}; note that the 
relations for Cox rings of the resolutions
are still unknown in some of the cases.
In the sequel, we will write a Cox ring 
as a quotient $\KT{r}/I$ and specify
generators for the ideal $I$.
The $\Cl(X)$-grading is encoded 
by a {\em degree matrix\/}, i.e.~a
matrix with $\deg(T_1), \ldots, \deg(T_r) \in \Cl(X)$ 
as columns.

\begin{theorem}
\label{thm:gorensteinlogpezzos}
The following table lists the Cox rings of 
the Gorenstein log terminal del Pezzo surfaces 
$X$ of Picard number one that do not allow  
a non-trivial $\KK^*$-action.
\begingroup
\footnotesize
\begin{longtable}{p{.7cm}p{4.8cm}l}
\hline
${\rm S}(X)$  
&  
Cox ring $\mathcal{R}(X)$ 
& 
$\Cl(X)$ and degree matrix 
\\* 
\hline
\\
$2A_4$ 
&
$
\begin{array}{l}
\KK[T_1,\ldots,T_6]/I \text{ with $I$ generated by }
\\[1ex]
\scriptstyle
-T_{2}T_{5}+T_{3}T_{4}+T_{6}^{2},\
-T_{2}T_{4}+T_{3}^{2}+T_{5}T_{6},
\\
\scriptstyle
T_{1}T_{6}-T_{3}T_{5}+T_{4}^{2}, \
T_{1}T_{3}-T_{4}T_{6}+T_{5}^{2}, 
\\
\scriptstyle
T_{1}T_{2}-T_{3}T_{6}+T_{4}T_{5}
\end{array}
$
&
$
\begin{array}{l}
\ZZ\oplus \ZZ/5\ZZ
\\[1ex]
\mbox{\tiny$
\left[
\begin{array}{rrrrrr}
1 & 1 & 1 & 1 & 1 & 1
\\
\bar 2 & \bar 2 & \bar 3 & \bar 4 & \bar 0 & \bar 1 
\end{array}
\right]
$
}
\end{array}
$
\\
\\
\hline
\\
$D_8$ 
& 
$
\begin{array}{l}
\KK[T_1,\ldots,T_4]/I \text{ with $I$ generated by }
\\[1ex]
\scriptstyle
T_{1}^{2}-T_{4}^2T_{2}T_{3}+T_{4}^{4}+T_{3}^{4}
\end{array}
$
& 
$
\begin{array}{l}
\ZZ\oplus \ZZ/2\ZZ
\\[1ex]
\mbox{\tiny$
\left[
\begin{array}{rrrr}
2 & 1 & 1 & 1 \\
\bar 1 & \bar 1 & \bar 1 & \bar 0 
\end{array}
\right]
$
}
\end{array}
$ 
\\
\\
\hline
\\
$D_5A_3$ 
& 
$
\begin{array}{l}
\KK[T_1,\ldots,T_5] / I \text{ with $I$ generated by }
\\[1ex]
\scriptstyle
T_{1}T_{3}-T_{4}^{2}-T_{5}^{2}, \
T_{1}T_{2}-T_{3}^{2}+T_{4}T_{5}
\end{array}
$
&
$
\begin{array}{l}
\ZZ\oplus \ZZ/4\ZZ
\\[1ex]
\mbox{\tiny$
\left[
\begin{array}{rrrrr}
1 & 1 & 1 & 1 & 1\\
\bar 2 & \bar 2 & \bar 0 & \bar 3 & \bar 1\\
\end{array}
\right]
$
}
\end{array}
$
\\
\\
\hline
\\
$D_62A_1$ 
& 
$
\begin{array}{l}
\KK[T_1,\ldots,T_5] / I \text{ with $I$ generated by }
\\[1ex]
\scriptstyle
 T_{5}T_{2}-T_{5}^{2}+T_{3}^{2}+T_{4}^{2},\ 
-T_{2}^{2}+T_{5}T_{2}+T_{1}^{2}-T_{4}^{2}
\end{array}
$
& 
$
\begin{array}{l}
\ZZ \oplus \ZZ/2\ZZ \oplus \ZZ/2\ZZ
\\[1ex]
\mbox{\tiny $
     \left[
    \begin{array}{rrrrr}
    1 & 1 & 1 & 1 & 1 \\
    \bar 1 & \bar 0 & \bar 0 & \bar 1 & \bar 0 \\
    \bar 0 & \bar 1 & \bar 0 & \bar 1 & \bar 1
    \end{array}
    \right]
$
}
\end{array}
$
\\
\\
\hline
\\
$E_6A_2$
& 
$
\begin{array}{l}
\KK[T_1,\ldots,T_4] / I \text{ with $I$ generated by }
\\[1ex]
\scriptstyle
-T_{1}T_{4}^{2}+T_{2}^{3}+T_{2}T_{3}T_{4}+T_{3}^{3}
\end{array}
$
& 
$
\begin{array}{l}
\ZZ \oplus \ZZ/3\ZZ 
\\[1ex]
\mbox{\tiny $
\left[ 
\begin{array}{rrrrrr}
 1 & 1 & 1 & 1\\
\bar 1 & \bar 2 & \bar 0 & \bar 1
\end{array}
 \right]
$
}
\end{array}
$
\\
\\
\hline
\\
$E_7A_1$
& 
$
\begin{array}{l}
\KK[T_1,\ldots,T_4] / I \text{ with $I$ generated by }
\\[1ex]
\scriptstyle
-T_{1}T_{3}^{3}-T_{2}^{2}+T_{2}T_{3}T_{4}+T_{4}^{4}
\end{array}
$
& 
$
\begin{array}{l}
\ZZ \oplus \ZZ/2\ZZ
\\[1ex]
\mbox{\tiny$
\left[ 
\begin{array}{rrrrrr}
1 & 2 & 1 & 1\\
\bar 1 & \bar 1 & \bar 1 & \bar 0
\end{array}
 \right]
$
}
\end{array}
$
\\
\\
\hline
\\
$E_8$
& 
$
\begin{array}{l}
\KK[T_1,\ldots,T_4] / I \text{ with $I$ generated by }
\\[1ex]
\scriptstyle
  T_{1}^{3}+T_{1}^{2}T_{4}^{2}+T_{2}^{2}-T_{3}T_{4}^{5}
\end{array}
$
& 
$
\begin{array}{l}
\ZZ  
\\[1ex]
\mbox{\tiny$
\left[ 
\begin{array}{rrrrrr}
2 & 3 & 1 & 1
\end{array}
 \right]
$
}
\end{array}
$
\\
\\
\hline
\\
$A_7$ 
& 
$
\begin{array}{l}
\KK[T_1,\ldots,T_4] / I \text{ with $I$ generated by }
\\[1ex]
\scriptstyle
T_{1}^{2}-T_{4}T_{2}T_{3}+T_{4}^{4}+T_{3}^{4}
\end{array}
$
& 
$
\begin{array}{l}
\ZZ \oplus \ZZ/2\ZZ
\\[1ex]
\mbox{\tiny$
\left[ 
\begin{array}{rrrr}
2 & 2 & 1 & 1\\
\bar 1 & \bar 1 & \bar 1 & \bar 0
\end{array}
 \right]
$
}
\end{array}
$
\\
\\
\hline
 \\
$A_8$ 
& 
$
\begin{array}{l}
\KK[T_1,\ldots,T_4] / I \text{ with $I$ generated by }
\\[1ex]
\scriptstyle
-T_{1}T_{2}T_{3}+T_{2}^{3}+T_{3}^{3}+T_{4}^{3}
\end{array}
$
& 
$
\begin{array}{l}
\ZZ \oplus \ZZ/3\ZZ
\\[1ex]
\mbox{\tiny$
     \left[
    \begin{array}{rrrr}
    1 & 1 & 1 & 1 \\
    \bar 1 & \bar 2 & \bar 0 & \bar 1
    \end{array}
    \right]
$
}
\end{array}
$
\\
\\
\hline
\\
$A_7A_1$ 
& 
$
\begin{array}{l}
\KK[T_1,\ldots,T_5] / I \text{ with $I$ generated by }
\\[1ex]
\scriptstyle
-T_{2}T_{3}+T_{4}^{2}-T_{5}^{2},\ 
T_{1}^{2}-T_{3}^{2}+T_{4}T_{5}
\end{array}
$
& 
$
\begin{array}{l}
\ZZ \oplus \ZZ/4\ZZ
\\[1ex]
\mbox{\tiny$
     \left[
    \begin{array}{rrrrr}
    1 & 1 & 1 & 1 & 1\\
    \bar 2 & \bar 2 & \bar 0 & \bar 3 & \bar 1
    \end{array}
    \right]
$
}
\end{array}
$
\\
\\
\hline
\\
$A_5A_2A_1$
& 
$
\begin{array}{l}
\KK[T_1,\ldots,T_7]/I \text{ with $I$ generated by }
\\[1ex]
\scriptstyle
T_{5}^2+T_{6}^2-T_{7}T_{1}, \
T_{4}T_{5}+T_{6}T_{1}-T_{2}T_{6}-T_{7}^2,
\\
\scriptstyle
-T_{3}T_{6}-T_{5}T_{7}+T_{1}T_{4},\
T_{3}^2-T_{6}T_{1}+T_{7}^2,
\\
\scriptstyle
T_{1}T_{5}-T_{2}T_{5}-T_{4}T_{6}+T_{7}T_{3}, \
\\
\scriptstyle
T_{3}T_{4}-T_{6}^2+T_{7}T_{1}-T_{2}T_{7},
\\[1ex]
\text{The class group and degree matrix are}\\[1ex]
\ZZ\oplus \ZZ/6\ZZ
\end{array}
$
& 
$
\begin{array}{l}
\scriptstyle
T_{1}T_{3}-T_{2}T_{3}+T_{6}T_{5}-T_{7}T_{4},\
\\
\scriptstyle
T_{1}T_{2}-T_{2}^2-T_{4}^2+T_{3}T_{5},
\\
\scriptstyle
T_{1}^2-T_{2}^2-T_{4}^2+2T_{3}T_{5}-T_{7}T_{6}
\\[6ex]
\mbox{\tiny$
\left[
\begin{array}{rrrrrrr}
1 & 1 & 1 & 1 & 1 & 1 & 1 \\
\bar 2 & \bar 2 & \bar 3 & \bar 5 & \bar 1 & \bar 4 & \bar 0 
\end{array}
\right]
$
}
\end{array}
$
\\
\\
\hline
\\
$2A_32A_1$ 
& 
$
\begin{array}{l}
\KK[T_1,\ldots,T_9]/I \text{ with $I$ generated by }
\\[1ex]
\scriptstyle
-\frac{1}{2}T_{4}^{2}+T_{5}^{2}+\frac{1}{2}T_{7}T_{9},
\\ 
\scriptstyle
    -\frac{1}{2}T_{3}T_{8}-\frac{1}{2}T_{4}T_{5}+T_{6}^{2},
\\
\scriptstyle
    -\frac{1}{2}T_{3}T_{4}+T_{5}T_{8}+\frac{1}{2}T_{6}T_{9},
\\ 
\scriptstyle
    T_{2}T_{6}-T_{7}T_{9}-{4}T_{8}^{2},
\\
\scriptstyle
    T_{5}T_{2}-{2}T_{3}T_{7}+T_{8}T_{9},
\\ 
\scriptstyle
    -\frac{1}{4}T_{2}T_{4}+\frac{1}{4}T_{3}T_{9}+T_{7}T_{8},
\\
\scriptstyle
    T_{1}T_{7}+T_{2}T_{7}-{4}T_{3}T_{4}+{2}T_{6}T_{9},
\\ 
\scriptstyle
    T_{1}T_{6}-{2}T_{4}^{2}+T_{7}T_{9},
\\
\scriptstyle
    \frac{1}{2}T_{1}T_{6}-\frac{1}{2}T_{2}T_{6}+T_{3}^{2}-T_{4}^{2},
\\ 
\scriptstyle
    T_{1}T_{8}-{2}T_{4}T_{7}+T_{5}T_{9},
 \\[1ex]
 \text{The class group and degree matrix are}\\[1ex]
 \ZZ \oplus \ZZ/2\ZZ \oplus \ZZ/4\ZZ
\end{array}
$
&
$
 \begin{array}{l}
\scriptstyle
    T_{1}^{2}-{16}T_{4}T_{5}+{8}T_{7}^{2}-T_{9}^{2},
\\     
\scriptstyle
    T_{2}^{2}-{16}T_{3}T_{8}+{8}T_{7}^{2}-T_{9}^{2},
\\
\scriptstyle
    T_{2}T_{3}-T_{4}T_{9}+{4}T_{5}T_{7}-{8}T_{6}T_{8},
\\ 
\scriptstyle
    T_{1}T_{2}-{8}T_{7}^{2}-T_{9}^{2},
 \\
 \scriptstyle
    T_{1}T_{5}+{2}T_{3}T_{7}-{4}T_{4}T_{6}+T_{8}T_{9},
\\ 
\scriptstyle
    T_{1}T_{3}-T_{4}T_{9}-{4}T_{5}T_{7},
\\
\scriptstyle
    -\frac{1}{8}T_{4}T_{1}+\frac{1}{8}T_{2}T_{4}+T_{5}T_{6}-T_{7}T_{8},
\\
\scriptstyle
    -\frac{1}{16}T_{9}T_{1}+\frac{1}{16}T_{2}T_{9}-T_{4}T_{8}+T_{6}T_{7},  
\\
\scriptstyle
    -\frac{1}{8}T_{9}T_{1}+\frac{1}{8}T_{2}T_{9}+T_{3}T_{5}-T_{4}T_{8},
\\
\scriptstyle
    \frac{1}{4}T_{1}T_{8}-\frac{1}{4}T_{2}T_{8}+T_{6}T_{3}-T_{4}T_{7},
 \\[4ex]
    \mbox{\tiny$
      \left[
    \begin{array}{rrrrrrrrr}
    1 & 1 & 1 & 1 & 1 & 1 & 1 & 1 & 1 \\
    \bar 1 & \bar 1 & \bar 0 & \bar 1 & \bar 1 & \bar 1 & \bar 0 & \bar 0 & \bar 0 \\
    \bar 3 & \bar 3 & \bar 2 & \bar 0 & \bar 2 & \bar 1 & \bar 3 & \bar 0 & \bar 1
    \end{array}
    \right]
    $}
 \end{array}
$
\\
\\
\hline
\\
$4A_2$ 
& 
$
\begin{array}{l}
\KK[T_1,\ldots,T_{10}] / I \text{ with $I$ generated by }\\[1ex]
\scriptstyle
    {3}T_{3}T_{6}+{3}T_{4}T_{7}\zeta +(-{3}\zeta -{3})T_{5}T_{8},\\ \scriptstyle 
    (\zeta -{1})T_{2}T_{8}+{3}T_{3}^{2}+(-\zeta -{2})T_{6}T_{9},\\ \scriptstyle 
    {3}T_{2}T_{7}\zeta +{3}T_{6}T_{10}+(-{3}\zeta -{3})T_{8}T_{9},\\ \scriptstyle 
    (-\zeta +{1})T_{2}T_{5}+(\zeta -{1})T_{4}T_{9}+{3}T_{6}T_{8},\\ \scriptstyle 
    -\zeta T_{1}T_{10}+T_{2}T_{10}\zeta +{3}T_{4}T_{7}-{3}T_{5}T_{8},\\ \scriptstyle 
    (\zeta +{1})T_{1}T_{10}-T_{2}T_{10}\zeta +{3}T_{5}T_{8}-T_{9}^{2},\\  \scriptstyle 
    -T_{1}T_{9}\zeta -T_{2}T_{9}+{3}T_{3}T_{7}+(\zeta +{1})T_{10}^{2},\\ \scriptstyle 
    -T_{1}T_{9}\zeta +T_{2}T_{9}\zeta +{3}T_{3}T_{7}-{3}T_{8}T_{4},\\  \scriptstyle 
    (-\zeta +{1})T_{1}T_{8}+(\zeta -{1})T_{2}T_{8}+{3}T_{3}^{2}-{3}T_{4}T_{5},\\ \scriptstyle 
    (\zeta +{2})T_{1}T_{7}+(-\zeta -{2})T_{2}T_{7}+{3}T_{4}T_{3}-{3}T_{5}^{2},\\
    \scriptstyle 
    T_{2}T_{1}+(-\zeta -{1})T_{2}^{2}+{3}T_{8}T_{3}+T_{9}T_{10}\zeta,\\ \scriptstyle 
    -\zeta  T_{1}T_{2}+{3}T_{4}T_{6}+T_{9}T_{10}\zeta ,\\ \scriptstyle 
    T_{1}^{2}+(-\zeta -{1})T_{1}T_{2}+{3}T_{5}T_{7}+T_{9}T_{10}\zeta
    \\[1ex]
    \text{The class group and degree matrix are}\\[1ex]
    \ZZ \oplus \ZZ/3\ZZ \oplus \ZZ/3\ZZ
    \end{array}
$
& 
$
 \begin{array}{l}
 \scriptstyle 
    (-\zeta -{1})T_{1}T_{9}+(\zeta +{1})T_{2}T_{9}-{3}T_{8}T_{4}+{3}T_{5}T_{6},\\
    \scriptstyle 
    (-{2}\zeta -{1})T_{1}T_{8}+(\zeta -{1})T_{2}T_{8}+{3}T_{3}^{2}+(\zeta -{1})T_{7}T_{10},
    \\
    \scriptstyle 
    -{3}T_{1}T_{6}+({3}\zeta +{3})T_{2}T_{6}+(-{3}\zeta -{3})T_{7}T_{9}+{3}T_{8}T_{10},
    \\
    \scriptstyle 
    (-\zeta -{2})T_{1}T_{7}+({2}\zeta +{1})T_{2}T_{7}+{3}T_{5}^{2}+(-\zeta -{2})T_{8}T_{9},
    \\ 
    \scriptstyle
    (-\zeta -{2})T_{1}T_{6}+({2}\zeta +{1})T_{2}T_{6}+{3}T_{3}T_{5}+(-\zeta -{2})T_{7}T_{9},\\ \scriptstyle 
    ({2}\zeta +{1})T_{1}T_{6}+(-{2}\zeta -{1})T_{2}T_{6}-{3}T_{3}T_{5}+{3}T_{4}^{2},\\ \scriptstyle 
    (-\zeta +{1})T_{1}T_{5}+(-\zeta -{2})T_{2}T_{5}+({2}\zeta +{1})T_{3}T_{10}+{3}T_{7}^{2},\\ \scriptstyle 
    (-\zeta +{1})T_{1}T_{5}+(\zeta -{1})T_{2}T_{5}-{3}T_{6}T_{8}+{3}T_{7}^{2},\\ \scriptstyle 
    -{3}T_{1}T_{4}+({3}\zeta +{3})T_{2}T_{4}+(-{3}\zeta -{3})T_{3}T_{9}+{3}T_{5}T_{10},\\ \scriptstyle 
    (-{2}\zeta -{1})T_{1}T_{4}+({2}\zeta +{1})T_{5}T_{10}+{3}T_{6}^{2},\\ \scriptstyle 
    (\zeta +{2})T_{1}T_{4}+(-{2}\zeta -{1})T_{2}T_{4}+(\zeta -{1})T_{3}T_{9}+{3}T_{7}T_{8},\\ \scriptstyle 
    (-\zeta +{1})T_{1}T_{3}+(\zeta -{1})T_{5}T_{9}+{3}T_{6}T_{7},\\ \scriptstyle 
    {3}\zeta  T_{1}T_{3}+{3}T_{4}T_{10}+(-{3}\zeta -{3})T_{5}T_{9},\\ \scriptstyle 
    (\zeta +{2})T_{1}T_{3}+(-{2}\zeta -{1})T_{2}T_{3}+(\zeta -{1})T_{5}T_{9}+{3}T_{8}^{2},\\[1ex]
    \text{where $\zeta$ is a primitive third root of unity.}\\[2ex]
    \mbox{\tiny$
     \left[
    \begin{array}{rrrrrrrrrr}
    1 & 1 & 1 & 1 & 1 & 1 & 1 & 1 & 1 & 1 \\
    \bar  2 &\bar  2 &\bar  1 &\bar  0 &\bar  2 &\bar  1 &\bar  2 &\bar  0 &\bar  1 &\bar  0 \\
    \bar  1 &\bar  1 &\bar  2 &\bar  2 &\bar  2 &\bar  0 &\bar  0 &\bar  0 &\bar  1 &\bar  1
    \end{array}
    \right].
$
}
\end{array}
$
\\
\\
\hline
\end{longtable}
\endgroup
\end{theorem}

\begin{proof}
First recall from~\cite[Theorem~8.3]{AlNi}
that the isomorphy classes of Gorenstein log del Pezzo 
surfaces of Picard number one that do not allow a 
$\KK^*§$-action correspond bijectively to singularity types 
we listed.

Each surface $X$ of the list is obtained by contracting curves 
of a smooth surface $X_2$ arising as a blow up of $\PP_2$ 
with generators for the Cox ring known by~\cite{Der,AGL}.
A direct application of Algorithms~\ref{algo:modifycemds}
and~\ref{algo:contractcemds} is not always feasible.
However, we have enough information
to present the blow ups of $\PP_2$ as a weak CEMDS.
As an example, we treat the $D_5A_3$-case.
By~\cite{AGL}, with $X_2 := X_{141}$, additional 
generators for $\Cox(X_2)$ correspond in $\Cox(\PP_2)$ to
\[
 f_1 := T_1-T_2,\qquad\qquad
 f_2 := T_1T_2-T_2^2+T_1T_3.
\]
Using Algorithm~\ref{algo:stretchcemds} with input the CEMDS $\PP_2$
and $(f_1,f_2)$, we obtain a CEMDS $X_1$.
Again by~\cite[Sec.~6]{AGL}, we know the degree matrix $Q_2$
of $X_2$.
Write $Q_2 = [D,C]$ with submatrices
$D$ and $C$ consisting of the first $r_1$ and
the last $r_2-r_1$ columns respectively. 
We compute a Gale dual matrix $P_2$ 
of the form $P_2 = [P_1, B]$ by solving 
 $CB^t = -DP_1^t$.
Let
$p_1\colon \TT^5\to \TT^4$ and
$p_2\colon \TT^{14}\to \TT^4$ 
 be the maps of tori 
corresponding to $P_1$ and $P_2$.
Instead of using Algorithm~\ref{algo:modifycemds},
we directly produce the equations $G_2'$
for $X_2$ on the torus:
\begin{align*}
 p_2^\sharp (p_1)_\sharp \,f_1 \ &=\  
 T_1T_6T_7T_8T_{14}-T_2T_{10}T_{11}^2-T_3T_{12}T_{13}^2,\\
 p_2^\sharp (p_1)_\sharp \,f_2\  &=\ 
 T_1T_4T_{14}^2+T_2T_3T_9T_{11}T_{13}-T_5T_6^2T_7.
\end{align*}
Note that by~\cite{AGL}, the variables define pairwise non-associated 
$\Cl(X_2)$-prime generators for $\Cox(X_2)$.
This makes $X_2$ a weak CEMDS with data
$(P_2,\Sigma_2,G_2')$, where $\Sigma_2$ is the stellar subdivision
of the fan $\Sigma_1$ of the CEMDS $X_1$ at the columns of $B$.
We now use Algorithm~\ref{algo:contractcemds} to
contract on $X_2$ the curves corresponding to the variables
$T_i$ with $i\in\{ 2,3,5,7,8,9,10,12,14\}$.
The resulting ring is the one listed in the table of 
the theorem.
\end{proof}

\begin{remark}
The minimal resolutions $\t{X}$ of the 
surfaces $X$ listed in 
Theorem~\ref{thm:gorensteinlogpezzos}
arise from the plane $\PP_2$ by blowing
up $9-d$ points in almost general position,
where $d$ is the \emph{degree} of the weak
del Pezzo surface $\t{X}$.
In the case of singularity type $A_7$
we obtain degree~2 and in all other cases 
we have degree~1.
\end{remark}

\begin{remark}
The surfaces with singularity type $E_6A_2$, $E_7A_1$ 
and $E_8$ are the only ones in the list 
Theorem~\ref{thm:gorensteinlogpezzos}
for which the minimal resolution has a 
hypersurface as Cox ring, see~\cite[Sect.~3, Table~9]{Der}.
Moreover, these surfaces 
admit small degenerations into $\KK^*$-surfaces. 
In fact, multiplying the monomials $T_2T_3T_4$
and $T_1^2T_4^2$ in the respective Cox rings with a 
parameter $\alpha \in \KK$ gives rise to a flat family 
of Cox rings over $\KK$. 
The induced flat family of surfaces over $\KK$ has
a $\KK^*$-surface as zero fiber, compare also
the corresponding Cox rings listed in~\cite[Theorem~5.6]{HaSu}.
\end{remark}

\begin{remark}
The surfaces of singularity type $A_7$, $A_8$ and $D_8$ 
have hypersurface Cox rings but their resolutions do 
not, see again~\cite[Sect.~3, Table~9]{Der}. For example, the Cox ring of the minimal resolution
of the surface with a $D_8$ singularity is 
$\KK[T_1,\ldots,T_{14}]/I$, where $I$ is generated by
\begingroup
\footnotesize
\begin{gather*}
T_{1}T_{2}T_{11}^4T_{12}^4T_{13}^3T_{14}^2 + T_{3}^2T_{10} - T_{4}T_{7}^2T_{8},
\qquad
T_{1}T_{4} + T_{2}^3T_{7}^2T_{8}^3T_{9}^4T_{10}^2T_{13}T_{14}^2 - T_{6}T_{11},
\\
T_{1}^2T_{2}T_{11}^3T_{12}^4T_{13}^3T_{14}^2 + T_{5}T_{10} - T_{6}T_{7}^2T_{8},
\qquad
T_{1}T_{3}^2 + T_{2}^3T_{7}^4T_{8}^4T_{9}^4T_{10}T_{13}T_{14}^2 - T_{5}T_{11},
\\
T_{1}T_{2}^4T_{7}^2T_{8}^3T_{9}^4T_{10}T_{11}^3T_{12}^4T_{13}^4T_{14}^4
+ T_{3}^2T_{6} - T_{4}T_{5}
\end{gather*}
\endgroup
and the degree matrix of the $\ZZ^9$-grading is given by
$$
\mbox{\tiny $
\left[
\begin{array}{rrrrrrrrrrrrrr}
-1 & -1 &  0 & 0 & -1 & -1 & 0 & 0 & 0 & 0 & 0 & 0 & 0 &  1
\\
-1 & 0 & 0 & 0 & -1 & -1 & 0 & 0 & 0 & 0 & 0 & 0 & 1 & -1
\\
-1 & 0 & 0 & 0 & -1 & -1 & 0 & 0 & 0 & 0 & 0 & 1 & -1 &  0
\\ 
0 & 0 & 0 & 0 & -1 & -1 & 0 & 0 & 0 & 0 & 1 & -1 & 0 &  0
\\
1 & 1 & 1 & 2 & 3 & 3 & 0 & 0 & 0 & 0 & 0 & 0 & 0 &  0
\\
0 & -1 & -1 & -1 & -2 & -1 & 0 & 0 & 0 & 1 & 0 & 0 & 0 &  0
\\ 
0 & -1 & 0 & -1 & 0 & -1 & 0 & 0 & 1 & -1 & 0 & 0 & 0 &  0
\\
0 & 0 & 0 & -1 & 0 & -1 & 0 & 1 & -1 & 0 & 0 & 0 & 0 & 0
\\
0 & 0 & 0 & -1 & 0 & -1 & 1 & -1 & 0 & 0 & 0 & 0 & 0 &  0
\end{array}
\right]
$}.
$$
\noindent
We remark here that the computation of the Cox rings for
the minimal resolutions of the surfaces with singularity 
type $A_7$ and $A_8$ were not feasible on our systems.
\end{remark}

\begin{remark}
\label{rem:anticanmod}
Given a CEMDS $X \subseteq Z$ with Cox ring $R$
and a divisor class $w \in \Cl(X)$, consider 
the graded ring 
$$
R(w) 
\ = \ 
\bigoplus_{m \in \ZZ_\ge 0} R_{mw}
\ \subseteq  \ 
R.
$$
Given monomial generators $f_1, \ldots, f_s$ 
of a Veronese subalgebra $R(dw) \subseteq R(w)$,
where $d > 0$, one obtains homogeneous 
equations for $\Proj(R(w))$ by computing the 
closure of the image of $\b{X}$ under the 
toric morphism 
$$
\b{Z} \ \to \ \KK^s, 
\qquad
z \ \mapsto \ (f_1(z), \ldots, f_s(z)),
$$
where $\b{Z}$ denotes the ambient toric total 
coordinate space.
If the canonical class of $X$ is known, e.g.~if 
$R$ is a complete intersection Cox ring~\cite[Prop.~4.15]{Ha2},
one obtains this way equations for anticanonical 
models. 
For the surfaces $X$ of Theorem~\ref{thm:gorensteinlogpezzos}
with a minimal resolution $\t{X}$ of degree 1,
the anticanonical model is always a hypersurface in 
$\PP(1,1,2,3)$. 
For example, we have the following equations
(listed by singularity type):
$$ 
\begin{array}{ll}
2A_4 \colon
& 
\scriptstyle 
T_{1}^2T_{2}^2T_{3} + T_{1}^2T_{2}T_{4} - 2T_{1}T_{2}T_{3}^2
-T_{1}T_{3}T_{4} + T_{2}T_{3}T_{4} + T_{3}^3 + T_{4}^2
\\[3pt]
D_8 \colon
&
\scriptstyle 
T_{1}^4T_{3} - T_{1}T_{2}T_{3}^2 + T_{3}^3 + T_{4}^2
\\[3pt]
D_5A_3 \colon
&
\scriptstyle 
T_{1}^2T_{2}T_{4} - T_{1}T_{2}T_{3}^2 + T_{1}T_{3}T_{4} -
T_{3}^3 - T_{4}^2
\\[3pt]
D_62A_1 \colon
&
\scriptstyle 
T_{1}^3T_{2}T_{3} - 3T_{1}^2T_{2}^2T_{3} + T_{1}^2T_{3}^2 +
3T_{1}T_{2}^3T_{3} - 3T_{1}T_{2}T_{3}^2 - T_{2}^4T_{3} +
2T_{2}^2T_{3}^2 - T_{3}^3 + T_{4}^2
\\[3pt]
E_6A_2 \colon
&
\scriptstyle 
T_{1}^2T_{2}T_{4} - T_{1}T_{3}T_{4} - T_{3}^3 - T_{4}^2
\\[3pt]
E_7A_1 \colon
&
\scriptstyle 
T_{1}^3T_{2}T_{3} - T_{1}T_{3}T_{4} - T_{3}^3 + T_{4}^2
\\[3pt]
E_8 \colon
&
\scriptstyle 
T_{1}T_{2}^5 - T_{2}^2T_{3}^2 - T_{3}^3 - T_{4}^2
\\[3pt]
A_8 \colon
&
\scriptstyle 
T_{1}^3T_{4} - T_{2}T_{3}T_{4} + T_{3}^3 + T_{4}^2
\\[3pt]
A_7A_1 \colon
&
\scriptstyle 
T_{1}^2T_{2}T_{4} - T_{1}^2T_{3}^2 + T_{2}T_{3}T_{4} - T_{3}^3 + T_{4}^2
\end{array}
$$
\end{remark}

\section{The lattice ideal method}
\label{section:latticeideal}
We consider the blow up of a given Mori dream 
space with known Cox ring and develop a method 
to produce systematically
generators for the new Cox ring.
The key step is a description of the Cox 
ring of a blow up as a saturated Rees algebra.

Let $X_1$ be a Mori dream space and 
$\pi \colon X_2 \to X_1$ the blow up of 
an irreducible subvariety $C \subseteq X_1$ 
contained in the smooth locus of $X_1$.
As before, write $K_i := \Cl(X_i)$ for the divisor 
class groups and $R_i := \mathcal{R}(X_i)$ for 
the Cox rings.
Then we have the canonical pullback maps 
$$ 
\pi^* \colon K_1 \to K_2, \ [D] \mapsto [\pi^*D],
\qquad
\pi^* \colon R_1 \to R_2, 
\ 
(R_1)_{[D]} \ni f \mapsto \pi^*f \in (R_2)_{[\pi^*D]}.
$$
Moreover, identifying $U := X_2 \setminus \pi^{-1}(C)$ with 
$X_1 \setminus C$, we obtain canonical push forward maps
$$ 
\pi_* \colon K_2 \to K_1, \ [D] \mapsto [\pi_*D],
\qquad
\pi_* \colon R_2 \to R_1, 
\ 
(R_2)_{[D]} \ni f \mapsto f_{\vert U} \in (R_1)_{[\pi_*D]}.
$$
Let $J \subseteq R_1$ be the irrelevant ideal,
i.e.~the vanishing ideal of $\b{X}_1 \setminus \rq{X}_1$,
and $I  \subseteq R_1$ the vanishing ideal of 
$p_1^{-1}(C) \subseteq \b{X}_1$
with the  characteristic space
$p_1\colon \rq X_1 \to X_1$.
We define the {\em saturated Rees algebra}
to be the subalgebra
$$ 
R_1[I]^{\rm sat}
\ := \ 
\bigoplus_{d \in \ZZ} (I^{-d}:J^\infty) t^d
\ \subseteq \ 
R_1[t^{\pm 1}],
\qquad
\text{where }
I^k := R_1
\text{ for } 
k \le 0.
$$

\begin{remark}
The usual Rees algebra $R_1[I] = \oplus_{\ZZ} I^ {-d}t^d$ is 
a subalgebra of the saturated Rees algebra $R_{1}[I]^{\rm sat}$.
In the above situation, $I \subseteq R_1$ is a $K_1$-prime
ideal and thus we have $I : J^{\infty} = I$.
Consequently, $R_{1}[I]^{\rm sat}$ equals $R_{1}[I]$
if and only if $R_{1}[I]^{\rm sat}$ is generated in the 
$\ZZ$-degrees $0$ and $\pm 1$.
In this case, $R_{1}[I]^{\rm sat}$ is finitely 
generated because $R_1[I]$ is so.
\end{remark}

Note that the saturated Rees algebra $R_1[I]^{\rm sat}$ is naturally 
graded by $K_1 \times \ZZ$ as $R_{1}$ is $K_{1}$-graded and 
the ideals $I$, $J$ are homogeneous.
Let $E = \pi^{-1}(C)$ denote the exceptional divisor.
Then we have a splitting $K_2 = \pi^*K_1 \times \ZZ \cdot [E]
\cong K_1 \times \ZZ$.

\begin{proposition}
\label{prop:reesalg}
In the above situation, we have the following 
mutually inverse isomorphisms of graded 
algebras
\begin{eqnarray*}
R_2 
& \longleftrightarrow & 
R_1[I]^{\rm sat},
\\
(R_2)_{[\pi^*D]+d[E]} \ni f
& \mapsto & 
\pi_*f \cdot t^d \in (R_1[I]^{\rm sat})_{([D],d)},
\\
(R_2)_{[\pi^*D]+d[E]} \ni \pi^* f \cdot 1_E^d
& \mapsfrom & 
f \cdot t^d \in (R_1[I]^{\rm sat})_{([D],d)}.
\end{eqnarray*}
\end{proposition}

\begin{lemma}
\label{lem:reeslem}
In the above situation, consider the 
characteristic spaces $p_i \colon \rq{X}_i \to X_i$ 
and let 
$\mathcal{I}_C$, $\mathcal{I}_{\rq{C}}$, $\mathcal{I}_E$,
$\mathcal{I}_{\rq{E}}$ be the ideal sheaves of $C$, 
$\rq{C} = p_1^{-1}(C)$, $E$, $\rq{E} = p_2^{-1}(E)$
on $X_1$, $\rq{X}_1$, $X_2$, $\rq{X}_2$ respectively.
Then, for any $m > 0$, we have
$$
p_1^*(\mathcal{I}_{C}^m) \ = \ \mathcal{I}_{\rq{C}}^m,
\qquad
p_2^*(\mathcal{I}_E^m) \ = \ \mathcal{I}_{\rq{E}}^m,
\qquad
\pi^*(\mathcal{I}_{C}^m) \ = \ \mathcal{I}_E^m.
$$
Moreover, with the vanishing ideals $I \subseteq R_1$ 
of the closure of $\rq{C}$ and $J \subseteq R_1$ of
$\Spec \, R_1 \setminus \rq{X}_1$, we have 
$\Gamma(\rq{X}_1,\mathcal{I}_{\rq{C}}^m) = 
I^m : J^\infty$
for any $m > 0$.
\end{lemma}

\begin{proof}
For the first equality, we use that $C$
is contained in the smooth locus of $X_1$. 
This implies that $p_1$ has no multiple fibers 
near $C$ and the claim follows.
The second equality is obtained by the same reasoning.
The third one is a standard fact on blowing up~\cite[Prop.~8.1.7 
and Cor.~8.1.8]{EGA2}.
The last statement follows from the fact that $\rq{X}_1$ 
is quasiaffine.
\end{proof}

\begin{proof}[Proof of Proposition~\ref{prop:reesalg}]
We only have to prove that the maps are 
well defined.
For the map from $R_2$ to $R_1[I]^{\rm sat}$
consider
$f \in (R_2)_{([\pi^* D],d)}$.
We have to show that for $d < 0$, the 
push forward $\pi_* f$ belongs to $I^{-d}:J^{\infty}$.
Note that $\pi_* f = \pi_* f'$ holds with 
$f' := f \cdot 1_E^{-d} \in  (R_2)_{[\pi^* D]}$.
Pushing $f'$ locally to $X_2$, then to $X_1$ 
and finally lifting it to $\rq{X}_1$, we see 
using Lemma~\ref{lem:reeslem}
that $\pi_* f'$ is a global section of
the $-d$-th power of 
the ideal sheaf of $p_1^{-1}(C) \subseteq \rq{X}_1$. 
This gives
$\pi_*f' \in I^{-d} : J^{\infty} \subseteq R_1$.

For the map from $R_1[I]^{\rm sat}$ to $R_2$, 
consider $f \cdot t^d \in (R_1[I]^{\rm sat})_{([D],d)}$, 
where $d < 0$.
We need that $1_E^{-d}$ divides $\pi^*f$ in $R_2$. 
By definition, there exist an $h \in J$ and a $k \ge 0$
such that $h \not\in I$ and $fh^k \in I^{-d}$.
Then we have $(\pi^*f) (\pi^*h)^k \in (\pi^* I)^{-d}$.
Using $\bangle{\pi^*I} = \bangle{1_E}$ and
the fact that $1_E$ is a $K_2$-prime not dividing any 
power of $\pi^*h$ we see that $1_E^{-d}$ divides $\pi^*f$.
\end{proof}

For the computation of the Cox ring $R_2$,
we work in the notation of Setting~\ref{set:ambmod};
in particular $X_1$ comes as a CEMDS $X_1 \subseteq Z_1$.
As before, $C \subseteq X_1$ is an irreducible subvariety
contained in the smooth locus of $X_1$
and $\rq{C} \subseteq \rq{X}_1$ denotes its inverse 
image with respect to $p_1 \colon \rq{X}_1 \to X_1$.
The idea is to stretch the given embedding 
$X_1 \subseteq Z_1$ by suitable generators of the 
vanishing ideal $I \subseteq R_1$ of $ \rq{C} \subseteq \b{X}_1$
and then perform an ambient modification.

\begin{algorithm}[BlowUpCEMDS]
\label{algo:latticeideal}
{\em Input: } a CEMDS $(P_1,\Sigma_1,G_1)$, 
a $K_1$-prime ideal $I = \<f_1,\ldots,f_l\> \subseteq R_1$
with pairwise non-associated $K_1$-primes $f_i\in R_1$
defining an irreducible subvariety $C \subseteq X_1$ inside the 
smooth locus and coprime positive integers $d_1,\ldots,d_l$
with $f_i \in I^{d_i}:J^\infty$.
\begin{itemize}
\item
Compute the stretched CEMDS $(P_1',\Sigma_1',G_1')$ 
by applying Algorithm~\ref{algo:stretchcemds} 
to $(P_1,\Sigma_1,G_1)$ and  $(f_1,\ldots,f_l)$.
\item
Define a multiplicity vector $v \in \ZZ^{r_1+l}$ 
by 
$v_i := 0$ for $1\leq i\leq r_1$
and
$v_{i} := d_{i-r_{1}}$ for $r_{1}+1 \leq i \leq r_{1}+l$.
\item
Determine the stellar subdivision $\Sigma_2 \to \Sigma_1'$ 
of the fan $\Sigma_1'$ along the ray through 
$P_1' \cdot v$.
Set $P_2 := [P_1',P_1'\cdot v]$.
\item
Compute $(P_2,\Sigma_2,G_2)$ 
by applying Algorithm~\ref{algo:modifycemds}
to $(P_1',\Sigma_1',G_1')$ and the pair $(P_2,\Sigma_2)$.
\item
Let $T^\nu$ be the product over all $T_i$
with $C\not\subseteq D_i$
where $D_i\subseteq X_1$ is the divisor corresponding
to $T_i$.
Test whether 
$\dim(I_2 + \<T_{r_2}\>) > \dim(I_2 + \<T_{r_2},T^\nu\>)$.
\item
Set $(P_2',\Sigma_2',G_2') := (P_2,\Sigma_2,G_2)$.
Eliminate all fake relations by 
applying Algorithm~\ref{algo:compresscemds}.
Call the output $(P_2,\Sigma_2,G_2)$. 
\end{itemize}
{\em Output: } $(P_2,\Sigma_2,G_2)$.
If the verification in the next to last step was
positive, then $(P_2,\Sigma_2,G_2)$ is a CEMDS
describing the blow up $X_2$ of $X_1$ along~$C$.
In particular then the $K_2$-graded algebra 
$R_2$ is the Cox ring of~$X_2$.
\end{algorithm}

\begin{proof}
Consider the $K_2$-graded ring 
$R_2 = \KK[T_1,\ldots, T_{r_2}]/I_2$ 
associated to the output $(P_2,\Sigma_2,G_2)$.
The first step is to show that $R_2$ is normal;
then $(P_2,\Sigma_2,G_2)$ is a CEMDS and $R_{2}$ is 
the Cox ring of the output variety $X_{2}$.
In a second step we show that $X_2$ equals
the blow up of $X_1$ along~$C$.

Consider the output $(P_2,\Sigma_2,G_2)$ of the fourth item,
i.e.~the situation before eliminating fake relations.
The variables $T_{r_1+1}, \ldots, T_{r_2-1}$ correspond to 
$f_1,\ldots,f_l$ and $T_{r_2}$ to the exceptional 
divisor.
Observe that we have a canonical $K_2$-graded homomorphism
$R_2 \to R_1[I]^{\rm sat}$ induced by 
$$ 
\KK[T_1,\ldots, T_{r_2}]
\ \to \ 
R_1[I]^{\rm sat},
\qquad
T_i 
\ \mapsto \ 
\begin{cases}
T_i, & 1 \le i \le r_1,
\\
f_{i-r_{1}}t^{-v_i}, & r_1 < i < r_2,
\\
t,        & i = r_2.
\end{cases}
$$
Indeed, 
because $C$ is contained in the smooth locus of
$X_1$, the cone generated by the last $l$ columns 
of $P_1'$ is regular and, because in addition $d_1, \ldots, d_l$ 
are coprime, the vector $P_1' \cdot v$ is primitive.
Thus, the ideal $I_2$ of $X_2$ is the saturation 
with respect to $T_{r_2}$ of 
$$
I_1 + \bangle{T_{i}T_{r_2}^{v_i}-f_{i-r_1}; \; r_1 < i < r_2}
\ \subseteq \ 
\KK[T_1,\ldots,T_{r_2}].
$$
Consequently, the above assignment induces a homomorphism 
$R_2 \to R_1[I]^{\rm sat}$.
This homomorphism induces an isomorphism of 
the $K_2$-graded localizations
$$ 
\left( R_2 \right)_{T_{r_2}}
\ = \ 
\bigoplus_{d \in \ZZ} R_1 T_{r_2}^d
\ \cong \ 
\bigoplus_{d \in \ZZ} R_1 t^d
\ = \ 
\left( R_1[I]^{\rm sat} \right)_{t}
$$
and hence is in particular injective. 
As the image $\phi(R_2)$ contains 
generators
$t=\phi(T_{r_2})$ and $\phi(T_{r_1+i}T_{r_2}^{v_i-1}) = f_it^{-1}$
for the Rees algebra $R_1[I]$,
we obtain
$$
R_1[I]
\ \ \subseteq \ \ 
\phi(R_2)
\ \ \subseteq \ \ 
R_1[I]^{\rm sat}.
$$

We show that $\phi(R_2)$ equals $R_1[I]^{\rm sat}$.
Otherwise, the algebras must be different in some degree,
i.e.~we can choose $n\in \ZZ_{\geq 1}$ minimal such that
there is 
$ft^{-n}\in (R_1[I]^{\rm sat})_{-n}\setminus \phi(R_2)_{-n}$.
The minimality implies 
$ft^{-n+1}\in \<t\>_{R_1[I]^{\rm sat}}\cap \phi(R_2)$
with $ft^{-n+1}\not\in \<t\>_{\phi(R_2)}$.
In particular,
$\<t\>_{\phi(R_2)}$
is properly contained in
$\<t\>_{R_1[I]^{\rm sat}}\cap \phi(R_2)$. 

Define $T^\nu$ as the product over all $T_i$
such that $C\not\subseteq D_i$
where $D_i\subseteq X_1$ is the divisor corresponding to $T_i$.
Note that $T^\nu \in J\subseteq R_1$.
In the localized algebras
$$
R_1[I]_{T^\nu}
\ =\ 
\phi(R_2)_{T^\nu}
\ =\  
\left(R_1[I]^{\rm sat}\right)_{T^\nu}
$$
the ideal
$\<t\>_{\phi(R_2)}$
and the $K_2$-prime ideal
$\<t\>_{R_1[I]^{\rm sat}}\cap \phi(R_2)$
coincide, i.e.~$t$ is $K_2$-prime in $\phi(R_2)_{T^\nu}$.
 By the dimension test we know that $t$ and $T^\nu$
are coprime in $\phi(R_2)$.
Consequently, $t$ is $K_2$-prime in $\phi(R_2)$.
Therefore, 
$\<t\>_{R_1[I]^{\rm sat}}\cap \phi(R_2)=
\<t\>_{\phi(R_2)}$
in $\phi(R_2)$, a contradiction.
We conclude $R_2 \cong R_1[I]^{\rm sat}$.

By Proposition~\ref{prop:reesalg}, 
$R_1[I]^{\rm sat}\cong R_2$ is the Cox ring of 
the blow up $X_2'$ of $X_1$ at $C$.
In particular, 
$R_2\cong R_2'$ is normal.
We may now apply 
Algorithm~\ref{algo:compresscemds}.
Note that there is no need to use the \texttt{verify} option
as the variables $T_1,\ldots,T_{r_2}\in R_2$ are $K_2$-prime
and the generators surviving the elimination process are 
$K_2$-prime as well. As for any Cox ring,
the $K_2$-grading is almost-free.

We show that $X_2 \cong X_2'$ holds. 
Let $\lambda \subseteq \Mov(R_2)$ be the chamber 
representing $X_1$. 
Then $\lambda$ is of codimension one in 
$\QQ \otimes K_2$ and lies on the boundary 
of $\Mov(R_2)$. 
Since there are the contraction morphisms
$X_2 \to X_1$ and $X_2' \to X_1$, the chambers
$\lambda_2, \lambda_2'$ corresponding to 
$X_2$, $X_2'$ both have $\lambda$ as a face.
We conclude $\lambda_2 = \lambda_2'$ and thus
$X_2 \cong X_2'$.
\end{proof}

An important special case of  Algorithm~\ref{algo:latticeideal}
is blowing up a smooth point.
The point $x_1 \in X_1$ can be given in Cox coordinates, 
i.e.~as a point $z \in \rq{X}_1 \subseteq \KK^{r_1}$ 
with $x_1 = p_1(z)$.
The steps are as follows.

\begin{remark}[Blow up of a point]
Let $X_1 =(P_1,\Sigma_1,G_1)$ be a CEMDS
and $x \in X_1$ a smooth point given in 
Cox coordinates $z \in \KK^{r_1}$.
Let $i_1,\ldots, i_k$ be the indices 
with $z_{i_j} \ne 0$ and $\nu_1, \ldots, \nu_s \in \ZZ^r$ 
a lattice basis for
${\rm im}(P_1^*)\cap \lin(e_{i_1},\ldots,e_{i_k})$.
The {\em associated ideal\/}
to $P_1$ and $z$ is
\begin{eqnarray*}
I(P_1,z) 
& := &
\bangle{
z^{-\nu_1^+}T^{\nu_1^+} - z^{-\nu_1^-}T^{\nu_1^-}, 
\ldots, 
z^{-\nu_s^+}T^{\nu_s^+} - z^{-\nu_s^-}T^{\nu_s^-}
} : (T_1 \cdots T_r)^\infty
\\
&\hphantom{ := }& 
+\ \bangle{T_{j};\ z_j\,=\,0}
\ \ \subseteq\ \ \KK[T_1, \ldots, T_r],
\end{eqnarray*}
where $\nu_i = \nu_i^+ + \nu_i^-$ is the unique 
representation as a sum of a nonnegative
and a nonpositive vector having disjoint supports.
The ideal is generated by variables and binomials;
for general $z$ it is a lattice ideal see~\cite{MiSt}.
Then $I(P_1,z)\subseteq R_1$ is the vanishing ideal
of the orbit closure $\b{H_1\cdot z}$ in $\b X_1$.
Let $(f_1,\ldots,f_l)$ be a list of pairwise 
non-associated $K_1$-prime generators 
for $I(P_1,z_1)\subseteq R_1$
and $d_1,\ldots,d_l\in \ZZ_{\geq 1}$.
Then the Cox ring of the blow up of $X_1$ in $x$ 
can be computed with Algorithm~\ref{algo:latticeideal}
with input $(f_1,\ldots,f_l)$ and $(d_1,\ldots,d_l)$.
\end{remark}

The following algorithm produces systematically
generators and their multiplicities $d_i\in \ZZ_{\geq 1}$ 
of the Cox ring of a blow up of a Mori dream space
in the sense that adds step by step generator sets 
for the positive Rees algebra components.

\begin{algorithm}[BlowUpCEMDS2]
\label{algo:latticeideal2}
{\em Input: } a CEMDS $(P_1,\Sigma_1,G_1)$, 
a $K_1$-prime ideal $I$
defining an irreducible subvariety $C \subseteq X_1$ inside the 
smooth locus.
\begin{itemize}
\item 
Let $F$ and $D$ be empty lists.
\item 
For each $k=1,2,\ldots\in \ZZ_{\geq 1}$ do
\begin{itemize}
\item
compute a set $G_k$ of generators
for $A_k := I^k : J^\infty\subseteq R_1$.
Let $f_{k1},\ldots,f_{kl_i}$ be 
a maximal subset of pairwise non-associated 
elements of $G_k$ with 
\[
\qquad\qquad
f_{kj}
\ \not\in\ 
A_1A_{k-1} + \ldots + A_{\lfloor \frac{k}{2} \rfloor }A_{\lceil \frac{k}{2}\rceil}
\quad \text{ if }\ k>1.
\]
\item
Determine integers $d_{k1},\ldots,d_{kl_i}\in \ZZ_{\geq k}$
such that
$f_{kj}\in A_{d_{ki}}\setminus A_{d_{ki}+1}$.
\item 
Add the elements of $f_{k1},\ldots,f_{kl_i}$ 
to $F$ that are not associated to any other element of $F$.
Add the respective integers among $d_{k1},\ldots,d_{kl_i}$ to~$D$.
\item
Run Algorithm~\ref{algo:latticeideal}
with input $(P_1,\Sigma_1,G_1)$, $F$ and~$D$.
\item 
If Algorithm~\ref{algo:latticeideal} terminated
with $(P_2,\Sigma_2,G_2)$ and positive verification,
return $(P_2,\Sigma_2,G_2)$.
\end{itemize}
\end{itemize}
{\em Output (if provided): } 
the algorithm terminates if and only if $X_2$ is 
a Mori dream space.
In this case, the CEMDS $(P_2,\Sigma_2,G_2)$ describes the blow 
up $X_2$ of $X_1$ along $C$.
In particular then the $K_2$-graded algebra 
$R_2$ is the Cox ring of~$X_2$.
\end{algorithm}

\begin{proof}
Note that each $f_{ki}$ is a $K_1$-prime element.
Otherwise, $f_{ki}=f_1f_2$ 
with $K_1$-homogeneous elements $f_i\in R_1$.
As $I$ is $K_1$-prime, 
$f_1$ or $f_2$ lies in $A_{k'}$ with 
$k'<k$, i.e.~$f_{ki}\in A_{k'}$.
This contradicts the choice of~$f_{ki}$.

By Proposition~\ref{prop:reesalg},
the Cox ring $R_2$ of the blow up is isomorphic to the 
saturated Rees algebra $R_1[I]^{\rm sat}$.
After the $k$-th step, 
$(F,T_1,\ldots,T_{r_1},t)$
are generators for a subalgebra $B_k\subseteq R_1[I]^{\rm sat}$
such that
\[
\KK\left[
\{t\}
\,\cup\,
R_1
\,\cup\, 
A_1t^{-1}
\,\cup\,
\ldots
\,\cup\,
A_kt^{-k}
\right]
\ \subseteq\ 
B_k
\ \subseteq\ 
\bigoplus_{k\in \ZZ}
A_{k}t^{-k}
\ =\ 
R_1[I]^{\rm sat}.
\]
If the algorithm stops,
 by the correctness of Algorithm~\ref{algo:latticeideal},
the output then is a CEMDS
describing the blow $X_2$ with Cox ring $R_2$.
Vice versa, 
if $X_2$ has finitely generated Cox ring,
there is $k_0 \geq 1$ with
$R_1[I]^{\rm sat} = B_{k_0}$.
Then Algorithm~\ref{algo:latticeideal} is called with
$K_1$-prime non-associated generators
for $\Cox(X_2)\cong R_1[I]^{\rm sat}$
and thus terminates with positive verification.
\end{proof}

\begin{example}
\label{ex:wpp345}
We compute the Cox ring of the blow up
of the weighted projective plane $X_1 := \PP_{3,4,5}$ 
at the general point
with Cox coordinates $z_1 := (1,1,1)\in \KK^3$
by the steps of Algorithm~\ref{algo:latticeideal2}.
The lattice ideal of $z_1$ with respect to $P_1$ is
\begin{center}
\begin{minipage}{5cm}
  \begin{align*}
  I(P_1,z_1)
  \ &=\
  \<
   T_{2}^{2}-T_{1}T_{3},\
    T_{1}^{2}T_{2}-T_{3}^{2},\
    T_{1}^{3}-T_{2}T_{3}\>,
    \end{align*}
\end{minipage}
\qquad
\begin{minipage}{3cm}
\[
P_1 \ :=\
\left[
\mbox{\tiny $
\begin{array}{rrr}
1 & -2 & 1 \\
-2 & -1 & 2
\end{array}
$}
\right].
\]
\end{minipage}
\end{center}
An application of Algorithm~\ref{algo:latticeideal}
with the three generators $(f_1,f_2,f_3)$ of $I := I(P_1,z_1)$
and all $d_i := 1$ fails.
However, adding the additional generator
\begin{eqnarray*}
f_4
&:=&
T_{1}^{5}-{3}T_{1}^{2}T_{2}T_{3}+T_{1}T_{2}^{3}+T_{3}^{3}
\ \ \in\ \ I^2:J^\infty
\end{eqnarray*}
with $d_4 := 2$ to the input, Algorithm~\ref{algo:latticeideal}
returns the $\Cl(X_2)=\ZZ^2$-graded Cox ring $R_2=\Cox(X_2)$ of the
blow up $X_2$ of $X_1$ in $[z_1]$.
All verifications are positive.
The ring is given as $R_2 = \KT{8}/I_2$ with generators for $I_2$
and the degree matrix being
\begin{center}
\begin{minipage}{3cm}
  \begin{align*}
  \mbox{\tiny $
\begin{array}{ll}
    -T_{1}T_{7}+T_{4}T_{5}+T_{6}^{2},\,\ \
    &T_{1}T_{4}^{2}-T_{2}T_{7}+T_{5}T_{6},\,\\
    -T_{1}T_{4}T_{6}-T_{3}T_{7}+T_{5}^{2},\,
    &-T_{1}T_{5}+T_{2}T_{6}+T_{3}T_{4},\,\\
     T_{2}^{2}-T_{1}T_{3}-T_{4}T_{8},\,
    &T_{1}^{3}-T_{2}T_{3}-T_{6}T_{8},\,\\
    T_{1}^{2}T_{4}-T_{2}T_{5}+T_{3}T_{6},\,
    &T_{1}^{2}T_{6}+T_{1}T_{2}T_{4}-T_{3}T_{5}-T_{7}T_{8},\,\\
    T_{1}^{2}T_{2}-T_{3}^{2}-T_{5}T_{8}
 \end{array}
 $}
    \end{align*}
\end{minipage}
\
\begin{minipage}{5cm}
    \[
    \left[
    \mbox{\tiny $
    \begin{array}{rrrrrrrr}
    3 & 4 & 5 & -1 & 1 & 0 & -3 & 9 \\
    0 & 0 & 0 & 1 & 1 & 1 & 2 & -1
    \end{array}
    $}
    \right].
    \]
\end{minipage}
\end{center}
\end{example}

In Algorithm~\ref{algo:latticeideal},
the saturation computation may become infeasible.
In this case, the following variant can be used
to obtain at least finite generation.

\begin{algorithm}[Finite generation]
\label{algo:fg}
{\em Input: } 
a CEMDS $(P_1,\Sigma_1,G_1)$, 
a $K_1$-prime ideal $I = \<f_1,\ldots,f_l\> \subseteq R_1$
with pairwise non-associated $K_1$-primes $f_i$
defining an irreducible subvariety $C \subseteq X_1$ inside the 
smooth locus and coprime positive integers $d_1,\ldots,d_l$
with $f_i \in I^{d_i}:J^\infty$.
\begin{itemize}
\item
Compute the stretched CEMDS $(P_1',\Sigma_1',G_1')$ 
by applying Algorithm~\ref{algo:stretchcemds} 
to $(P_1,\Sigma_1,G_1)$ and  $(f_1,\ldots,f_l)$.
\item
Define a multiplicity vector $v \in \ZZ^{r_1+l}$ 
by $v_i := 0$ if $1\le i \le r_1$ and 
$v_i := d_{i-r_1}$ for $r_1+1 \leq i \leq r_1+l$.
\item
Determine the stellar subdivision $\Sigma_2 \to \Sigma_1'$ 
of the fan $\Sigma_1'$ along the ray through 
$P_1' \cdot v$.
Set $P_2 := [P_1',P_1'\cdot v]$.
\item
Compute $G_2' := (h_1,\ldots,h_s)$,
where $h_i = p_2^\sharp(p_1')_\sharp(g_i)$
and $G_1' = (g_1,\ldots,g_s)$.
\item
Choose a system of generators $G_2$ of
an ideal $I_2\subseteq \KT{r_2}$ with
$\bangle{G_2'}:(T_1\cdots T_{r_2})^\infty\supseteq I_2 \supseteq \bangle{G_2'}$.
\item
Check if $\dim(I_2) - \dim(I_2 + \bangle{T_i,T_j})\geq 2$ 
for all $i\not= j$.
\item
Check if $T_{r_2}$
is prime in $\KK[T_j^{\pm 1}; \; j \ne r_2][T_{r_2}]/I_2$.
\end{itemize}
{\em Output: } 
$(P_2,\Sigma_2,G_2)$.
The ES $(P_2,\Sigma_2,G_2)$ describes the blow 
up $X_2$ of $X_1$ along $C$.
If all verifications in the last steps were 
positive, the Cox ring $\Cox(X_2)$ is finitely generated and 
is given by the $H_2$-equivariant normalization of 
$\KT{r_2} / I_2:(T_1\cdots T_{r_2})^\infty$.
\end{algorithm}

\begin{proof}
By the last verification, the exceptional divisor $D_{r_2}\subseteq X_2$
inherits a local defining equation from the toric ambient variety $Z_2$.
 Thus, the ambient modification is
  neat in the sense of~\cite[Def.~5.4]{Ha2}.
  By~\cite[Prop.~5.5]{Ha2}, $X_2\subseteq Z_2$ is a neat embedding.
  In turn, the dimension checks enable us to use~\cite[Cor.~2.7]{Ha2},
  which completes the proof.
\end{proof}

\section{Smooth rational surfaces}
\label{sec:smoothrat}

We consider smooth rational surfaces $X$ of Picard 
number $\varrho(X) \le 6$. 
Using Algorithm~\ref{algo:latticeideal}, we show 
that they are all Mori dream surfaces and we 
compute their Cox rings.
Recall that every smooth rational surface $X$ of 
Picard number $\varrho(X) = k$ can be obtained by 
blowing up the projective plane $\PP_2$ at $k-1$ 
points or a Hirzebruch surface $\Fa$ at $k-2$ 
points, where, in both cases, some points may be 
infinitely near, i.e.~one also performs iterated blow ups.
Whereas blow ups of the projective plane $\PP_2$ 
can be done in a purely computational manner, 
the treatment of the (infinitely many) Hirzebruch 
surfaces $\Fa$ requires also theoretical work
due to their parameter $a \in \ZZ_{\ge 1}$.

In the following statement,
we concentrate on those surfaces $X$ that do 
not admit a (non-trivial) $\KK^*$-action; for 
the full list of Cox rings in the case 
$\varrho(X) \le 5$, we refer to~\cite{Ke}.
Note that the rational $\KK^*$-surfaces as well
as the toric ones admit a combinatorial description 
which opens a direct approach to their Cox rings,
see~\cite{Ha3, Hug}.
We denote the iterated blow up
of a point $x$ by sequences of $g_i$ and $s_i$ indicating general
and special points on the $i$-th exceptional divisor over $x$. For
example, $s_3g_2g_1$ indicates a fourfold blow up of a point
with a general point $g_1$ on the first exceptional divisor,
a general point $g_2$ on the second and a special point $s_3$ on
the third. The special points will be precisely defined in
each case.

\begin{theorem}
\label{thm:pic6}
 Let $X$ be a smooth rational surface with Picard number 
 $\varrho(X) \leq 6$.
  Then $X$ is a Mori dream space. 
  If $\varrho(X) \leq 5$, then either $X$ admits 
  a $\KK^*$-action or is isomorphic to $\b M_{0,5}$, 
  the blow up of $\PP_2$ in four general points. 
  If $\varrho(X) = 6$, then $X$ admits a $\KK^*$-action
  or its Cox ring is isomorphic to exactly one of the following,
  where $a\in \ZZ_{\geq 3}$.
 \begin{enumerate}
\item
The $\ZZ^6$-graded ring $\KT{10}/I$, where 
generators for $I$ and the degree matrix are given as
\begin{center}
\begin{minipage}{7cm}
\tiny
\begin{gather*}
\begin{array}{l}
T_{3}^{2}T_{4}-T_{1}T_{2}-T_{6}T_{7}T_{8}T_{10},\ 
T_{1}T_{2}^{2}T_{3}T_{4}T_{5}-T_{6}^{2}T_{7}-T_{9}T_{10}
\end{array}
\end{gather*}
\end{minipage}
\\[1ex]
\begin{minipage}{6cm}
\mbox{\tiny$
  \left[
    \mbox{\tiny $
    \begin{array}{rrrrrrrrrr}
    1 & 0 & 0 & 1 & 0 & 0 & 2 & 0 & 3 & -1 \\
    0 & 1 & 0 & 1 & 0 & 0 & 3 & 0 & 5 & -2 \\
    0 & 0 & 1 & -2 & 0 & 0 & -1 & 0 & -2 & 1 \\
    0 & 0 & 0 & 0 & 1 & 0 & 1 & 0 & 2 & -1 \\
    0 & 0 & 0 & 0 & 0 & 1 & -2 & 0 & -1 & 1 \\
    0 & 0 & 0 & 0 & 0 & 0 & 0 & 1 & 1 & -1
    \end{array}
    $}
    \right]  
$}
\end{minipage}
\end{center}
In this case, $X$ is a $5$-fold blow up of  $\PP_2$ in a point 
of type $g_4s_3g_2g_2$, where $s_3$ is the intersection
point of the 3rd and the 2nd exceptional divisor. 
\begin{center}
\begin{minipage}{1cm}
\tiny
\begin{tikzpicture}[scale=.22]
\coordinate (p1) at (-2,0);
\coordinate (p2) at (2,0);
\coordinate (p3) at (0,3);

\fill[color=black!25] (p1)--(p2)--(p3)--cycle;
\draw (p1)--(p2)--(p3)--cycle;
 
\fill (p1) circle (.20cm) node[anchor=east]{$5$};
\end{tikzpicture}
\end{minipage}
\end{center}

\item
 The $\ZZ^6$-graded ring $\KT{10}/I$, where 
generators for $I$ and the degree matrix are given as
\begin{center}
\begin{minipage}{7cm}
\tiny
\begin{gather*}
\begin{array}{l}
T_{3}T_{5}T_{8}-T_{2}T_{6}-T_{9}T_{10},\ 
T_{1}T_{5}+T_{7}T_{8}-T_{2}T_{6}^{2}T_{4}T_{10}
\end{array}
\end{gather*}
\end{minipage}
\\[1ex]
\begin{minipage}{6cm}
\mbox{\tiny$
  \left[
    \mbox{\tiny $
    \begin{array}{rrrrrrrrrr}
    1 & 0 & 0 & 0 & 0 & 0 & 1 & 0 & -1 & 1 \\
    0 & 1 & 0 & 0 & 0 & 0 & -1 & 1 & 2 & -1 \\
    0 & 0 & 1 & 0 & 0 & 0 & 1 & -1 & 0 & 0 \\
    0 & 0 & 0 & 1 & 0 & 0 & 0 & 0 & 1 & -1 \\
    0 & 0 & 0 & 0 & 1 & 0 & 2 & -1 & -1 & 1 \\
    0 & 0 & 0 & 0 & 0 & 1 & -1 & 1 & 3 & -2
    \end{array}
    $}
    \right]  
$}
\end{minipage}
\end{center}
In this case, $X$ is a $3$-fold blow up of $\PP_2$
in $[1,0,0]$ of type $g_1g_1$, 
and single blow ups in
$[0,1,0]$ and $[1,1,0]$.
\begin{center}
\begin{minipage}{1cm}
\tiny
\begin{tikzpicture}[scale=.22]
\coordinate (p1) at (-2,0);
\coordinate (p2) at (2,0);
\coordinate (p3) at (0,3);

\fill[color=black!25] (p1)--(p2)--(p3)--cycle;
\draw (p1)--(p2)--(p3)--cycle;
 
\fill (p1) circle (.20cm) node[anchor=east]{$3$};
\fill (p2) circle (.20cm);
\fill ($(p1)!.5!(p2)$) circle (.20cm);
\end{tikzpicture}
\end{minipage}
\end{center}
\item
The $\ZZ^6$-graded ring $\KT{11}/I$, where 
generators for $I$ and the degree matrix are given as
\begin{center}
\begin{minipage}{7cm}
\tiny
\begin{gather*}
\begin{array}{ll}
T_{3}^{2}T_{4}T_{5}^{2}T_{8}-T_{2}T_{7}-T_{11}T_{10},\, &
    T_{2}^{2}T_{4}T_{6}^{2}T_{11}-T_{5}T_{9}+T_{8}T_{10},\,\\
    T_{1}T_{5}+T_{7}T_{8}-T_{2}T_{4}T_{6}^{2}T_{11}^{2},\,&
    T_{3}^{2}T_{4}T_{5}T_{8}^{2}+T_{1}T_{2}-T_{9}T_{11},\,\\
    T_{3}^{2}T_{4}^{2}T_{5}T_{8}T_{2}T_{6}^{2}T_{11}-T_{7}T_{9}-T_{1}T_{10}
\end{array}
\end{gather*}
\end{minipage}
\\[1ex]
\begin{minipage}{6cm}
\mbox{\tiny$
  \left[
    \mbox{\tiny $
    \begin{array}{rrrrrrrrrrr}
    1 & 0 & 0 & 0 & 1 & 0 & 2 & 0 & 0 & 1 & 1 \\
    0 & 1 & 0 & 0 & 1 & 0 & 1 & 0 & 1 & 2 & 0 \\
    0 & 0 & 1 & 0 & 0 & 0 & 1 & -1 & 0 & 1 & 0 \\
    0 & 0 & 0 & 1 & 1 & 0 & 2 & -1 & 0 & 2 & 0 \\
    0 & 0 & 0 & 0 & 2 & 0 & 3 & -1 & -1 & 2 & 1 \\
    0 & 0 & 0 & 0 & 0 & 1 & 0 & 0 & 1 & 1 & -1
    \end{array}
    $}
    \right]  
$}
\end{minipage}
\end{center}
In this case, $X$ is the $3$-fold blow up of $\PP_2$ in 
$[1,0,0]$ of type $g_2g_1$ 
and single blow ups in 
$[0,1,0]$ and $[1,1,0]$.
\begin{center}
\begin{minipage}{1cm}
\tiny
\begin{tikzpicture}[scale=.22]
\coordinate (p1) at (-2,0);
\coordinate (p2) at (2,0);
\coordinate (p3) at (0,3);

\fill[color=black!25] (p1)--(p2)--(p3)--cycle;
\draw (p1)--(p2)--(p3)--cycle;
 
\fill (p1) circle (.20cm) node[anchor=east]{$3$};
\fill (p2) circle (.20cm);
\fill ($(p1)!.5!(p2)$) circle (.20cm);
\end{tikzpicture}
\end{minipage}
\end{center}
\item
The $\ZZ^6$-graded ring $\KT{10}/I$, where 
generators for $I$ and the degree matrix are given as
\begin{center}
\begin{minipage}{7cm}
\tiny
\begin{gather*}
\begin{array}{l}
T_{1}T_{5}+T_{7}T_{8}-T_{2}T_{4}T_{6}T_{10},\ 
    T_{3}T_{5}T_{7}T_{8}^{2}-T_{2}^{2}T_{4}-T_{9}T_{10}
\end{array}
\end{gather*}
\end{minipage}
\\[1ex]
\begin{minipage}{6cm}
\mbox{\tiny$
  \left[
    \mbox{\tiny $
    \begin{array}{rrrrrrrrrr}
    1 & 0 & 0 & 0 & 0 & 0 & 2 & -1 & -1 & 1 \\
    0 & 1 & 0 & 0 & 0 & 0 & -2 & 2 & 3 & -1 \\
    0 & 0 & 1 & 0 & 0 & 0 & 1 & -1 & 0 & 0 \\
    0 & 0 & 0 & 1 & 0 & 0 & -1 & 1 & 2 & -1 \\
    0 & 0 & 0 & 0 & 1 & 0 & 3 & -2 & -1 & 1 \\
    0 & 0 & 0 & 0 & 0 & 1 & 0 & 0 & 1 & -1
    \end{array}
    $}
    \right]  
$}
\end{minipage}
\end{center}
In this case, $X$ is the $3$-fold blow up of $\PP_2$ in 
$[1,0,0]$ of type $g_1s_1$ with the intersection point $s_1$ 
 of the 1st exceptional divisor 
 and the transform of $V(T_3)\subseteq \PP_2$
and single blow ups in 
$[0,1,0]$ and $[1,1,0]$.
\begin{center}
\begin{minipage}{1cm}
\tiny
\begin{tikzpicture}[scale=.22]
\coordinate (p1) at (-2,0);
\coordinate (p2) at (2,0);
\coordinate (p3) at (0,3);

\fill[color=black!25] (p1)--(p2)--(p3)--cycle;
\draw (p1)--(p2)--(p3)--cycle;
 
\fill (p1) circle (.20cm) node[anchor=east]{$3$};
\fill (p2) circle (.20cm);
\fill ($(p1)!.5!(p2)$) circle (.20cm);
\end{tikzpicture}
\end{minipage}
\end{center}

\item
The $\ZZ^6$-graded ring $\KT{13}/I$, where 
generators for $I$ and the degree matrix are given as
\begin{center}
\begin{minipage}{8cm}
\tiny
\begin{gather*}
\begin{array}{ll}
    T_{1}T_{11}-T_{4}T_{3}T_{9}-T_{8}T_{12},\, & 
    T_{1}T_{7}-T_{2}T_{8}+T_{3}T_{9}T_{13},\,\\
    T_{2}T_{6}+T_{7}T_{10}-T_{3}T_{5}T_{13},\,&
    T_{1}T_{6}+T_{8}T_{10}-T_{3}T_{4}T_{13},\,\\
   T_{2}T_{11}-\lambda T_{5}T_{3}T_{9}-T_{7}T_{12},\,&
   (\lambda -1)T_{1}T_{5}-T_{10}T_{9}-T_{12}T_{13},\,\\
   (\lambda -1)T_{5}T_{8}+T_{6}T_{9}-T_{11}T_{13},\,&
   T_{10}T_{11}-(\lambda -1)T_{4}T_{3}T_{5}+T_{6}T_{12},\\
   (\lambda -1)T_{4}T_{7}+\lambda T_{6}T_{9}-T_{11}T_{13},\,&
   (\lambda -1)T_{2}T_{4}-\lambda T_{10}T_{9}-T_{12}T_{13}.
\end{array}
\end{gather*}
\end{minipage}
\\[1ex]
\begin{minipage}{6cm}
\mbox{\tiny$
  \left[
    \mbox{\tiny $
       \begin{array}{rrrrrrrrrrrrr}
    1 & 0 & 0 & 0 & -1 & 0 & 0 & 1 & 0 & 0 & -1 & -1 & 1 \\
    0 & 1 & 0 & 0 & 1 & 0 & 0 & -1 & 0 & 1 & 0 & 1 & 0 \\
    0 & 0 & 1 & 0 & 0 & 0 & 0 & 0 & 0 & 0 & 1 & 1 & -1 \\
    0 & 0 & 0 & 1 & 1 & 0 & 0 & 0 & 1 & 0 & 2 & 2 & -1 \\
    0 & 0 & 0 & 0 & 0 & 1 & 0 & 0 & -1 & 1 & -1 & -1 & 1 \\
    0 & 0 & 0 & 0 & 0 & 0 & 1 & 1 & 1 & -1 & 1 & 0 & 0
    \end{array}
    $}
    \right]  
$}
\end{minipage}
\end{center}
In this case, $X$ is the the blow up of $\PP_2$ in 
$[1,0,0]$, $[0,1,0]$, $[0,0,1]$, $[1,1,1]$ and $[1,\lambda,0]$
where $\lambda \in \KK^*\setminus\{1\}$.
\begin{center}
\begin{minipage}{1cm}
\tiny
\begin{tikzpicture}[scale=.22]
\coordinate (p1) at (-2,0);
\coordinate (p2) at (2,0);
\coordinate (p3) at (0,3);
\coordinate (gen) at (0,1.25);

\fill[color=black!25] (p1)--(p2)--(p3)--cycle;
\draw (p1)--(p2)--(p3)--cycle;
 
\fill (p1) circle (.20cm);
\fill (p2) circle (.20cm);
\fill (p3) circle (.20cm);
\fill (gen) circle (.20cm);
\fill ($(p1)!.72!(p2)$) circle (.20cm) node[anchor=north]{$\lambda$};
\end{tikzpicture}
\end{minipage}
\end{center}

\item
The $\ZZ^6$-graded ring $\KT{16}/I$, where 
generators for $I$ and the degree matrix are given as
\begin{center}
\begin{minipage}{9cm}
\tiny
\setlength\arraycolsep{2pt}
\begin{gather*}
\begin{array}{ll}
    T_{6}T_{12}+\lambda T_{7}T_{14}-T_{8}T_{13},\, &
    T_{5}T_{12}-\mu T_{7}T_{15}-T_{9}T_{13},\,\\
    T_{4}T_{13}-\lambda T_{5}T_{14}-\mu T_{6}T_{15},\, &
    T_{4}T_{12}-\mu T_{8}T_{15}-\lambda T_{9}T_{14},\,\\
    T_{3}T_{11}+T_{7}T_{14}-T_{8}T_{13},\, &
    T_{1}T_{13}-T_{2}T_{14}-T_{3}T_{15},\,\\
    T_{1}T_{11}-T_{8}T_{15}-T_{9}T_{14},\, &
    T_{2}T_{11}-T_{7}T_{15}-T_{9}T_{13},\,\\
    (\lambda -\mu )T_{3}T_{5}+\mu T_{7}T_{10}-T_{13}T_{16},\, &
    (-\lambda +{1})T_{5}T_{14}+(-\mu +{1})T_{6}T_{15}+T_{10}T_{11},\,\\
     (\lambda -{1})T_{5}T_{8}+(-\mu +{1})T_{6}T_{9}-T_{11}T_{16},\, &
    (\lambda -{1})T_{4}T_{7}+(\lambda -\mu )T_{6}T_{9}-T_{11}T_{16},\,\\
    (\mu -{1})T_{3}T_{4}-\mu T_{8}T_{10}+T_{14}T_{16},\, &
    (-\lambda +{1})T_{2}T_{14}+(-\mu +{1})T_{3}T_{15}+T_{10}T_{12},\,\\
    (\lambda \mu -\mu )T_{2}T_{8}+(-\lambda \mu +\lambda )T_{3}T_{9}-T_{12}T_{16}, \,&
 (\lambda -\mu )T_{2}T_{6}+\lambda T_{7}T_{10}-T_{13}T_{16},\,\\
    (\lambda -{1})T_{2}T_{4}-\lambda T_{9}T_{10}-T_{15}T_{16},\, &
    (\lambda \mu -\mu )T_{1}T_{7}+(\lambda -\mu )T_{3}T_{9}-T_{12}T_{16},\,\\
    (\mu -{1})T_{1}T_{6}-T_{8}T_{10}+T_{14}T_{16},\, &
    (\lambda -{1})T_{1}T_{5}-T_{9}T_{10}-T_{15}T_{16}
\end{array}
\end{gather*}
\setlength\arraycolsep{5pt}
\end{minipage}
\\[1ex]
\begin{minipage}{9cm}
\mbox{\tiny$
  \left[
    \mbox{\tiny $
    \begin{array}{rrrrrrrrrrrrrrrr}
    1 & 0 & 0 & 0 & -1 & -1 & 0 & 1 & 1 & -1 & 0 & 1 & -1 & 0 & 0 & 0 \\
    0 & 1 & 0 & 0 & 1 & 0 & 0 & -1 & 0 & 1 & 0 & 0 & 1 & 0 & 1 & 0 \\
    0 & 0 & 1 & 0 & 0 & 1 & 0 & 0 & -1 & 1 & 0 & 0 & 1 & 1 & 0 & 0 \\
    0 & 0 & 0 & 1 & 1 & 1 & 0 & 0 & 0 & 1 & 0 & -1 & 0 & 0 & 0 & 1 \\
    0 & 0 & 0 & 0 & 0 & 0 & 1 & 1 & 1 & -1 & 0 & 0 & -1 & -1 & -1 & 1 \\
    0 & 0 & 0 & 0 & 0 & 0 & 0 & 0 & 0 & 0 & 1 & 1 & 1 & 1 & 1 & -1
    \end{array}
    $}
    \right]  
$}
\end{minipage}
\end{center}
In this case, $X$ is the blow up of $\PP_2$ in 
$[1,0,0]$, $[0,1,0]$, $[0,0,1]$, $[1,1,1]$ and $[1,\lambda,\mu]$
where $\lambda\ne \mu \in \KK^*\setminus\{1\}$.
\begin{center}
\begin{minipage}{.7cm}
\tiny
\begin{tikzpicture}[scale=.22]
\coordinate (p1) at (-2,0);
\coordinate (p2) at (2,0);
\coordinate (p3) at (0,3);
\coordinate (gen) at (.2,1.25);
\coordinate (gen2) at (-.65,1.15);

\fill[color=black!25] (p1)--(p2)--(p3)--cycle;
\draw (p1)--(p2)--(p3)--cycle;
 
\fill (p1) circle (.20cm); 
\fill (p2) circle (.20cm);
\fill (p3) circle (.20cm);
\fill (gen) circle (.20cm); 
\fill (gen2) circle (.20cm) node[anchor=east]{$\mu,\lambda\ $};
\end{tikzpicture}
\end{minipage}
\end{center}

\item
The $\ZZ^6$-graded ring $\KT{11}/I$, where 
generators for $I$ and the degree matrix are given as
\begin{center}
\begin{minipage}{7cm}
\tiny
\begin{gather*}
\begin{array}{ll}
    T_{6}T_{2}T_{4}+T_{5}T_{9}-T_{8}T_{10},\, &
    T_{3}T_{4}T_{8}-T_{1}T_{6}-T_{9}T_{11},\,\\
    T_{3}T_{4}T_{5}+T_{6}T_{7}-T_{11}T_{10},\, &
    T_{1}T_{5}+T_{7}T_{8}-T_{2}T_{4}T_{11},\,\\
    T_{3}T_{4}^{2}T_{2}-T_{7}T_{9}-T_{1}T_{10}
\end{array}
\end{gather*}
\end{minipage}
\\[1ex]
\begin{minipage}{6cm}
\mbox{\tiny$
  \left[
    \mbox{\tiny $
        \begin{array}{rrrrrrrrrrr}
    1 & 0 & 0 & 0 & 0 & 0 & 0 & 1 & 0 & -1 & 1 \\
    0 & 1 & 0 & 0 & 0 & 0 & 0 & 0 & 1 & 1 & -1 \\
    0 & 0 & 1 & 0 & 0 & 0 & 1 & -1 & 0 & 1 & 0 \\
    0 & 0 & 0 & 1 & 0 & 0 & 1 & -1 & 1 & 2 & -1 \\
    0 & 0 & 0 & 0 & 1 & 0 & 1 & 0 & -1 & 0 & 1 \\
    0 & 0 & 0 & 0 & 0 & 1 & -1 & 1 & 1 & 0 & 0
    \end{array}
    $}
    \right]  
$}
\end{minipage}
\end{center}
In this case, $X$ is the blow up of $\PP_2$ in 
the points $[1,0,0]$, $[0,1,0]$, $[0,0,1]$, $[1,1,0]$
and $[1,0,1]$.
\begin{center}
\begin{minipage}{1cm}
\tiny
\begin{tikzpicture}[scale=.22]
\coordinate (p1) at (-2,0);
\coordinate (p2) at (2,0);
\coordinate (p3) at (0,3);
\coordinate (gen) at (0,1.25);

\fill[color=black!25] (p1)--(p2)--(p3)--cycle;
\draw (p1)--(p2)--(p3)--cycle;
 
\fill (p1) circle (.20cm); 
\fill (p2) circle (.20cm);
\fill (p3) circle (.20cm);
\fill ($(p1)!.5!(p2)$) circle (.20cm);
\fill ($(p1)!.5!(p3)$) circle (.20cm);
\end{tikzpicture}
\end{minipage}
\end{center}

\item
The $\ZZ^6$-graded ring $\KT{10}/I$, where 
generators for $I$ and the degree matrix are given as
\begin{center}
\begin{minipage}{7cm}
\tiny
\begin{gather*}
\begin{array}{l}
T_{3}T_{5}T_{8}-T_{2}T_{6}-T_{9}T_{10},\ 
    T_{1}T_{5}+T_{7}T_{8}-T_{2}T_{4}T_{10}
\end{array}
\end{gather*}
\end{minipage}
\\[1ex]
\begin{minipage}{6cm}
\mbox{\tiny$
  \left[
    \mbox{\tiny $
    \begin{array}{rrrrrrrrrr}
    1 & 0 & 0 & 0 & 0 & 0 & 1 & 0 & -1 & 1 \\
    0 & 1 & 0 & 0 & 0 & 0 & -1 & 1 & 2 & -1 \\
    0 & 0 & 1 & 0 & 0 & 0 & 1 & -1 & 0 & 0 \\
    0 & 0 & 0 & 1 & 0 & 0 & 0 & 0 & 1 & -1 \\
    0 & 0 & 0 & 0 & 1 & 0 & 2 & -1 & -1 & 1 \\
    0 & 0 & 0 & 0 & 0 & 1 & -1 & 1 & 1 & 0
    \end{array}
    $}
    \right]  
$}
\end{minipage}
\end{center}
In this case, $X$ is the $2$-fold blow up of $\PP_2$ in 
$[1,0,0]$ of type $g_1$
and single blow ups in $[0,1,0]$, $[0,0,1]$ and $[1,1,0]$.
\begin{center}
\begin{minipage}{1cm}
\tiny
\begin{tikzpicture}[scale=.22]
\coordinate (p1) at (-2,0);
\coordinate (p2) at (2,0);
\coordinate (p3) at (0,3);
\coordinate (gen) at (0,1.25);

\fill[color=black!25] (p1)--(p2)--(p3)--cycle;
\draw (p1)--(p2)--(p3)--cycle;
 
\fill (p1) circle (.20cm) node[anchor=east]{$2$};
\fill (p2) circle (.20cm);
\fill (p3) circle (.20cm);
\fill ($(p1)!.5!(p2)$) circle (.20cm);
\end{tikzpicture}
\end{minipage}
\end{center}

\item
The $\ZZ^6$-graded ring $\KT{10}/I$, where 
generators for $I$ and the degree matrix are given as
\begin{center}
\begin{minipage}{7cm}
\tiny
\begin{gather*}
\begin{array}{l}
T_{1}T_{5}T_{10}-T_{2}T_{6}-T_{7}T_{8},\ 
T_{2}T_{4}T_{7}^{a-1}T_{8}^{a-2}-T_{3}T_{5}-T_{9}T_{10}
\end{array}
\end{gather*}
\end{minipage}
\\[1ex]
\begin{minipage}{6cm}
\mbox{\tiny$
  \left[
    \mbox{\tiny $
    \begin{array}{rrrrrrrrrr}
    1 & 0 & 0 & 0 & 0 & 0 & 0 & 0 & 1 & -1 \\
    0 & 1 & 0 & 0 & 0 & 0 & -a+1 & a & -1 & 1 \\
    0 & 0 & 1 & 0 & 0 & 0 & 1 & -1 & 1 & 0 \\
    0 & 0 & 0 & 1 & 0 & 0 & -1 & 1 & 0 & 0 \\
    0 & 0 & 0 & 0 & 1 & 0 & 1 & -1 & 2 & -1 \\
    0 & 0 & 0 & 0 & 0 & 1 & -a+2 & a-1 & -1 & 1
    \end{array}
    $}
    \right]  
$}
\end{minipage}
\end{center}
In this case, $X$ is the blow up of $\FF_a$ in 
 $[0,1,0,1]$, $[1,0,0,1]$, $[1,1,0,1]$ and $[0,1,1,1]$
 where $a\in \ZZ_{\geq 3}$.
 \begin{center}
 \begin{minipage}{1cm}
\tiny
  \begin{tikzpicture}[scale=.22]
    \fill[color=black!25] (0,0) rectangle (4,-2.5);
    
    \draw (0,0) -- (4,0);
    \draw (0,-2.5) -- (4,-2.5);
    \draw (4,0) -- (4,-2.5);
    \draw (0,0) -- (0,-2.5);
     
     \fill (0,0) circle (.2cm); 
     \fill (2,0) circle (.2cm);    
     \fill (4,0) circle (.2cm);  
     \fill (0,-1.25) circle (.2cm); 
   \end{tikzpicture}
\end{minipage}
 \end{center}

\item
The $\ZZ^6$-graded ring $\KT{10}/I$, where 
generators for $I$ and the degree matrix are given as
\begin{center}
\begin{minipage}{7cm}
\tiny
\begin{gather*}
\begin{array}{l}
T_{2}^{a}T_{4}-T_{3}T_{5}T_{6}^{2}T_{10}-T_{7}T_{8},\ 
T_{1}T_{2}^{a-1}T_{4}T_{8}-T_{3}T_{6}-T_{9}T_{10}  
\end{array}
\end{gather*}
\end{minipage}
\\[1ex]
\begin{minipage}{6cm}
\mbox{\tiny$
  \left[
    \mbox{\tiny $
        \begin{array}{rrrrrrrrrr}
    1 & 0 & 0 & 0 & 0 & 0 & 1 & -1 & 0 & 0 \\
    0 & 1 & 0 & 0 & 0 & 0 & 2a-1 & -a+1 & -a & a \\
    0 & 0 & 1 & 0 & 0 & 0 & -1 & 1 & 2 & -1 \\
    0 & 0 & 0 & 1 & 0 & 0 & 2 & -1 & -1 & 1 \\
    0 & 0 & 0 & 0 & 1 & 0 & 0 & 0 & 1 & -1 \\
    0 & 0 & 0 & 0 & 0 & 1 & -1 & 1 & 3 & -2
    \end{array}
    $}
    \right]  
$}
\end{minipage}
\end{center}
In this case, $X$ is the $3$-fold blow up of $\FF_a$, $a\in \ZZ_{\geq 3}$, in 
 $[0,1,0,1]$ of type $g_1s_1$ with the intersection point $s_1$ of the 1st 
 exceptional divisor and the transform of $V(T_3)\subseteq \FF_a$
 and a single blow up of $[0,1,1,1]$. 
 \begin{center}
 \begin{minipage}{1cm}
\tiny
  \begin{tikzpicture}[scale=.22]
    \fill[color=black!25] (0,0) rectangle (4,-2.5);
    
    \draw (0,0) -- (4,0);
    \draw (0,-2.5) -- (4,-2.5);
    \draw (4,0) -- (4,-2.5);
    \draw (0,0) -- (0,-2.5);
     
     \fill (0,0) circle (.2cm) node[anchor=east]{$3$};     
     \fill (0,-1.25) circle (.2cm); 
   \end{tikzpicture}
\end{minipage}
\end{center}

\item
The $\ZZ^6$-graded ring $\KT{10}/I$, where 
generators for $I$ and the degree matrix are given as
\begin{center}
\begin{minipage}{7cm}
\tiny
\begin{gather*}
\begin{array}{l}
T_{2}^{a}T_{4}-T_{3}T_{6}T_{10}T_{5}-T_{7}T_{8},\ 
T_{1}T_{2}^{a-1}T_{4}T_{7}T_{8}^{2}-T_{3}^{2}T_{5}-T_{9}T_{10}
\end{array}
\end{gather*}
\end{minipage}
\\[1ex]
\begin{minipage}{6cm}
\mbox{\tiny$
  \left[
    \mbox{\tiny $
    \begin{array}{rrrrrrrrrr}
    1 & 0 & 0 & 0 & 0 & 0 & 1 & -1 & 0 & 0 \\
    0 & 1 & 0 & 0 & 0 & 0 & 3a-1 & -2a+1 & -a & a \\
    0 & 0 & 1 & 0 & 0 & 0 & -2 & 2 & 3 & -1 \\
    0 & 0 & 0 & 1 & 0 & 0 & 3 & -2 & -1 & 1 \\
    0 & 0 & 0 & 0 & 1 & 0 & -1 & 1 & 2 & -1 \\
    0 & 0 & 0 & 0 & 0 & 1 & 0 & 0 & 1 & -1
    \end{array}
    $}
    \right]  
$}
\end{minipage}
\end{center}
In this case, $X$ is the $3$-fold blow up of $\FF_a$, $a\in \ZZ_{\geq 3}$, 
in $[0,1,0,1]$ of type $g_2s_1$ with the intersection point $s_1$ 
of the 1st exceptional divisor and the transform of $V(T_1)\subseteq\FF_a$
and a single blow up of $[0,1,1,1]$.
\begin{center}
\begin{minipage}{1cm}
\tiny
  \begin{tikzpicture}[scale=.22]
    \fill[color=black!25] (0,0) rectangle (4,-2.5);
    
    \draw (0,0) -- (4,0);
    \draw (0,-2.5) -- (4,-2.5);
    \draw (4,0) -- (4,-2.5);
    \draw (0,0) -- (0,-2.5);
     
     \fill (0,0) circle (.2cm) node[anchor=east]{$3$};     
     \fill (0,-1.25) circle (.2cm);      
   \end{tikzpicture}
\end{minipage}
\end{center}

\item
$\dagger$
The $\ZZ^6$-graded ring $\KT{9}/I$, where 
generators for $I$ and the degree matrix are given as
\begin{center}
\begin{minipage}{7cm}
\tiny
\begin{gather*}
\begin{array}{l}
 T_{1}T_{2}^{2a-1}T_{4}^{2}
 -
 T_{3}^{2}T_{5}
 -
 T_{1}T_{2}^{a-1}T_{3}T_{4}T_{5}T_{6}T_{7}T_{9}
 +
 T_{7}T_{8}T_{9}^{2}
\end{array}
\end{gather*}
\end{minipage}
\\[1ex]
\begin{minipage}{6cm}
\mbox{\tiny$
  \left[
    \mbox{\tiny $
    \begin{array}{rrrrrrrrr}
    1 & 0 & 0 & 0 & 1 & 0 & 0 & 3 & -1 \\
    0 & 1 & 0 & 0 & 2a-1 & 0 & 0 & 4a-3 & -a+1 \\
    0 & 0 & 1 & 0 & -2 & 0 & 0 & -2 & 1 \\
    0 & 0 & 0 & 1 & 2 & 0 & 0 & 4 & -1 \\
    0 & 0 & 0 & 0 & 0 & 1 & 0 & 2 & -1 \\
    0 & 0 & 0 & 0 & 0 & 0 & 1 & 1 & -1
    \end{array}
    $}
    \right]  
$}
\end{minipage}
\end{center}
In this case, $X$ is the $4$-fold blow up of $\FF_a$, $a\in \ZZ_{\geq 3}$, in 
 $[0,1,0,1]$ of type $g_3g_2s_1$ with the intersection point $s_1$ 
of the 1st exceptional divisor and the transform of $V(T_1)\subseteq\FF_a$.
\begin{center}
\begin{minipage}{1cm}
\tiny
  \begin{tikzpicture}[scale=.22]
    \fill[color=black!25] (0,0) rectangle (4,-2.5);
    
    \draw (0,0) -- (4,0);
    \draw (0,-2.5) -- (4,-2.5);
    \draw (4,0) -- (4,-2.5);
    \draw (0,0) -- (0,-2.5);
     
     \fill (0,0) circle (.2cm) node[anchor=east]{$4$};        
   \end{tikzpicture}
\end{minipage}
\end{center}

\item
$\dagger$
The $\ZZ^6$-graded ring $\KT{9}/I$, where 
generators for $I$ and the degree matrix are given as
\begin{center}
\begin{minipage}{7cm}
\tiny
\begin{gather*}
\begin{array}{l}
  T_{2}^{a}T_{4}
   -
   T_{3}T_{5}T_{6}^{2}
   -T_{7}T_{9}T_{1}T_{2}^{a-1}T_{4}T_{5}T_{6}
   +T_{7}T_{9}^{2}T_{8}
\end{array}
\end{gather*}
\end{minipage}
\\[1ex]
\begin{minipage}{6cm}
\mbox{\tiny$
  \left[
    \mbox{\tiny $
    \begin{array}{rrrrrrrrr}
    1 & 1 & 0 & -a & 0 & 0 & 0 & 0 & 0 \\
    -1 & 0 & -1 & 0 & 1 & 0 & 0 & 0 & 0 \\
    -1 & 0 & -2 & 0 & 0 & 1 & 0 & 0 & 0 \\
    -1 & 0 & 1 & 1 & 0 & 0 & 1 & 0 & 0 \\
    0 & 0 & 1 & 1 & 0 & 0 & 0 & 1 & 0 \\
    -1 & 0 & 2 & 2 & 0 & 0 & 0 & 0 & 1
    \end{array}
    $}
    \right]  
$}
\end{minipage}
\end{center}
In this case, $X$ is the $2$-fold blow up of $\FF_a$, $a\in \ZZ_{\geq 3}$, in 
 $[0,1,0,1]$ of type $s_1$ with
 the intersection point $s_1$ of the 1st exceptional divisor
 and the transform of $V(T_3)\subseteq\FF_a$
 and the $2$-fold blow up of $[0,1,1,1]$ of type $g_1$.
\begin{center}
\begin{minipage}{1cm}
\tiny
  \begin{tikzpicture}[scale=.22]
    \fill[color=black!25] (0,0) rectangle (4,-2.5);
    
    \draw (0,0) -- (4,0);
    \draw (0,-2.5) -- (4,-2.5);
    \draw (4,0) -- (4,-2.5);
    \draw (0,0) -- (0,-2.5);
     
     \fill (0,0) circle (.2cm) node[anchor=east]{$2$};     
     \fill (0,-1.25) circle (.2cm) node[anchor=east]{$2$};
   \end{tikzpicture}
\end{minipage}
\end{center}

\item
$\dagger\dagger$
The $\ZZ^6$-graded ring $\KT{11}/I$, where 
generators for $I$ and the degree matrix are given as
\begin{center}
\begin{minipage}{8cm}
\tiny
\begin{gather*}
\begin{array}{ll}
 T_7T_8 - T_2^aT_4T_6^{a-1}T_{11}^a + T_3T_5,&
  T_9T_{11} - T_{1}^{a}T_{4}T_{5}^{a}T_{8}^{a-1} - \kappa T_{6}T_{7},\\
    T_{10}T_{11} - T_{1}^{a}T_{4}T_{5}^{a-1}T_{8}^{a} + \kappa T_{3}T_{6},&
  -\kappa T_{2}^{a}T_{4}T_{6}^{a}T_{11}^{a-1} + T_{8}T_{9} - T_{5}T_{10},\\
  \multicolumn{2}{c}{
  T_{1}^{a}T_{2}^{a}T_{4}^{2}T_{5}^{a-1}T_{6}^{a-1}T_{8}^{a-1}T_{11}^{a-1}-T_{3}T_{9}-T_{7}T_{10}}
\end{array}
\end{gather*}
\end{minipage}
\\[1ex]
\begin{minipage}{9cm}
\mbox{\tiny$
  \left[
    \mbox{\tiny $
           \begin{array}{rrrrrrrrrrr}
    1 & 0 & 0 & 0 & 0 & 0 & 1 & -1 & 1 & 0 & 0 \\
    0 & 1 & 0 & 0 & 0 & 0 & 0 & 0 & 1 & 1 & -1 \\
    0 & 0 & 1 & 0 & 0 & a-1 & 0 & 1 & 2a-3 & 2a-2 & -a+2 \\
    0 & 0 & 0 & 1 & 0 & 1 & 0 & 0 & 2 & 2 & -1 \\
    0 & 0 & 0 & 0 & 1 & a-1 & 1 & 0 & 2a-2 & 2a-3 & -a+2 \\
    0 & 0 & 0 & 0 & 0 & a & -1 & 1 & 2a-2 & 2a-1 & -a+1
    \end{array}
    $}
    \right]  
$}
\end{minipage}
\end{center}
In this case, $X$ is the blow up of $\FF_a$, $a\in \ZZ_{\geq 3}$, in 
$[0,1,0,1]$, $[1,0,0,1]$, $[0,1,1,1]$, $[1,0,1,\kappa]$ 
where $\kappa\in \KK^*$.
\begin{center}
\begin{minipage}{.7cm}
\tiny
  \begin{tikzpicture}[scale=.22]
    \fill[color=black!25] (0,0) rectangle (4,-2.5);
    
    \draw (0,0) -- (4,0);
    \draw (0,-2.5) -- (4,-2.5);
    \draw (4,0) -- (4,-2.5);
    \draw (0,0) -- (0,-2.5);
     
     \fill (0,0) circle (.2cm); 
     \fill (4,0) circle (.2cm); 
     \fill (4,-1.25) circle (.2cm) node[anchor=west]{$\kappa$};
     \fill (0,-1.25) circle (.2cm);      
   \end{tikzpicture}
\end{minipage}
\end{center}
\end{enumerate}
Of the surfaces marked with a single $\dagger$ we do 
not know whether they admit a $\KK^*$-action.
The surface marked with $\dagger\dagger$ does not admit a non-trivial
$\KK^*$-action and the listed
ring is the Cox ring for $a\leq 15$, whereas for $a > 15$,
the surface is a Mori dream surface having
the $H_2$-equivariant normalization of 
$\KT{11}/I:(T_1\cdots T_{11})^\infty$
as its Cox ring.
\end{theorem}

\begin{remark}
In Theorem~\ref{thm:pic6}, 
surface (vi) is a smooth del Pezzo surface of degree~4.
The surfaces $(iii)$, $(v)$, $(vii)$, $(viii)$
are weak del Pezzo surfaces of degree $4$ of singularity type
$A_3$, $A_1$, $2A_1$, $A_2$ respectively.
The number of generators and relations of their Cox rings
was given in~\cite[Table~6.2]{Der2}.
All other surfaces listed in the table of Theorem~\ref{thm:pic6}
contain $(-k)$-curves with $k\geq 3$
and therefore are not weak del Pezzo surfaces.
\end{remark}

\begin{remark}
The Cox ring generators occurring in Theorem~\ref{thm:pic6}, 
are either contractible curves on $X$ or they define curves in $\PP_2$ or,
in the case of a blown up $\FF_a$, in $\PP(1,1,a)$,
where their degrees are given as follows (the contractible ones 
are indicated by ``$-$''): 
\begingroup
\footnotesize
\begin{longtable}{ll}
\hline
surface & generator degrees in $\PP_2$ or $\PP(1,1,a)$
\\
\hline
$(i)$
&
$1, 1, 1, -, -, 2, -, -, 4, -$
\\
$(ii)$
&
$1, 1, 1, -, -, -, 1, -, 1, -$
\\
$(iii)$
&
$1, 1, 1, -, -, -, 1, -, 2, 2, -$
\\
$(iv)$
&
$1, 1, 1, -, -, -, 1, -, 2, -$
\\
$(v)$
&
$1, 1, 1, -, -, -, 1, 1, 1, -, 1, 1, -$
\\
$(vi)$
&
$1, 1, 1, -, -, -, 1, 1, 1, -, 1, 2, 1, 1, 1, -$
\\
$(vii)$
&
$1, 1, 1, -, -, -, 1, -, 1, 1, -$
\\
$(viii)$
&
$1, 1, 1, -, -, -, 1, -, 1, -$
\\
$(ix)$
&
$1, 1, a, -, -, -, 1, -, a, -$
\\
$(x)$
&
$1, 1, a, -, -, -, a, -, a, -$
\\
$(xi)$
&
$1,   1,   a,   -,   -,   -,   a,   -, 2a, -$
\\
$(xii)$
&
$1, 1, a, -, -, -, -, 2a, -$
\\
$(xiii)$
&
$1, 1, a, -, -, -, -, a, -$
\\
$(xiv)$
&
$1, 1, a, -, -, -, a, -, a, a, -$
\\
\hline
\end{longtable}
\endgroup 
\end{remark}

\begin{lemma}
\label{lem:isblowup}
Consider Setting~\ref{set:ambmod}.
Assume that $X_1 \subseteq Z_1$ 
is a CEMDS, $Z_2 \to Z_1$ 
arises from a barycentric subdivision
of a regular cone $\sigma \in \Sigma_1$
and $X_2 \to X_1$ has as center a point
$x \in X_1 \cap \TT^n \cdot z_\sigma$.
Let $f$ be the product over all $T_i$,
where $P_1(e_i) \not\in \sigma$,
and choose $z\in \KK^{r_1}$ with $p_1(z) = x$.
Then $X_2 \to X_1$ is the blow up 
at $x$ provided we have 
$$
 \bangle{
 T_i;\,z_i = 0
 }_f
 +
 I(P_1,z)_f
 \ =\ 
 \bangle{
 T_i;\,e_i\in \widehat\sigma
 }_f
 + 
 I(
 \overline{X_1}
 )_f
 \ \subseteq \ 
 \KK[T_1,\ldots,T_{r_1}]_{f}.
$$
\end{lemma}

\begin{proof}
Let $Z_{1,\sigma} \subseteq Z_1$ 
be the affine chart given by $\sigma$ and set
$X_{1,\sigma} := X_1 \cap Z_{1,\sigma}$. 
In order to see that the toric blow up $Z_2 \to Z_1$ 
induces a blow up $X_2 \to X_1$, we have to 
show 
$$ 
\mathfrak{m}_x
\ = \ 
I(\TT^n \cdot z_\sigma) + I(X_{1,\sigma})
\ \subseteq \ 
\Gamma(Z_{1,\sigma},\mathcal{O}).
$$
Consider the 
quotient map $p_1 \colon \rq{Z}_1 \to Z_1$.
Then we have $p_1^{-1}(Z_{1,\sigma}) = \KK^{r_1}_f$
and, since $\sigma$ is regular, 
$\Gamma(p_1^{-1}(Z_{1,\sigma}),\mathcal{O})$ admits 
units in every $K_1$-degree.
This implies 
\begin{align*}
&\qquad\qquad\qquad p_1^*(\mathfrak{m}_x)
\ = \ 
 \bangle{
 T_i;\,z_i = 0
 }_f
 +
I(P_1,z)_f,\\
&p_1^*(I(\TT^n \cdot z_\sigma))
\ = \ 
\bangle{T_i;\,e_i\in \widehat\sigma}_f,
\qquad\qquad
p_1^*(I(X_{1,\sigma}))
\ = \ 
I( \overline{X_1} )_f.
\end{align*}
Consequently, the assumption together with injectivity 
of the pullback map $p_1^*$ give the assertion.
\end{proof}

\begin{proof}[Proof of Theorem~\ref{thm:pic6}]
The idea is a stepwise classification of the Cox rings
for all smooth rational surfaces of Picard number
$\varrho(X) = 1,\ldots,6$.
For $\varrho(X) \leq 5$, it is possible to list all
occurring Cox rings up to isomorphism.
See~\cite{Ke} for the full list.
We treat exemplarily the surface $(x)$ in the table of the theorem,
which is obtained as blow up of a $\KK^*$-surface
$X_1$ with $\varrho(X_1) = 5$.
The Cox ring and degree matrix $Q_1$ of $X_1$ are
\[
\KK[T_1,\ldots,T_{8}]/\<
T_{2}^{a}T_{4}-T_{3}T_{5}T_{6}^{2}-T_{7}T_{8}
\>
,\quad
\mbox{\tiny $
\left[
\begin{array}{rrrrrrrr}
    1 & 1 & 0 & -a & 0 & 0 & 0 & 0 \\
    -1 & 0 & -1 & 0 & 1 & 0 & 0 & 0 \\
    -1 & 0 & -2 & 0 & 0 & 1 & 0 & 0 \\
    0 & 0 & 1 & 1 & 0 & 0 & 1 & 0 \\
    -1 & 0 & 1 & 1 & 0 & 0 & 0 & 1
    \end{array}
 \right]
 $}
 .
\]

The surface $X_2$ is the blow up of $X_1$ in the point 
with Cox coordinates $q := (1,1,1,1,0,1,1,1)\in \b X_1$.
To compute the Cox ring of $X_2$, 
we formally apply the steps of Algorithm~\ref{algo:latticeideal}.
In Setting~\ref{set:stretchcompress},
we choose the embedding
\[
 \b\iota\colon \KK^8\,\to\, \KK^{9}
 ,\qquad
 z\, \mapsto\, (z,f_1(z))
 ,\qquad
 f_1\ :=\ T_1T_2^{a-1}T_4T_8 - T_3T_6.
\]
The new degree matrix $Q_1'$ and the
matrix $P_1'$, whose columns are 
generators for the rays of the fan $\Sigma_1'$
of the ambient toric variety $Z_1'$, are
 \[
 Q_1' = 
 \left[
 Q_1
  \mbox{
  \tiny
  $
  \begin{array}{|r}
   0\\
   -1\\
   -1\\
   1\\
   1
  \end{array}
  $
  }
  \right]
  ,\quad
  P_1' = 
   \left[
  \mbox{
  \tiny
  $
  \begin{array}{rrrrrrrrr}
    1 & a-1 & 0 & 1 & 0 & 0 & 0 & 1 & -1 \\
    0 & a & 0 & 1 & 0 & 0 & -1 & -1 & 0 \\
    0 & 0 & 1 & 0 & 0 & 1 & 0 & 0 & -1 \\
    0 & 0 & 0 & 0 & 1 & 1 & -1 & -1 & 1
    \end{array}
  $
  }
  \right]
  .
 \]
To blow up the point on $X_1'$
 with Cox coordinates 
 $\b\iota(q) = (1,1,1,1,0,1,1,1,0)\in \b X_1'$,
 we perform the toric modification $\pi\colon Z_2\to Z_1'$
 given by the stellar subdivision $\Sigma_2 \to \Sigma_1'$
 at $v := (-1,0,-1,2)\in \ZZ^4$.
 Note that this completes the first threes steps of
 Algorithm~\ref{algo:latticeideal}.
 For the fourth one,
 we now formally apply Algorithm~\ref{algo:modifycemds}.
 Let $P_2 := [P_1',v]$.
 The ideal $I_2\subseteq \KT{10}$ 
 of $\b X_2$ is generated by
\begin{gather*}
  g_1
  \ :=\ 
  p_2^\sharp\,(p_1)_\sharp\,(T_{2}^{a}T_{4}-T_{3}T_{5}T_{6}^{2}-T_{7}T_{8})
  \,=\ 
 T_{2}^{a}T_{4}-T_{3}T_{5}T_{6}^{2}T_{10}-T_{7}T_{8}
 ,\\
 g_2
  \ :=\ 
   p_2^\sharp\,(p_1)_\sharp\,(T_{9} - f_1)
  \,=\ 
    T_{1}T_{2}^{a-1}T_{4}T_{8}-T_{3}T_{6}-T_{9}T_{10}.
\end{gather*}
We show that $I_2$ is prime; this implies in particular
that $I_2$ is saturated with respect to $T_{10}$.
On the open subset
\[
U \ :=\  
 \left\{
 x\in \overline X_2;\ 
 x_8x_{9}\not=0
 \ \,\text{or}\ \,
 x_7x_{10}\not=0 
 \right\}
 \ 
 \subseteq 
 \ 
 \b X_2
\]
the Jacobian $(\partial g_i/\partial T_j)_{i,j}$
is of rank two
and $U$ is a subset of the union of
the $8$-dimensional subspaces 
\[
 V\left(\KK^{10};\,T_{8},\,T_{7}\right),
    \quad
 V\left(\KK^{10};\,T_{8},\,T_{10}\right),
    \quad
 V\left(\KK^{10};\,T_{9},\,T_{7}\right),
    \quad
 V\left(\KK^{10};\,T_{9},\,T_{10}\right),
\]
all of which have six-dimensional intersection with
$\b X_2$.
Hence, $\overline X_2 \setminus U$ is of codimension
at least two in $\b X_2$.
Furthermore, since the effective cone of $X_2$ is pointed,
$\b X_2$ is connected.
By Serre's criterion
the ideal $I_2$ is prime, see e.g.~\cite{Kra}.

We claim that in $R_2 = \KT{10}/I_2$, 
the variable $T_{10}$ defines a prime element.
Instead of showing that 
$I_2 + \< T_{10}\>\subseteq\KT{10}$
is prime, removing non-used variables, 
we may show this for
$$
I_0
\,:=\,
\<
 T_{2}^{a}T_{4} - T_{6}T_{7},\ 
 T_{1}T_{2}^{a-1}T_{4}T_{7} - T_{3}T_{5}
\>\ \subseteq\ \KT{7}.
$$
Considered as an ideal in $\laurant{7}$, 
it is prime since the matrix with the exponents of the binomial
generators as its rows
\begin{align*}
     \left[
    \mbox{\tiny $
    \begin{array}{rrrrrrr}
    0 & a & 0 & 1 & 0 & -1 & -1 \\
    1 & a-1 & -1 & 1 & -1 & 0 & 1
    \end{array}
    $}
    \right]
\end{align*}
has Smith normal form $[E_2,0]$, 
where $E_2$ is the $2\times 2$ unit matrix.
Now, by Lemma~\ref{lem:satprops}, $I_0$ is prime if 
$I_0 = I_0 : (T_1\cdots T_7)^\infty$.
To this end, the set
$$
 \mathcal{G}
\ :=\ 
\left\{
T_{1}T_{6}T_{7}^{2}-T_{2}T_{3}T_{5},\ 
T_{2}^{a}T_{4}-T_{6}T_{7},\ 
T_{1}T_{2}^{a-1}T_{4}T_{7}-T_{3}T_{5}
\right\}    
\ \subseteq\ I_0
$$
turns out to be a Gr\"obner basis for $I_0$
with respect to the degree reverse lexicographical 
ordering for any ordering of the variables.
By~\cite[Lem.~12.1]{Stu}, we know that
\[
 \left\{
 \frac{f}{T_{i}^{k_i(f)}};\ 
 f\in \mathcal G
 \right\}
 \ =\ 
 {\mathcal G}
 ,\qquad
 k_i(f)
 \ :=\ 
 \max\bigl(
 n\in \ZZZ;\,T_{i}^n\mid f
 \bigr)
\]
is a Gr\"obner basis for $I_0 : T_{i}^\infty$ for any $i$.
As in~\cite[p.~114]{Stu}, the claim follows from
\[
 I_0 : (T_1\cdots T_7)^\infty
 \ =\ 
 \left(\left(\cdots (I_0 : T_1^\infty)\cdots\right) : T_7^\infty \right)
 \ =\ 
 I_0.
\]

Moreover,
no two variables $T_i$, $T_j$
are associated since $\deg T_i\ne \deg T_j$
for all $i$,~$j$
and $T_{10}\nmid T_{i}$ for all $i<10$ because 
$\dim \,\b X_2 \cap V(T_{i},\,T_{10})$ is at most six 
for each $i$. For instance,
$$
    \b X_2 \cap V(T_{5},\,T_{10})
    \ =\  
    V(
    T_{10},\,
    T_{5},\,
    T_{2}^{a}T_{4}-T_{7}T_{8},\,
    T_{1}T_{2}^{a-1}T_{4}T_{8}-T_{3}T_{6}
    )\ \subseteq\ \KK^{10}
$$
is of dimension six on $\TT^{10}\cdot (1,1,1,1,0,1,1,1,1,0)$
since its dimension equals $8 - \rank B$ with 
\[
B \ :=\  
 \left[
    \mbox{\tiny $
    \begin{array}{rrrrrrrrrr}
    0 & a & 0 & 1 & 0 & 0 & -1 & -1 & 0 & 0 \\
    1 & a-1 & -1 & 1 & 0 & -1 & 0 & 1 & 0 & 0
    \end{array}
    $}
    \right]
\]
having the exponents of the binomial equations as its rows.
Similarly, on the smaller tori, the dimension is at most six. 
By Theorem~\ref{thm:ambientblow} and Algorithm~\ref{algo:modifycemds},
$R_2$ is the Cox ring of the performed modification.
We claim that we performed the desired blow up.
The ideal
$$
 I'
 \ :=\  
 \bangle{
 T_5,\ T_9,\  f_1,\ T_2^aT_4 -T_7T_8
 }
 \  \subseteq\ 
 \KT{9}
$$
is prime since $I_0$ is prime,
$\b\iota(q)\in V(I')\subseteq \KK^{9}$
and, as seen above,
\[
\dim\, V\left(\KK^9;\,I'\right)
\ =\ 
-1\,+\,\dim\, \b X_2 \cap V\left(\KK^{10};\,T_5,\,T_{10}\right)
\ =\ 
5.
\]
By~\cite[Thm.~7.4]{MiSt}, as $K_1'$ is free,
$I(P_1',\b\iota(q))$ is prime.
This implies $I' = I(P_1',\b\iota(q))$.
Lemma~\ref{lem:isblowup} applies.
The Cox ring and degree matrix of $X_2$ 
are listed in the table under $(x)$.
Note that $X_2$ is not a $\KK^*$-surface:
by the blow up sequence, its
graph of negative curves contains the subgraph
 \begin{center}
  \includegraphics{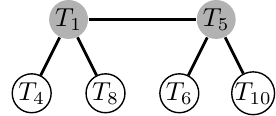}
 \end{center}
 where, by the theory of $\KK^*$-surfaces~\cite{OrWa}, 
 the curves corresponding to $T_1$ and $T_5$
  must correspond to the sink and source of the 
 $\KK^*$-action. 
 On a $\KK^*$-surface, sink and source must not meet. 
\end{proof}

\section{Linear generation}
\label{section:lineargen}
We consider the blow up $X$ of a projective 
space $\PP_n$ at $k$ distinct points 
$x_1, \ldots,x_k$, where $k > n+1 $.
Our focus is on special configurations in 
the sense that the Cox ring of $X$ is 
generated by the exceptional divisors and 
the proper transforms of hyperplanes.
We assume that $x_1, \ldots, x_{n+1}$ are 
the standard toric fixed points, i.e.~we have
$$ 
x_1 \ = \ [1,0,\ldots,0], 
\quad
\ldots,
\quad 
x_{n+1} \ = \ [0, \ldots, 0,1].
$$
Now, write $\mathcal{P} := \{x_1,\ldots, x_k\}$
and let $\mathcal{L}$ denote the set of all 
hyperplanes $\ell \subseteq \PP_n$ containing 
$n$ (or more) points of $\mathcal{P}$. 
For every $\ell \in \mathcal{L}$, we fix a linear 
form $f_{\ell} \in \KK[T_1, \ldots, T_{n+1}]$ with 
$\ell = V(f_\ell)$.
Note that the $f_\ell$ are homogeneous elements 
of degree one in the Cox ring of $\PP_n$.

The idea is now to take all $T_\ell$, where 
$\ell \in \mathcal{L}$, as prospective generators 
of the Cox ring of the blow up $X$ and then 
to compute the Cox ring using 
Algorithms~\ref{algo:stretchcemds}, 
\ref{algo:modifycemds} and~\ref{algo:compresscemds}.
Here comes the algorithmic formulation.

\begin{algorithm}[LinearBlowUp]
\label{algo:lineargen}
{\em Input: } a collection  
$x_1,\ldots, x_k \in \PP_n$ 
of pairwise distinct points.
\begin{itemize}
\item 
Set $X_1 := \PP_n$, let $\Sigma_1$ be the fan 
of $\PP_n$ and $P_1$ the matrix with columns 
$e_0, \ldots, e_n$, where $e_0 = - (e_1+ \ldots +e_n)$.
\item
Compute the set $\mathcal{L}$ of all 
hyperplanes through any $n$ of the 
points $x_1, \ldots, x_k$, let 
$(f_\ell; \; \ell \in \mathcal{L}')$ 
be the collection of the $f_\ell$ different 
from all $T_i$.
\item
Compute the stretched CEMDS $(P_1',\Sigma_1',G_1')$ 
by applying Algorithm~\ref{algo:stretchcemds} to 
$(P_1,\Sigma_1,G_1)$ and  $(f_\ell; \; \ell \in \mathcal{L}')$.
\item
Determine the Cox coordinates $z_i' \in \KK^{r_1'}$ 
of the points $x_i' \in X_1'$ corresponding 
to $x_i \in X_1$.
\item
Let $\Sigma_2$ be the barycentric subdivision 
of $\Sigma_1'$ at the cones $\sigma_i'$,
corresponding to the toric orbits containing
$x_i' = p_1'(z_i')$.
Write primitive generators for the rays of $\Sigma_2$
into a matrix $P_2 = [P_1,B]$.
\item
Compute $(P_2,\Sigma_2,G_2)$ 
by applying Algorithm~\ref{algo:modifycemds}
to $(P_1',\Sigma_1',G_1')$ and the pair $(P_2,\Sigma_2)$ 
\item
Set $(P_2',\Sigma_2',G_2') := (P_2,\Sigma_2,G_2)$.
Eliminate all fake relations by 
applying Algorithm~\ref{algo:compresscemds}
with option \texttt{verify}.
Call the output $(P_2,\Sigma_2,G_2)$. 
\end{itemize}
{\em Output: } $(P_2,\Sigma_2,G_2)$.
If the verifications in the last step
were positive, 
this is a CEMDS describing the blow up of
$\PP_n$ at the points $x_1, \ldots, x_k$; 
in particular the $K_2$-graded algebra $R_2$ 
is the Cox ring of~$X_2$.
\end{algorithm}

\begin{lemma}
\label{lem:lingenisblowup}
In Algorithm~\ref{algo:lineargen}
for each $x_i$
the barycentric subdivision of $\sigma_i'$
induces a blow up of $X_1'$ in $x_i'$.
\end{lemma}

\begin{proof}
In Algorithm~\ref{algo:lineargen},
let $G_1'= \{T_{n+1+j}-f_j;\,1\leq j\leq s\}$.
We have
\begin{align*}
\overline{X_1'}
\cap
V
\left(
T_j;\,e_j\in\widehat\sigma_i'
\right)
\ \supseteq\
V\left(
I(P_1',z_i')
\right)
=
\overline{H_1'\cdot z_i'}
\end{align*}
since the left hand side
is $H_1'$-invariant.
Equality is achieved by comparing
 dimensions and the fact that $x_i$
is cut out by $n$ hyperplanes.
Taking ideals, this implies that
$I(\overline{X_1'})
+
\bangle{
T_j;\,e_j\in\widehat{\sigma}_i'
}
$
equals
$I(P_1',z_i')$ because the ideals
are linear and thus radical.
Since $\sigma_i'$
is smooth, we may use Lemma~\ref{lem:isblowup}.
\end{proof}

\begin{proof}[Proof of Algorithm~\ref{algo:lineargen}]
By Lemma~\ref{lem:lingenisblowup}, $X_2\to X_1'$ is the blow up 
at~$x_1',\ldots,x_k'$.
It remains to show
that the input ring $R_2'$ of the last step is normal; 
this is necessary for Algorithm~\ref{algo:compresscemds}.
We only treat the case $k=1$.
Consider the stretched ring $R_1'$
obtained from the third step
and the ring $R_2$ obtained after the sixth step
\begin{align*}
 R_1'
 \ =\ 
 \KK[T_1,\ldots,T_{r_1'}]/\bangle{G_1'},
\qquad 
 R_2
 \ =\ 
 \KK[T_1,\ldots,T_{r_1'},T_{r_2}]/\bangle{G_2},
\end{align*}
where $T_{r_2}$ corresponds to the 
exceptional divisor.
We assume that of the $r_1'-r_1$ new 
equations $T_i - f_i$ in $\bangle{G_1'}$
the last $l$ will result in 
fake relations in $\bangle{G_2}$.
Localizing and passing to degree zero,
we are in the situation
\[
 \xymatrix{
 (R_1)_{T_1\cdots T_{r_1}}
 \ar[r]
 &
 (R_1)_{T_1\cdots T_{r_1}f_1\cdots f_{r_1'-l}}
 &
 (R_2)_{T_1\cdots T_{r_1'-l}T_{r_2}} 
 \\
  \left(
  (R_1)_{T_1\cdots T_{r_1}}
  \right)_0
 \ar[r]
 &
 \ar[u]
 \ar@{=}[r]
 \left(
 (R_1)_{T_1\cdots T_{r_1}f_1\cdots f_{r_1'-l}}
 \right)_0
 &
 \ar[u]
  \left(
 (R_2)_{T_1\cdots T_{r_1'-l}T_{r_2}} 
 \right)_0
 }
\]
The upper left ring is $K_1$-factorial by assumption.
By~\cite[Thm.~1.1]{Be}
the middle ring in the lower row
is a UFD and the ring on the 
upper right is $K_2$-factorial.
Thus, $R_2$ is $K_2$-factorial.
Since $K_1$ is free, also $K_2$ is,
so $R_2$ is a UFD. In particular, $R_2$ is normal and we may
apply Algorithm~\ref{algo:compresscemds}.
\end{proof}

Before eliminating fake relations, 
the ideal of the intersection of $\b{X}_2$ 
with the ambient big torus $\TT^{r_2}$ 
admits the following description in terms of 
incidences.

\begin{remark}
At the end of the fifth step in Algorithm~\ref{algo:lineargen},
the Cox ring of $Z_2$ is the polynomial 
ring $\KK[T_\ell,S_p]$ with indices
$\ell \in \mathcal{L}$ and $p \in \mathcal{P}$.
Consider the homomorphism
$$
\beta \colon 
\KK[T_\ell;\,\ell\in\mathcal{L}]
\ \to \
\KK[T_\ell,S_p;\,\ell\in\mathcal{L},\,p\in\mathcal{P}],
\qquad
T_\ell 
\ \mapsto \ 
T_\ell \cdot \prod_{p\in \ell} S_p.
$$
Then the extension of the ideal $I_2 \subseteq \KK[T_\ell,S_p]$ 
to the Laurent polynomial ring 
$\KK[T_\ell^\pm,S_p^\pm]$ is generated by 
$\beta(T_\ell - f_\ell)$, where $\ell \in \mathcal{L}'$.
\end{remark}

\begin{remark}\label{rem:linearspace} 
Some geometric properties of the blow 
up $X$ may be seen directly from the combinatorics
of the underlying \textit{finite linear space}
$\mathfrak{L}=(\mathcal{P},\mathcal{L},\in)$,
compare e.g.~\cite{finitelinearspaces}.
For instance, if $\mathfrak{L}$ is
an \textit{$n$-design}, i.e.~each line 
$\ell\in \mathcal{L}$ contains
exactly $n$ points, the ideal
of the Cox ring of $X$ is classically homogeneous.
Furthermore, for a surface $X$,
if all points but one lie on a common line,
i.e.~$\mathfrak{L}$ is a \textit{near-pencil},
then $X$ comes with a $\KK^*$-action.
It would be interesting to see further
relations between $X$ and~$\mathfrak{L}$.
\end{remark}

\begin{example}
\label{ex:almostfanoplane}
Let $X$ be the blow up of $\PP_2$
in the seven points 
\begin{center}
 \begin{minipage}{3cm}
  \begin{align*}
  &x_1 := [1,0,0],\quad 
  x_2 := [0,1,0],\quad
  x_3 := [0,0,1],\\
  &x_4 := [1,1,0],\quad 
  x_5 := [1,0,-1],\quad
  x_6 := [0,1,1],\\ 
  &x_7 := [1,1,1].
 \end{align*}
 \end{minipage}
 \ 
\begin{minipage}{3cm}
 \includegraphics{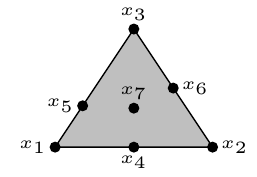}
\end{minipage}
\end{center}
Write $S_i$ for the variables 
corresponding to $x_i$ and
let $T_1,\ldots,T_9$ correspond
to the nine lines in $\mathcal{L}$.
Algorithm~\ref{algo:lineargen}
provides us with the Cox ring of $X$.
It is given as the factor ring
$\KK[T_1,\ldots,T_9,S_1,\ldots,S_{7}]/I$
where $I$ is generated by
{\tiny
\begin{alignat*}{2}
&{2}T_{8}S_4S_6-T_{5}S_2+T_{9}S_7,
&{2}T_{1}S_3S_6+T_{5}S_5-T_{6}S_7,\\
&{2}T_{4}S_1S_6+T_{6}S_2-T_{9}S_5,
&-T_{1}S_2S_6+T_{2}S_1S_5-T_{7}S_4S_7,\\
&{2}T_{7}S_3S_4+T_{6}S_2+T_{9}S_5,
&-T_{2}S_5S_3+T_{3}S_4S_2-T_{4}S_7S_6,\\
&{2}T_{3}S_1S_4+T_{5}S_5+T_{6}S_7,
&T_{1}S_2S_3+T_{8}S_4S_5+T_{4}S_1S_7,\\
&{2}T_{2}S_1S_3+T_{5}S_2+T_{9}S_7,
&T_{2}T_{6}S_3-T_{3}T_{9}S_4-T_{4}T_{5}S_6,\\
&T_{3}S_1S_2+T_{8}S_5S_6-T_{7}S_3S_7,
&T_{3}T_{9}S_1-T_{5}T_{7}S_3-T_{6}T_{8}S_6,\\
&T_{2}T_{6}S_1-T_{5}T_{7}S_4+T_{1}T_{9}S_6,
&T_{4}T_{5}S_1+T_{1}T_{9}S_3+T_{6}T_{8}S_4,\\
&T_{3}T_{7}S_4^{2}+T_{1}T_{4}S_6^{2}+T_{2}T_{6}S_5,
&T_{2}T_{7}S_3^{2}+T_{4}T_{8}S_6^{2}+T_{3}T_{9}S_2,\\
&T_{1}T_{2}S_3^{2}+T_{3}T_{8}S_4^{2}-T_{4}T_{5}S_7,
&T_{1}T_{3}S_2^{2}+T_{2}T_{8}S_5^{2}+T_{4}T_{7}S_7^{2},\\
&T_{3}T_{4}S_1^{2}+T_{1}T_{7}S_3^{2}-T_{6}T_{8}S_5,
&T_{2}T_{4}S_1^{2}+T_{7}T_{8}S_4^{2}-T_{1}T_{9}S_2,\\
&T_{2}T_{3}S_1^{2}+T_{1}T_{8}S_6^{2}+T_{5}T_{7}S_7,
&T_{4}T_{5}^{2}T_{7}+T_{2}T_{6}^{2}T_{8}+T_{1}T_{3}T_{9}^{2}
\end{alignat*}
}
and the $\ZZ^8$-grading is given by the degree matrix
{\tiny
\begin{align*}
\left[
 \begin{array}{rrrrrrrrrrrrrrrr}
0 & -1 & -1 & -1 & 0 & 0 & 0 & 0 & 0 & 1 & 0 & 0 & 0 & 0 & 0 & 0 \\
-1 & 0 & -1 & 0 & -1 & -1 & 0 & 0 & 0 & 0 & 1 & 0 & 0 & 0 & 0 & 0 \\ 
0 & 0 & -1 & 0 & 0 & 0 & 1 & 1 & 1 & 1 & 1 & 0 & 0 & 0 & 0 & 0 \\ 
-1 & -1 & 0 & 0 & 0 & 0 & -1 & 0 & 0 & 0 & 0 & 1 & 0 & 0 & 0 & 0 \\
0 & 0 & -1 & 0 & 0 & 0 & -1 & -1 & 0 & 0 & 0 & 0 & 1 & 0 & 0 & 0 \\ 
0 & 0 & 0 & 1 & -1 & 0 & 1 & 0 & 0 & 0 & 1 & 0 & 0 & 1 & 0 & 0 \\ 
-1 & 0 & 0 & -1 & 0 & 0 & 0 & -1 & 0 & 0 & 0 & 0 & 0 & 0 & 1 & 0 \\ 
0 & 1 & 0 & 0 & 0 & -1 & 0 & 1 & 0 & 0 & 1 & 0 & 0 & 0 & 0 & 1
 \end{array}
 \right].
 \end{align*}
}
\end{example}

The following theorem concerns blow ups of 
$\PP_3$ in six points $x_1,\ldots,x_6$.
As before, we assume that $x_1,\ldots,x_4$
are the standard toric fixed points.
We call the point configuration {\it edge-special\/} 
if at least one point of $\{x_5, x_6\}$ 
is contained in two different hyperplanes 
spanned by the other points.

\begin{theorem}
\label{thm:PP3}
Let $X$ be the blow up of $\mathbb{P}_3$
at distinct points $x_1,\dots,x_6$
not contained in a  hyperplane. 
Then $X$ is a Mori dream space.
Moreover, for the following typical 
edge-special configurations, we obtain:
 \begin{enumerate}
  \item 
  For $x_5 := [1,1,0,0]$, $x_6 := [0,1,1,1]$,
  the Cox ring of $X$ is $\KK[T_1,\ldots,T_{16}]/I$, 
  where $I$ is generated by
  
  \begin{center}
  \begin{minipage}{2cm}
  \includegraphics{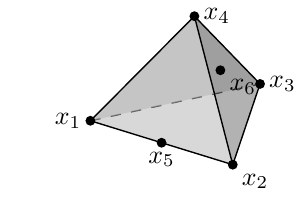}
  \end{minipage}
  \qquad\qquad
  \begin{minipage}{7cm}
   {\tiny
\begin{alignat*}{2}
&{2}T_{4}T_{13}-{2}T_{5}T_{16}-{2}T_{3}T_{14},
&T_{4}T_{12}T_{15}-T_{2}T_{14}-T_{6}T_{16},\\
&T_{5}T_{12}T_{15}-T_{6}T_{13}+T_{7}T_{14},
&T_{3}T_{12}T_{15}-T_{2}T_{13}-T_{7}T_{16},\\
&T_{5}T_{11}T_{12}-T_{9}T_{13}+T_{10}T_{14},
&T_{4}T_{11}T_{12}-T_{8}T_{14}-T_{9}T_{16},\\
&T_{3}T_{11}T_{12}-T_{8}T_{13}-T_{10}T_{16},
&T_{1}T_{12}T_{13}+T_{7}T_{11}-T_{10}T_{15},\\
&T_{1}T_{12}T_{14}+T_{6}T_{11}-T_{9}T_{15},
&T_{1}T_{12}T_{16}-T_{2}T_{11}+T_{8}T_{15},\\
&T_{5}T_{8}-T_{3}T_{9}+T_{4}T_{10},
&T_{2}T_{5}-T_{3}T_{6}+T_{4}T_{7},\\
&T_{1}T_{5}T_{12}^{2}+T_{7}T_{9}-T_{6}T_{10},
&T_{1}T_{3}T_{12}^{2}+T_{7}T_{8}-T_{2}T_{10},\\
&T_{1}T_{4}T_{12}^{2}+T_{6}T_{8}-T_{2}T_{9}
\end{alignat*}
}
  \end{minipage}
  \end{center} 
  with the $\ZZ^7$-grading given by the degree matrix 
  {\tiny
  \begin{align*}
   \left[
\begin{array}{rrrrrrrrrrrrrrrr}
1 & 1 & 1 & 1 & 1 & 1 & 1 & 1 & 1 & 1 & 0 & 0 & 0 & 0 & 0 & 0 \\ 
0 & -1 & -1 & -1 & -1 & -1 & -1 & 0 & 0 & 0 & 1 & 0 & 0 & 0 & 0 & 0 \\ 
-1 & 0 & -1 & -1 & -1 & 0 & 0 & 0 & 0 & 0 & 0 & 1 & 0 & 0 & 0 & 0 \\ 
-1 & -1 & 0 & -1 & 0 & -1 & 0 & -1 & -1 & 0 & 0 & 0 & 1 & 0 & 0 & 0 \\ 
0 & 0 & 0 & 1 & 1 & 1 & 0 & 0 & 1 & 0 & 0 & 0 & 0 & 1 & 0 & 0 \\ 
1 & 1 & 0 & 0 & 0 & 1 & 1 & 0 & 0 & 0 & 0 & 0 & 0 & 0 & 1 & 0 \\ 
0 & 1 & 1 & 1 & 0 & 0 & 0 & 1 & 0 & 0 & 0 & 0 & 0 & 0 & 0 & 1
\end{array}
\right].
\end{align*}
}
  \item 
  For $x_5 := [2,1,0,0]$, $x_6 := [1,1,0,1]$,
  the Cox ring of $X$ is $\KK[T_1,\ldots,T_{15}]/I$, 
  where $I$ is generated by
  
  \begin{center}
  \begin{minipage}{2cm}
    \includegraphics{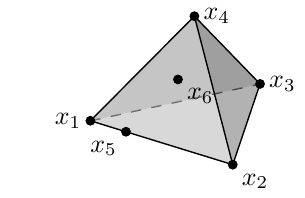}
  \end{minipage}
  \qquad\qquad
  \begin{minipage}{7cm}
   {\tiny
\begin{alignat*}{2}
&T_{1}T_{11}+T_{7}T_{14}+{2}T_{8}T_{15},
&T_{2}T_{10}+T_{7}T_{14}+T_{8}T_{15},\\
&T_{4}T_{11}T_{14}-T_{2}T_{13}-T_{5}T_{15},
&T_{4}T_{10}T_{14}-T_{1}T_{13}-T_{6}T_{15},\\
&T_{4}T_{10}T_{11}+T_{7}T_{13}-T_{9}T_{15},
&T_{6}T_{11}-{2}T_{8}T_{13}-T_{9}T_{14},\\
&T_{5}T_{10}-T_{8}T_{13}-T_{9}T_{14},
&{2}T_{4}T_{8}T_{10}+T_{6}T_{7}+T_{1}T_{9},\\
&T_{4}T_{8}T_{11}+T_{5}T_{7}+T_{2}T_{9},
&T_{4}T_{8}T_{14}+T_{1}T_{5}-T_{2}T_{6}
\end{alignat*}
}
\end{minipage}
\end{center} 
with the $\ZZ^7$-grading given by the degree matrix 
\[
\mbox{\tiny$
\left[
\begin{array}{rrrrrrrrrrrrrrr}
1 & 1 & 1 & 1 & 1 & 1 & 1 & 1 & 1 & 0 & 0 & 0 & 0 & 0 & 0 \\ 
0 & -1 & -1 & -1 & -1 & 0 & 0 & 0 & 0 & 1 & 0 & 0 & 0 & 0 & 0 \\ 
-1 & 0 & -1 & -1 & 0 & -1 & 0 & 0 & 0 & 0 & 1 & 0 & 0 & 0 & 0 \\ 
0 & 0 & 1 & 0 & 0 & 0 & 0 & 0 & 0 & 0 & 0 & 1 & 0 & 0 & 0 \\ 
-1 & -1 & -1 & 0 & 0 & 0 & -1 & -1 & 0 & 0 & 0 & 0 & 1 & 0 & 0 \\ 
1 & 1 & 0 & 0 & 1 & 1 & 0 & 1 & 0 & 0 & 0 & 0 & 0 & 1 & 0 \\ 
1 & 1 & 0 & 1 & 0 & 0 & 1 & 0 & 0 & 0 & 0 & 0 & 0 & 0 & 1
\end{array}
\right].
$}
\]
  \item  
  For $x_5 := [1,0,0,1]$, $x_6 := [0,1,0,1]$,
  the Cox ring of $X$ is $\KK[T_1,\ldots,T_{13}]/I$, 
  where $I$ is generated by
  \begin{center}
  \begin{minipage}{2cm}
    \includegraphics{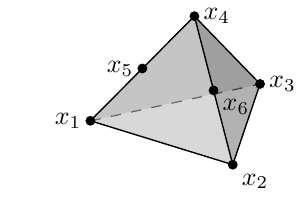}
  \end{minipage}
  \qquad\qquad
  \begin{minipage}{7cm}
   {\tiny
\begin{alignat*}{2}
&T_{2}T_{8}T_{11}-T_{6}T_{9}+T_{7}T_{13},
&T_{2}T_{11}T_{12}-T_{4}T_{9}+T_{5}T_{13},\\
&T_{1}T_{9}T_{11}-T_{5}T_{8}+T_{7}T_{12},
&T_{1}T_{11}T_{13}-T_{4}T_{8}+T_{6}T_{12},\\
&T_{1}T_{2}T_{11}^{2}-T_{5}T_{6}+T_{4}T_{7}
\end{alignat*}
}
\end{minipage}
\end{center} 
with the $\ZZ^7$-grading given by the degree matrix 
\[
\mbox{\tiny$
\left[
\begin{array}{rrrrrrrrrrrrr}
1 & 1 & 1 & 1 & 1 & 1 & 1 & 0 & 0 & 0 & 0 & 0 & 0 \\ 
0 & -1 & -1 & -1 & -1 & 0 & 0 & 1 & 0 & 0 & 0 & 0 & 0 \\ 
-1 & 0 & -1 & -1 & 0 & -1 & 0 & 0 & 1 & 0 & 0 & 0 & 0 \\ 
0 & 0 & 1 & 0 & 0 & 0 & 0 & 0 & 0 & 1 & 0 & 0 & 0 \\ 
-1 & -1 & -1 & 0 & 0 & 0 & 0 & 0 & 0 & 0 & 1 & 0 & 0 \\ 
1 & 0 & 0 & 1 & 1 & 0 & 0 & 0 & 0 & 0 & 0 & 1 & 0 \\ 
0 & 1 & 0 & 1 & 0 & 1 & 0 & 0 & 0 & 0 & 0 & 0 & 1
\end{array}
\right]
$}.
\]

  \item 
  For $x_5 := [1,0,0,1]$, $x_6 := [0,1,1,0]$,
  the Cox ring of $X$ is $\KK[T_1,\ldots,T_{12}]/I$, 
  where $I$ is generated by 
  \begin{center}
  \begin{minipage}{2cm}
 \includegraphics{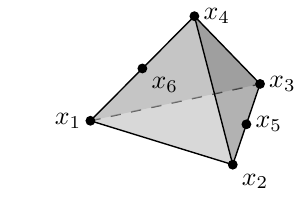}
  \end{minipage}
  \qquad
  \begin{minipage}{7cm}
   {\tiny
\begin{alignat*}{2}
&T_{3}T_{8}-T_{5}T_{12}-T_{2}T_{9},\\
&T_{4}T_{7}-T_{6}T_{11}-T_{1}T_{10}
\end{alignat*}
}
\end{minipage}
\end{center} 
with the $\ZZ^7$-grading given by the degree matrix 
\[
\mbox{\tiny$
\left[
\begin{array}{rrrrrrrrrrrr}
1 & 1 & 1 & 1 & 1 & 1 & 0 & 0 & 0 & 0 & 0 & 0 \\ 
0 & -1 & -1 & -1 & -1 & 0 & 1 & 0 & 0 & 0 & 0 & 0 \\ 
0 & 1 & 0 & 0 & 1 & 0 & 0 & 1 & 0 & 0 & 0 & 0 \\ 
0 & 0 & 1 & 0 & 1 & 0 & 0 & 0 & 1 & 0 & 0 & 0 \\ 
-1 & -1 & -1 & 0 & -1 & 0 & 0 & 0 & 0 & 1 & 0 & 0 \\ 
0 & -1 & -1 & 0 & -1 & -1 & 0 & 0 & 0 & 0 & 1 & 0 \\ 
0 & 1 & 1 & 0 & 0 & 0 & 0 & 0 & 0 & 0 & 0 & 1
\end{array}
\right]
$}.
\]
  \item 
  For $x_5 := [2,1,0,0]$, $x_6 := [1,2,0,0]$,
  the Cox ring of $X$ is $\KK[T_1,\ldots,T_{12}]/I$,
  where $I$ is generated by
  \begin{center}
  \begin{minipage}{2cm}
  \includegraphics{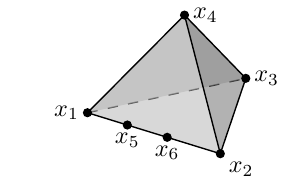}
  \end{minipage}
  \qquad
  \begin{minipage}{7cm}
   {\tiny
\begin{alignat*}{2}
&{3}T_{2}T_{7}+{2}T_{5}T_{11}+T_{6}T_{12},\\
&{3}T_{1}T_{8}+T_{5}T_{11}+{2}T_{6}T_{12}
\end{alignat*}
}
\end{minipage}
\end{center} 
with the $\ZZ^7$-grading given by the degree matrix 
\[
\mbox{\tiny$
\left[
\begin{array}{rrrrrrrrrrrr}
1 & 1 & 1 & 1 & 1 & 1 & 0 & 0 & 0 & 0 & 0 & 0 \\ 
0 & -1 & -1 & -1 & 0 & 0 & 1 & 0 & 0 & 0 & 0 & 0 \\ 
-1 & 0 & -1 & -1 & 0 & 0 & 0 & 1 & 0 & 0 & 0 & 0 \\ 
0 & 0 & 1 & 0 & 0 & 0 & 0 & 0 & 1 & 0 & 0 & 0 \\ 
0 & 0 & 0 & 1 & 0 & 0 & 0 & 0 & 0 & 1 & 0 & 0 \\ 
0 & 0 & -1 & -1 & -1 & 0 & 0 & 0 & 0 & 0 & 1 & 0 \\ 
1 & 1 & 0 & 0 & 1 & 0 & 0 & 0 & 0 & 0 & 0 & 1
\end{array}
\right]
$}.
\]
\end{enumerate}
\end{theorem}

Let $X$ be the blow up of $\mathbb{P}_3$
at six non-coplanar points $x_1,\dots,x_6$.
Denote by $H$ the total transform of 
a plane of $\mathbb{P}_3$ and by $E_i$
the exceptional divisor over $x_i$. 
Denote by $E_I := H-\sum_{i\in I}E_i$
and by $Q_i=2H-2E_i-\sum_{k\neq i}E_k$,
the last being the strict transform
of the quadric cone with vertex in $x_i$
and through the remaining five $x_j$.
We consider five possibilities for $X$
according to the collinear subsets
of $\{x_1,\dots,x_6\}$, modulo permutations
of the indices. The meaning of the divisor
$\Delta$ in the second column of the table
will be explained in the proof of 
Theorem~\ref{thm:PP3}.

\begin{center}
\tiny
\begin{longtable}{ll}
\hline
Collinear subsets & $\Delta$
\\
\hline
\\
& $\frac{1}{2}\left(Q_4+Q_5+Q_6+E_{123}\right)$
\\
\\
\hline
\\
$\{x_1,x_2,x_5\}$
&
$\frac{5}{8}\left(E_{1235}+E_{1256}+E_{134}+E_{246}+E_{456}\right)$
\\
\\
\hline
\\
$\{x_1,x_2,x_5\}$,
$\{x_1,x_3,x_6\}$
&
$\frac{2}{3} E_{12356}
+\frac{1}{2}(E_{1245}+E_{1346}+E_{234})$
\\
\\
\hline
\\
$\{x_1,x_2,x_5\}$,
$\{x_3,x_4,x_6\}$
&
$\frac{1}{3}(E_{1235}+E_{1346}+E_{2346})
+\frac{1}{2}(E_{1245}+E_{1256}+E_{3456})$
\\
\\
\hline
\\
$\{x_1,x_2,x_5,x_6\}$
&
$\frac{3}{4}(E_{12356}+E_{12456})
+\frac{1}{3}(E_{234}+E_{346})
+\frac{1}{4}E_{345}$
\\
\\
\hline
\end{longtable}
\end{center}

\begin{lemma}\label{base}
Let $X$ be a smooth projective variety,
let $D$ be an effective divisor of $X$
and let $C$ be an irreducible and reduced
curve of $X$ such that $C\cdot D<0$.
Then $C$ is contained in the stable
base locus of $|D|$.
\end{lemma}
\begin{proof}
The statement follows from the observation
that if $E$ is a prime divisor not containing
$C$ then $C\cdot E\geq 0$.
\end{proof}

\begin{proposition}\label{mori}
Let $X$ be the blow up of $\mathbb{P}_3$ at six
non-coplanar points $x_1,\dots,x_6$.
Denote by $e_I$ with $I\subset\{1,\dots,6\}$
the class of the strict transform of the line through
the points $\{x_i;\, i\in I\}$, if $|I|\geq 2$, or
the class of a line in the exceptional divisor
$E_i$ over $x_i$, if $I=\{i\}$.
Then the Mori cone of $X$ is
\[
 {\rm NE}(X)
 =
 \langle e_I;\, I\subset\{1,\dots,6\},
 |I| = 1 \text{ or the set $\{x_i;\, i\in I\}$ is contained in a line}\rangle.
\]
\end{proposition}
\begin{proof}
We consider five cases for $X$ according to
the collinear subsets of $\{x_1,\dots,x_6\}$.
In each case we denote by $\mathcal{E}$ be the cone
spanned by the classes of the $e_I$ defined
in the statement.
Recall that, via the intersection form between 
divisors and curves, the nef cone ${\rm Nef}(X)$ is
dual to the closure of the Mori cone ${\rm NE}(X)$
~\cite[Proposition 1.4.28]{Laz}.
Hence $\mathcal{E}\subset {\rm NE}(X)$
gives $\mathcal{E}^\vee
\supset {\rm NE}(X)^\vee= {\rm Nef}(X)$. 
Thus, it is enough to show that each
extremal ray of $\mathcal{E}^\vee$ is a nef class.
A direct calculation shows that the primitive 
generator $w$ of an extremal ray
is the strict transform of one of the following 
divisors (here, by ``points'' we mean a subset 
of $\{x_1,\dots,x_6\}$):
\begin{itemize}
\item a plane through one simple point,
\item a quadric through simple points,
\item a cubic with one double point and 
simple points,
\item a quartic with one triple point and
simple points.
\end{itemize}
In each case, $w$ is the class of a divisor
$D = \sum_{k\in S} E_{I_k}$, where 
$I_k\subset\{x_1,\dots,x_6\}$ 
is a collinear subset for any $k\in S$,
where $S$ is a finite set of indices.
Hence the base locus of $|D|$ is contained
in the union of the strict transforms
of the lines spanned by each 
subset $\{x_i;\, i \in I_k\}$.
Again, a direct calculation shows that
$D\cdot e_{I_k}\geq 0$ for any $k\in S$
and thus $D$ is nef by Lemma~\ref{base}.
\end{proof}

\begin{proof}[Proof of Theorem~\ref{thm:PP3}]
We prove the first statement.
The anticanonical divisor $-K_X$ of $X$ 
is big and movable since $-K_X\sim 2D$,
where $D=2H-E_1-\dots-E_6$ is the strict 
transform of a quadric through the six
points.
Hence, $X$ is Mori dream if and only
if it is log Fano by~\cite[Lemma 4.10]{MK}.
We will prove that $X$ is log Fano by showing 
that
\[
 -K_X\sim A+\Delta,
\]
where both $A$ and $\Delta$ are
$\mathbb{Q}$-divisors, $A$ is nef and big, 
$\Delta$ is effective, its support is simple 
normal crossing and $\lfloor\Delta\rfloor=0$.
In the above table we provide $\Delta$ 
in each case. The ampleness
of $A$ is a direct consequence of
our description of the Mori cone 
of $X$ given in Proposition~\ref{mori}.
The second part of the theorem
is an application of 
 Algorithm~\ref{algo:lineargen}.
\end{proof}


\bibliographystyle{abbrv}

\end{document}